\documentclass[12pt, a4paper]{amsart}

\usepackage{amsmath}

\usepackage{amssymb}
\usepackage{stmaryrd}
\usepackage{amsthm}
\usepackage[mathscr]{eucal}
\usepackage{amscd}
\usepackage[all]{xy}
\usepackage{url}

\newtheorem{Thm}{Theorem}[subsection]
\usepackage{enumerate}

\usepackage{comment}

\usepackage[dvipsnames, usenames]{color}
\usepackage[colorlinks]{hyperref}

\hypersetup{%
pdftitle={Preprint},
pdfauthor={},
pdfkeywords={},
bookmarksnumbered,
pdfstartview={FitH},
breaklinks=true,
bookmarksdepth=2
}%

\usepackage[all]{hypcap} 

\usepackage{breakurl}

\usepackage{color}

\newtheorem{Lem}[Thm]{Lemma}
\newtheorem{Prop}[Thm]{Proposition}
\newtheorem{Cor}[Thm]{Corollary}
\newtheorem{Conj}[Thm]{Conjecture}
\newtheorem{Quest}[Thm]{Question}
\newtheorem{Eg}[Thm]{Example}
\newtheorem{Rem}[Thm]{Remark}
\newtheorem{Def}[Thm]{Definition}

\newtheorem*{Def*}{Definition}
\newtheorem*{Thm*}{Theorem}

\newtheorem*{Assumption*}{Assumption}

\newtheorem*{Conj*}{Conjecture}
\newtheorem*{Notation*}{Notation}

\newcommand{\Z}{\mathbb{Z}}
\newcommand{\N}{\mathbb{N}}
\newcommand{\Q}{\mathbb{Q}}
\newcommand{\C}{\mathbb{C}}

\newcommand{\ie}{{\em i.e.}\ }
\newcommand{\cf}{{\em cf.}\ }

\newcommand{\wrt}{{w.r.t.}\ }

\renewcommand{\hat}[1]{\widehat{#1}}
\renewcommand{\tilde}[1]{\widetilde{#1}}

\newcommand{\opname}[1]{\operatorname{\mathsf{#1}}}

\newcommand{\Rep}{\opname{Rep}}
\newcommand{\per}{\opname{per}}

\newcommand{\pr}{\opname{pr}}

\newcommand{\inj}{\opname{inj}}
\newcommand{\add}{\opname{add}}

\newcommand{\op}{^{op}}

\newcommand{\ra}{\rightarrow}

\newcommand{\iso}{\stackrel{_\sim}{\rightarrow}}

\newcommand{\Gr}{\opname{Gr}}

\newcommand{\id}{\mathbf{1}}

\newcommand{\End}{\opname{End}}

\newcommand{\Hom}{\opname{Hom}}

\newcommand{\Ext}{\opname{Ext}}

\renewcommand{\deg}{\opname{deg}}

\newcommand{\Hf}{{\frac{1}{2}}}

\newcommand{\Rm}[1]{{\longmapsto}}
\newcommand{\Lm}[1]{{\longmapsfrom}}

\newcommand{\cA}{{\mathcal A}}

\newcommand{\cC}{{\mathcal C}}
\newcommand{\cD}{{\mathcal D}}
\newcommand{\cE}{{\mathcal E}}
\newcommand{\cF}{{\mathcal F}}

\newcommand{\cI}{{\mathcal I}}

\newcommand{\cL}{{\mathcal L}}
\newcommand{\cM}{{\mathcal M}}
\newcommand{\cN}{{\mathcal N}}

\newcommand{\cR}{{\mathcal R}}

\newcommand{\cT}{{\mathcal T}}
\newcommand{\cU}{{\mathcal U}}

\newcommand{\cW}{{\mathcal W}}

\newcommand{\bI}{{\mathbf I}}

\newcommand{\bL}{{\mathbf L}}

\newcommand{\bi}{{\mathbf i}}

\newcommand{\bm}{{\mathbf m}}

\newcommand{\sV}{{\mathbb V}}
\newcommand{\sW}{{\mathbb W}}

\newcommand{\uX}{{\underline{X}}}

\newcommand{\ua}{{\underline{a}}}

\newcommand{\uc}{{\underline{c}}}

\newcommand{\uk}{{\underline{k}}}

\newcommand{\tB}{{\widetilde{B}}}

\newcommand{\tQ}{{\widetilde{Q}}}

\newcommand{\tW}{{\widetilde{W}}}

\newcommand{\tb}{{\widetilde{b}}}

\newcommand{\hk}{{\widehat{k}}}

\newcommand{\tg}{{\tilde{g}}}

\newcommand{\simp}{\mbf{S}}
\newcommand{\stdMod}{\mbf{M}}
\newcommand{\can}{L}
\newcommand{\gen}{\mathbb{L}}

\newcommand{\qtChar}{{\chi_{q,t}}}

\newcommand{\redWSet}{{\mathcal{J}}}

\newcommand{\quotKGp}{{\cR_t}}
\newcommand{\extQuotKGp}{{\cR}}
\newcommand{\commQuotKGp}{{\cR_{t=1}}}
\newcommand{\YTorus}{\mathscr {Y}}

\newcommand{\monCat}{{\mathcal{U}}}

\newcommand{\qBase}{{\Z[q^{\pm\Hf}]}}

\newcommand{\projQuot}{{\cM}}
\newcommand{\affQuot}{{\projQuot_0}}
\newcommand{\lag}{{\cL}}
\newcommand{\grProjQuot}{{\projQuot^\bullet}}
\newcommand{\grAffQuot}{{\affQuot^\bullet}}

\newcommand{\grLag}{{\lag^\bullet}}

\newcommand{\clAlg}{{\cA}}
\newcommand{\qClAlg}{{\cA^q}}

\newcommand{\oh}{{\overline{h}}}

\newcommand{\diag}{{\delta}}

\newcommand{\RHS}{{\mathrm{RHS}}} 
\newcommand{\LHS}{{\mathrm{LHS}}}

\newcommand{\trunc}{{^{\leq 0}}}

\newcommand{\mfr}[1]{{\mathfrak{#1}}}
\newcommand{\mbf}[1]{{\mathbf{#1}}}

\newcommand{\set}[1]{\left\{#1\right\}}

\usepackage{tikz}
\usetikzlibrary{positioning,shapes,shadows,arrows,snakes}

\pgfdeclarelayer{edgelayer}
\pgfdeclarelayer{nodelayer}
\pgfsetlayers{edgelayer,nodelayer,main}

\tikzstyle{none}=[inner sep=0pt]
\tikzstyle{black box}=[draw=black, fill=black!25]
\tikzstyle{white box}=[draw=black, fill=white]
\tikzstyle{black circle}=[circle,draw=black!50, fill=black!25]
\tikzstyle{red circle}=[circle,draw=red!50, fill=red!25]
\tikzstyle{blue circle}=[circle,draw=blue!50, fill=blue!25]
\tikzstyle{green circle}=[circle,draw=green!50, fill=green!25]
\tikzstyle{yellow circle}=[circle,draw=yellow!50, fill=yellow!25]

\newcommand{\thistheoremname}{}
\newtheorem*{genericthm*}{\thistheoremname}
\newenvironment{namedthm*}[1]
  {\renewcommand{\thistheoremname}{#1}%
   \begin{genericthm*}}
  {\end{genericthm*}}

\begin{document}
\title[]{triangular bases in quantum cluster algebras and monoidal
  categorification conjectures}

\author{Fan QIN}
\dedicatory{To the memory of Professor Andrei Zelevinsky}
\email{qin.fan.math@gmail.com}

\begin{abstract}
We consider the quantum cluster algebras which are
injective-reachable and introduce a triangular basis in every seed. We prove that, under some initial conditions, there exists a
unique common triangular
basis with respect to all seeds. This basis is parametrized by tropical
points as expected in the Fock-Goncharov conjecture.

As an application, we prove the existence of the common
triangular bases for the quantum cluster algebras arising from
representations of quantum affine algebras and partially for those
arising from quantum unipotent subgroups. This result implies monoidal
categorification conjectures
of Hernandez-Leclerc and Fomin-Zelevinsky in the corresponding
cases: all cluster monomials correspond to simple modules.

\end{abstract}
\maketitle
\setcounter{tocdepth}{1}
\tableofcontents

\renewcommand{\qClAlg}{{\cA}}
\renewcommand{\diag}{{d}}

\newcommand{\redISet}{\cI}
\newcommand{\qNilAlg}{\cN}
\newcommand{\norm}{\alpha}
\newcommand{\sym}{\mathcal{S}}
\newcommand{\mHf}{{-\Hf}}
\newcommand{\height}{{\opname{ht}}}

\renewcommand{\inj}{{\bI}}
\renewcommand{\can}{{\bL}}
\newcommand{\Ind}{{\opname{Ind}}}
\newcommand{\Coind}{{\opname{Coind}}}

\newcommand{\wtMap}{{\opname{wt}}}
\newcommand{\gMap}{{\opname{wt}^{-1}}}

\renewcommand{\redWSet}{{\cW}}
\newcommand{\degL}{\mathrm{D}}
\newcommand{\domDegL}{\mathrm{D}^\dagger}

\newcommand{\condL}{\mathrm{L}}
\newcommand{\ord}{\opname{ord}}
\newcommand{\var}{\opname{var}}
\newcommand{\Var}{\opname{Var}}
\newcommand{\ex}{\opname{ex}}
\newcommand{\Diag}{\opname{Diag}}

\newcommand{\ptSet}{\mathcal{PT}}

\newcommand{\roots}{{\mbox{\boldmath $r$}}}
\newcommand{\bpi}{{\mbox{\boldmath $\pi$}}}

\newcommand{\hfg}{\hat{\mathfrak{g}}}
\renewcommand{\trunc}{^{\leq 2l}}


\section{Introduction}
\label{sec:intro}

\subsection{Background}
\label{sec:background}

Cluster algebras were invented by Sergey Fomin and Andrei Zelevinsky around the
year 2000 in their seminal work \cite{FominZelevinsky02}. These are commutative
algebras with generators defined recursively called cluster
variables. The quantum cluster algebras were later introduced in
\cite{BerensteinZelevinsky05}. Fomin and Zelevinsky aimed to develop a combinatorial
approach to the canonical bases in quantum groups
(Lusztig \cite{Lusztig90}, Kashiwara \cite{Kashiwara90}) and theory of total positivity in
algebraic groups (by Lusztig \cite{Lusztig93}\cite{Lusztig94}). They conjectured that the cluster structure should
serve as an algebraic framework for the study of the ``dual canonical bases'' in various coordinate rings and
their $q$-deformations. In particular, they proposed the following
conjecture.
\begin{Conj*}[Fomin-Zelevinsky]
  All monomials in the variables of any given cluster (the
  cluster monomials) belong to the dual canonical basis.
\end{Conj*}
This claim inspired the positivity conjecture of the Laurent
phenomenon of skew-symmetric cluster
algebras, which was recently proved by Lee and Schiffler
\cite{LeeSchiffler13} by elementary algebraic computation.

It soon turns out that (quantum) cluster algebras are related to many other
areas. For
example, by using cluster algebras, Bernhard Keller has proved the
periodicity conjecture of $Y$-systems in physics, \cf
\cite{keller2013periodicity}. We refer the reader to B. Keller's introductory
survey \cite{Keller12} for a review of applications of
cluster algebras. 

However, despite the success in other areas, the original
motivation to study cluster algebras remains largely open. In literature, the following bases have been constructed
and shown to contain the (quantum) cluster monomials.

\begin{enumerate}[(i)]
\item By using preprojective algebras, \cf
  \cite{GeissLeclercSchroeer10}, Geiss, Leclerc, and Schr\"oer have
  shown that, if $G$ is a semi-simple complex algebraic group and
  $N\subset G$ a maximal nilpotent subgroup, then the coordinate
  algebra $\C[N]$ admits a canonical cluster algebra structure. They
  further established the generic basis of this algebra, which
  contains the cluster monomials. As an important consequence, they
  identified the generic basis with Lusztig's dual semi-canonical
  basis, which is a variant of the dual canonical basis, \cf
  \cite{Lusztig00}.

\item For commutative cluster algebras arising from triangulated surfaces,
  Musiker, Schiffler and Williams constructed various bases,
  \cf\cite{musiker2013bases}. Their positivity were discussed in \cite{thurston2013positive}. Fock and Goncharov also
  constructed ``canonical bases'' in \cite{FockGoncharov03conj}.

\item For (quantum) cluster algebras of rank $2$, Kyungyong Lee, Li
  Li, A. Zelevinsky introduced greedy bases, \cf
  \cite{lee2014greedy} \cite{lee2014greedyQuantum}\cite{lee2014existence}.
\end{enumerate}

As a new approach to cluster algebras, David Hernandez
and Bernard Leclerc found a cluster algebra structure on the Grothendieck ring of finite dimensional representations
of quantum affine algebras in \cite{HernandezLeclerc09}. They proposed the following monoidal categorification
conjecture.
\begin{Conj*}[Hernandez-Leclerc]
All cluster monomials belong to the basis of the simple modules.
\end{Conj*}
\begin{Rem}
If we consider Cartan type $ADE$, then the Fomin and Zelevinsky's original conjecture about dual canonical basis can be deduced from Hernandez-Leclerc conjecture, because one can view dual canonical basis elements as simple modules, cf. \cite{HernandezLeclerc11}.

In fact, Hernandez and Leclerc has a stronger version of this conjecture in \cite{HernandezLeclerc09}: there exists a bijection between the cluster monomials and the real simple objects (a simple object $\simp$ is called \emph{real} if $\simp\otimes\simp$ remains simple). Its analogue in cluster categories would be that all rigid objects correspond to cluster monomials.
\end{Rem}

Partial results in type $A$
and $D$ were due to  \cite{HernandezLeclerc09}
\cite{hernandez2013monoidal}. Inspired by the work of Hernandez and Leclerc, Hiraku Nakajima used graded
quiver varieties to study cluster algebras \cite{Nakajima09} and
verified the monoidal categorification conjecture for cluster algebras
arising from bipartite quivers. His work was later generalized to all
acyclic
quivers by Yoshiyuki Kimura and the author
\cite{KimuraQin11} (cf. \cite{Lee13} for a different and quick proof).

By the work of Khovanov,
Lauda \cite{KhovanovLauda08}, and Rouquier \cite{Rouquier08}, a quantum group $U_q(\mfr{n})$
admits a categorification by the modules of quiver Hecke
algebras, in which the dual canonical basis vectors correspond to the finite
dimensional simple modules. Therefore, Fomin and Zelevinsky's conjecture can also be viewed as a monoidal
categorification conjecture of the corresponding (quantum) cluster algebras.

\subsection{Construction, results and comments}
\label{sec:main_result}

We shall consider quantum cluster algebras $\clAlg$ that are
injective-reachable, namely, for any seed $t$, we have a similar new seed $t[1]$ whose cluster variables are the ``injectives" of $t$ (Section \ref{sec:injective_reachable}). Notice that this assumption is satisfied by the cluster algebras appearing both in Fomin-Zelevinsky's conjecture and Hernandez-Leclerc's conjecture. Starting from $t[1]$, we again have a similar new seed $t[2]$. Repeating
this process, we obtain a chain of
seeds $\set{t[d]}_{d\in \Z}$.

Inspired by Kazhdan-Lusztig theory \cite{kazhdan1979representations}, with respect to any given seed $t$, we define the weakly triangular basis $\can^t$ to be the unique bar-invariant basis, such that normalized products of quantum cluster variables in $t$ and $t[1]$ has unitriangular decompositions into $\can^t$ with non leading coefficients in $q^{-\Hf}\Z[q^{-\Hf}]$. $\can^t$ is parametrized by the set of tropical points $\degL(t)$. It is further called the triangular basis \wrt $t$ if we can strengthen its triangularity property by an analog of Leclerc's conjecture (Section \ref{sec:injective_pointed} \ref{sec:triangular_basis_brief} \ref{sec:weakly_triangular_basis} Remark \ref{rem:Leclerc_conjecture}).

By a common triangular basis \wrt different given seeds we mean a basis which is triangular \wrt each given seed and whose parametrization is compatible with tropical transformations. Our Existence Theorem (Theorem \ref{thm:induction})
states that if there exists a common triangular basis \wrt the seeds
$\set{t[d]}$, $d\in\Z$, such that it has positive
  structure constants, then this
  basis lifts to the unique common triangular basis $\can$ \wrt all
  seeds, called the common triangular basis of $\qClAlg$.

Furthermore, because the seeds $t[d]$ are similar (Section \ref{sec:correction}), the existence of $\can$ is guaranteed by Reduced Existence Theorem (Theorem \ref{thm:reduction}) which only demands the finite criterion that the cluster variables
along the mutation sequence from $t$ to $t[1]$ and $t[-1]$ are
contained in the triangular basis $\can^t$.

By construction, the common triangular basis $\can$ contains all the quantum
cluster monomials. Therefore, if we can show that the basis consisting of
the simple modules (or the dual canonical basis) produces the common triangular
basis, the monoidal
categorification conjecture is verified.

As applications, we prove the main results of this paper. 
\begin{Thm}[Main results, Theorem
	\ref{thm:consequence} \ref{thm:simple_module} \ref{thm:canonical_basis}]
The common triangular basis exist
in the following cases.
\begin{enumerate}[(I)]
	\item Assume that $\qClAlg$ arises from the quantum unipotent subgroup $A_q(\mfr{n}(w))$
	associated with a reduced word $w$ (called type (i)). If the Cartan
	matrix is of Dynkin type $ADE$, or if the word $w$ is inferior to an
	adaptable word under the left or right weak order, then the dual canonical basis
	gives the common triangular basis after normalization and localization at the frozen
	variables. 
	
	\item Assume that $\clAlg$ arises from finite dimensional modules of quantum affine algebras (called type
	(ii)). Then, after localization at the frozen variables, the basis of
	the simple modules gives the common triangular basis. 
	
	In particular, all cluster monomials correspond to simples modules. This type includes those $\clAlg$ arising from level $N$ categories $\cC_N$ in the sense of \cite{HernandezLeclerc09}, thus verifying the Hernandez-Leclerc conjecture at all levels $N$.
\end{enumerate}
\end{Thm}

We also find an easy proof of Conjecture
\ref{conj:integral_quantum_group} about cluster structures on all integral quantum
unipotent subgroups, \cf Theorem \ref{thm:isom_integral}.

\begin{Rem}[Key points in the proofs]
The existence theorem \ref{thm:induction} is established by elementary
algebraic computation and the combinatorics of cluster algebras. We control the triangular bases by the positivity of
their structure
constants as well as their expected compatibility with the tropical
transformations. On the one hand, the expected compatibility implies that the new quantum
cluster variables obtained from one-step mutations are contained in
the triangular basis of the original seed, \cf Proposition \ref{prop:prove_X} Example
\ref{eg:control_expansion}. On the
other hand, if these new quantum cluster variables are contained
in the original triangular basis, we can construct the triangular bases \wrt the new seeds which are compatible with the original one, \cf Proposition \ref{prop:condition_change_seed}.

When we consider applications of Reduced Existence Theorem (Theorem \ref{thm:reduction}) to a quantum cluster algebra $\clAlg$ of type (i) or (ii), most initial conditions are obtained by induction based on the properties of the standard basis (or dual PBW basis) and the $T$-systems. It remains to verify that finite many
quantum cluster variables are contained in the initial triangular basis
$\can^{t_0}$. Some of them are known to be Kirillov-Reshetikhin
modules, \cf\cite{Nakajima03}\cite{GeissLeclercSchroeer11}\cite{hernandez2013monoidal}. We
compare the rest with the characters of some variants of Kirillov-Reshetikhin modules in
 the language of graded quiver varieties, \cf Section
 \ref{sec:consequence_KR_module}.
 
In the proofs, we often simplify our treatment by
 changing the coefficient part of the cluster algebra (this technique was introduced in \cite{Qin12} and called the correction
 technique , \cf Theorem \ref{thm:correction}).
\end{Rem}

  \begin{Rem}[Berenstein-Zelevinsky triangular basis]
    The idea of constructing the common triangular basis of a quantum
    cluster algebra first appeared
    in the work of Berenstein and Zelevinsky
    \cite{BerensteinZelevinsky12}, whose construction was over acyclic
    seeds. Their definition is very different from ours. For general acyclic seeds, it is not clear whether their
    construction agrees with the
    triangular basis in this paper or not.
  \end{Rem}

  \begin{Rem}[Assumption status]
     Existence Theorem relies on Cluster Expansion Assumption. This assumption is well known for cluster algebras arising from quivers (from skew-symmetric matrices)
thanks to the categorification by cluster categories. It might be possible to verify this assumption based on the recent paper \cite{gross2014canonical}.
  \end{Rem}

Recently, Gross, Hacking, Keel and Kontsevich \cite{gross2014canonical} found
another positive basis of commutative cluster algebras parametrized by
tropical points and a more precise form of the
Fock-Goncharov dual basis conjecture was verified. Their work also implied the
positivity of the Laurent phenomenon of commutative cluster algebras. 

When the author was preparing the present article, Seok-Jin Kang, Masaki Kashiwara, Myungho Kim and Se-Jin Oh established
an approach to cluster algebras on quantum unipotent subgroups by studying simple
modules of quiver Hecke algebras, which reduced the verification of
the corresponding monoidal
categorification conjecture to one-step mutations in one seed \cite{KKKO14}. This result allows them to
verify the conjecture
for quantum cluster algebras arising from quantum unipotent subgroups in a forthcoming
paper, namely, all type (i) quantum cluster algebras. The approaches in their paper and the present paper are very different and the results are different too: they verify all type (i) cluster algebras, while the present paper verify some type (i) and all type (ii). Also, we will see that our triangular basis is generated by the basis of simples and, consequently, their result implies our Conjecture \ref{conj:desired_module} for all type (i).

\subsection{Contents}
\label{sec:content}

In the first part of Section \ref{sec:preliminaries}, we recall basic definitions and properties of
quantum cluster algebras $\qClAlg$ and its categorification (including the Calabi-Yau reduction). This
part serves as a background for the construction of the common
triangular basis.

In Section
\ref{sec:dominant}, we consider the
lattice $\degL(t)$ of leading degrees for every seed $t$ of a
cluster algebra together with its dominance order $\prec_t$. We then define and
study pointed elements in the quantum torus $\cT(t)$, whose name is borrowed from \cite{lee2014greedy}. We recall the tropical transformation expected by \cite{FockGoncharov03}. We end the
section by a natural discussion of dominant degrees.

In Section \ref{sec:correction}, we recall and reformulate the correction technique
introduced in
\cite{Qin12}, which is easy but useful to keep track of equations involving
pointed elements when the coefficients and quantization change.

In Section \ref{sec:pointed_sets}, we construct and study the injective pointed set
$\inj^t$ consisting of normalized products of the cluster variables and
the injectives in any given seed $t$. We are interested in the cluster expansions of cluster variables related to injective and projective modules and impose Cluster Expansion Assumption.

In section \ref{sec:triangular_basis}, we introduce the (common) triangular bases
$\can^t$ of a quantum
cluster algebra and give theorems guaranteeing their existence. We present basic properties of $\bm$-unitriangular decompositions and introduce notions such as weakly triangular bases and admissibility. As a crucial step, we
study one step mutations (Section \ref{sec:one_step_mutation}). We
  conclude the section by proving existence theorems for the common triangular basis.

We devote the rest of the paper to the existence and applications of the common
triangular bases.

In Section \ref{sec:quiver_variety}, we review the theory of graded
quiver varieties for acyclic quivers and the associated
Kirillov-Reshetikhin modules. As the main result, we compute $q,t$-characters of a special class of simple
modules. These slightly generalize Nakajima's notations and results
\cite{Nakajima04}\cite{Nakajima03}. A reader unfamiliar with this topic might skip this section
and refer to Nakajima's works.

In Section \ref{sec:preliminaries_monoidal}, we review the notion of monoidal categorification, the quantum cluster
algebras $\qClAlg(\bi)$ associated with a word $\bi$, and the monoidal
categorification conjecture in type (i). We then introduce the quantum cluster
structure $\qClAlg(\bi)$ on a subring $\extQuotKGp(\bi,\ua)$ of the graded Grothendieck ring over graded quiver
varieties, where the word $\bi$ and grading $\ua$ are adaptable. We present the corresponding
monoidal categorification conjecture.

In section \ref{sec:application}, we show that the basis of the simple
modules produces the initial triangular bases of the corresponding
quantum cluster algebra $\qClAlg(\bi)$ after localization. In order to show that it lifts to the common triangular
basis, it remains to check that certain cluster variables are
contained in this basis. We mainly work in the case when $\bi$ is the power
$c^{N+1}$ of an acyclic Coxeter word $c$. Consequences are summarized in the
end.


\section*{Acknowledgments}
\label{sec:ack}
Special thanks are due to Bernhard Keller and Christof Geiss for their
encouragements. The author is grateful to Hiraku Nakajima and Yoshiyuki Kimura
for important discussion and comments. He thanks David Hernandez, Pierre-Guy
Plamondon, and Pavel Tumarkin for interesting discussion on quantum affine algebras,
cluster algebras and Coxeter groups. He also thanks Andrei Zelevinsky and
Dylan Rupel for suggesting the name ``pointed element''. He thanks
Yoshiyuki Kimura, David Hernandez, Bernhard Keller and Changjian Fu for remarks. He thanks Qiaoling Wei for suggestions on revising the manuscript. He is grateful to the anonymous reviewers for the careful reading and many suggestions.

The author thanks the Mathematical
Sciences Research Institute at Berkeley for the financial support,
where this project was initiated during his stay in 2012. 

\section{Preliminaries on additive categorification}
\label{sec:preliminaries}

\subsection{Quantum cluster algebras}
\label{sec:quantum_cluster_algebra}

We review the definition of quantum cluster algebras introduced by
\cite{BerensteinZelevinsky05}. Our notations will be similar to those in
\cite{Qin10}\cite{KimuraQin11}.

Let $m\geq n$ be two non-negative integers. Let $\tB$ be an $m\times n$
integer matrix. Its entry in position $(i,j)$ is denoted by $b_{ij}$. We define the \emph{principal part} and the
\emph{coefficient pattern} of the matrix $\tB$ to be its upper
$n\times n$ submatrix $B$ and lower $(m-n)\times n$ submatrix
respectively. We shall always assume the matrix $\tB$ to be of full
rank $n$ and its principal part matrix skew-symmetrizable, namely, there exists a diagonal matrix $D$ whose diagonal entries are strictly positive integers, such that $BD$ is skew-symmetric.

The indices $1,2,\ldots,m$ are called vertices. Let the intervals $[1,n]$, $[1,m]$ denote their sets $\{1,2,\ldots,n\}$, $\{1,2,\ldots,m\}$ for simplicity.

A \emph{quantization matrix} $\Lambda$ is an $m\times m$
skew-symmetric integer matrix.

\begin{Def}[Compatible pair]
  The pair $(\Lambda,\tB)$ is called \emph{compatible}, if their
  product satisfies
  \begin{align}
    \Lambda\cdot (-\tB)=\left(
      \begin{array}{c}
        D\\
0
      \end{array}
\right)  
  \end{align}
for some $n\times n$ diagonal matrix $D=\Diag(\mathbf{d}_1,\mathbf{d}_2,\ldots, \mathbf{d}_n)$
whose diagonal entries are strictly positive integers.
\end{Def}
By \cite{BerensteinZelevinsky05}, the product $DB$ is skew symmetric.

The quantization matrix $\Lambda$ gives the bilinear form $\Lambda(\ ,\ )$ on the
lattice $\Z^m$ such that $\Lambda(g,h)=g^T \Lambda h$ for $g,h\in
\Z^m$. Here we use $(\ )^T$ to denote the matrix transposition.

Let $q$ be an indeterminate. Define $q^\Hf$ to be its square root such
that $(q^\Hf)^2=q$. The quantization matrix $\Lambda$ makes the usual Laurent polynomial ring into a quantum torus.
\begin{Def}\label{def:quantum_torus}
  The quantum torus $\cT$ is the Laurent polynomial ring
  $\Z[q^{\pm\Hf}][X_1^\pm,\ldots,X_m^\pm]$ endowed with the twisted
  product $*$ such that
  \begin{align*}
    X^g*X^h=q^{\Hf\Lambda(g,h)}X^{g+h},
  \end{align*}
for any $g,h\in \Z^m$.
\end{Def}
We denote the usual product of $\cT$ by $\cdot$. The notation $X^g$ denotes the monomial $\prod_{1\leq i\leq
  m}X_i^{g_i}$ of the \emph{$X$-variables} $X_i$ with respect to the usual product. Recall that the \emph{$Y$-variables}
$Y_k$, $1\leq k\leq n$, are defined as $X^{\tB e_k}$, where $e_k$ is
the $k$-th unit vector in the lattice $\Z^m$. The monomials $Y^v$,
$v\in \N^n$, are similarly defined as $\prod_{k}Y_k^{v_k}$.

By \cite{BerensteinZelevinsky05}, the quantum torus $\cT$ is contained
in its \emph{skew-field of fractions} $\cF$.

The bar involution $\overline{(\ )}$ on $\cT$ is the anti-automorphism of
$\cT$ sending $q^\Hf$ to $q^{-\Hf}$ and $X^g$ to $X^g$.

The quantum torus $\cT$ contains the subalgebra $\mathbb{A}_B=\Z[q^{\pm\Hf}][Y_1,\ldots,Y_n]$. The completion of $\mathbb{A}_B$ is constructed by using the maximal ideal generated by $\{Y_k\}_{1\leq k\leq n}$ and denoted by $\hat{\mathbb{A}}_B$, which has a natural bar-involution such that $\overline{q^sY^v}=q^{-s}Y^v$, $v\in\N^n$. Following \cite[Section 1.2]{gross2014canonical} \cite[Section 2.1]{Davison16}, the completion of the quantum torus $\cT$ is defined to be
\begin{align}\label{eq:completion_torus}
\hat{\cT}=\cT\otimes_{\mathbb{A}_B}\hat{\mathbb{A}}_B.
\end{align}
Its bar involution is induced from that of $\cT$ and $\hat{\mathbb{A}}_B$.

 View $\hat{\cT}$ as a natural right $\hat{\mathbb{A}}_B$-module. Let $\{Z_i\}$ be any (possibly infinite) collection of elements in some finitely generated submodule $\sum_{j=1,\ldots,s} X^{\eta^{(j)} }\hat{\mathbb{A}}_B$ , where $\eta^{(j)}\in\degL(t)$. Assume that for any Laurent degree $\eta$, there are finitely many $Z_i$ with non-zero coefficient $(Z_i)_\eta$ at the Laurent monomial $X^\eta$. Then we can define the possibly infinite sum in the submodule
 \begin{align}\label{eq:infinite_sum}
 \sum_i Z_i=\sum_{\eta\in\Z^m} (\sum_i (Z_i)_\eta)X^\eta.
 \end{align}

A sign $\epsilon$ is an element in $\set{-,+}$. For any $1\leq k\leq
n$ and $\epsilon$, we associate to $\tB$ the $m\times m$ matrix
$E_\epsilon$ whose entry at position $(i,j)$ is
\begin{align*}
  e_{ij}=\left\{
      \begin{array}{ll}
        -1&\mathrm{if}\ i=j= k\\
        \max(0,-\epsilon b_{ik})&\mathrm{if}\ i\neq k,j=k\\
        \delta_{ij}&\mathrm{if}\ j\neq k
      \end{array}\right.,
\end{align*}
and the $n\times n$ matrix $F_\epsilon$ whose entry at position $(i,j)$ is
\begin{align*}
  f_{ij}=\left\{
      \begin{array}{ll}
        -1&\mathrm{if}\ i=j=k\\
        \max(0,\epsilon b_{kj})&\mathrm{if}\ i= k,j\neq k\\
        \delta_{ij}&\mathrm{if}\ i\neq k
      \end{array}\right. .
\end{align*}

\begin{Def}[Quantum seed]
A quantum seed is a triple $(\Lambda,\tB,\uX)$ such that $(\Lambda,\tB)$
is a compatible pair and $\uX=\set{X_1,\ldots,X_m}$ the $X$-variables
in the corresponding quantum torus $\cT$.
\end{Def}

Let $[\ ]_+$ denote the function $\max\set{\ ,0}$.
\begin{Def}[Mutation \cite{FominZelevinsky02} \cite{BerensteinZelevinsky05}]\label{def:mutation}
  Fix a sign $\epsilon$ and an integer $1\leq k\le n$. The mutation $\mu_k$ on a seed
  $(\Lambda,\tB,\uX)$ is the operation that generates the new seed
  $(\Lambda',\tB',\uX')$:
  \begin{align*}
    \tB'=&E_\epsilon \tB F_\epsilon,\\
\Lambda'=&E_\epsilon^T \Lambda E_\epsilon,\\
X'_i=&X_i,\ \forall\ 1\leq i\leq m,i\neq k,\\
X_k'=&X_k^{-1}\cdot (\prod_{1\leq i\leq m}X_i^{[b_{ik}]_+}+\prod_{1\leq j\leq m}X_j^{[-b_{jk}]_+}).
  \end{align*}
\end{Def}
The last equation above is called the \emph{exchange relation} at the vertex
$k$. By this definition, the quantum torus $\cT$ and the new quantum torus
$\cT'$ associated with the new triple share the same skew-field of fractions. By \cite[Proposition 6.2]{BerensteinZelevinsky05}, for any element $Z\in \cT\cap \cT'$, $Z$ is bar-invariant in $\cT$ if and only if it is bar-invariant in $\cT'$.

Choose an initial quantum seed $t_0=(\Lambda,\tB,\uX)$. We consider all seeds $t=(\Lambda(t),\tB(t),\uX(t))$ obtained from $t_0$ by iterated mutations at directions $1\leq k\leq n$. be a given quantum seed. The
$X$-variables $X_i(t)$, $1\leq i\leq m$, are called \emph{quantum cluster variables}. The \emph{quantum cluster monomials} are the monomials
of quantum cluster variables from the same seed. 

Notice that, when
$j>n$, the cluster variables $X_j(t)$ remain the same for all seeds
$t$. Consequently, we call them the \emph{frozen variables} or \emph{coefficients} and simply
denote them by $X_j$. Correspondingly, the vertices $\{n+1,\ldots,m\}$ are said to be \emph{frozen}. The vertices $1,2,\ldots,n$ are said to be
\emph{exchangeable} (or \emph{unfrozen}) and form the set $\ex$. Let $\Z[q^{\pm\Hf}]P$ denote the
coefficient ring $\Z[q^{\pm\Hf}][X_j^\pm]_{n<j\leq m}$. The frozen
variables commute with any quantum cluster variable up to a power
of $q^\Hf$, which situation is called \emph{q-commute} or \emph{quasi-commute}.

\begin{Def}[Quantum cluster algebras]
	The quantum cluster algebra $\qClAlg^\dagger$ is the
	$\Z[q^{\pm\Hf}]$-subalgebra of $\cF$ generated by the quantum cluster
	variables $X_i(t)$, $\forall t$, $\forall 1\leq i\leq m$, with respect to the
	twisted product.
	
 The quantum cluster algebra $\qClAlg$ is the localization of $\qClAlg^\dagger$ at the frozen variables $X_{n+1},\ldots,X_m$.
\end{Def}
The specialization $\qClAlg|_{q^\Hf\mapsto 1}$ gives us the
\emph{commutative cluster algebra}, which is denoted by $\clAlg_\Z$. 

For any given seed $t$, we define the corresponding quantum torus $\cT(t)$ using the data provided by $t$ and Definition \ref{def:quantum_torus}. Then the quantum cluster algebra is contained
in the quantum torus $\cT(t)$, \cf the \emph{quantum Laurent
phenomenon} \cite{FominZelevinsky02}\cite{BerensteinZelevinsky05}

Specialize $q^\Hf$ to $1$ and fix the initial seed $t_0$. For any given cluster variable $X_i(t)$, $1\leq i\leq m$, by
\cite{Tran09} \cite{FominZelevinsky07}, we define its \emph{extended $g$-vector} $\tg_i(t)$ in $\Z^m$ such that the commutative cluster variable takes the following form:
$$X_i(t)|_{q^\Hf\mapsto 1}=X^{\tg_i(t)} F(Y_1,\ldots,Y_n),$$
where $F(Y_1,\ldots,Y_n)$ is the $F$-polynomial in \cite{FominZelevinsky07}. The restriction of $\tg_i(t)$ to the first
$n$ components is called the corresponding \emph{$g$-vector}, which is the
grading of the cluster variable $X_i(t)$ when the $B$-matrix is of
\emph{principal coefficients type}, namely, when $$\tB(t_0)=B^{\pr}(t_0)=
\left(\begin{array}{c}
  B(t_0)\\
\id_n
\end{array}\right).$$

We refer the reader to the survey \cite{Keller12} for more details on
the combinatorics of cluster algebras.

\begin{Thm}[{\cite{DerksenWeymanZelevinsky09}\cite{gross2014canonical}
	\cite[Theorem 6.1]{Tran09}}]\label{thm:quantum_cluster_expansion}
Any quantum cluster variables $X_i(t)$
	has the following Laurent expansion in the quantum torus $\cT(t_0)$:
	\begin{align}\label{eq:quantum_cluster_expansion}
	X_i(t)=X(t_0)^{\tg_i(t)}\cdot(1+\sum_{0\neq v\in \N^n}c_v(q^\Hf)Y(t_0)^v),
	\end{align}
	where the coefficients $c_v(q^\Hf)\in\Z[q^{\pm\Hf}]$ are invariant under the bar involution.
\end{Thm}
\subsection{Quivers and cluster categories}
\label{sec:ice_quiver}

We review the necessary notions on the additive categorification of a
cluster algebra by
cluster categories, more details could be found in \cite{Keller08c} \cite{Plamondon10a}.

Let $\tB$ be the $m\times n$ matrix introduced before. Assume that its
principal part $B$ is skew-symmetric.

Then we can construct the
quiver $\tQ=\tQ(\tB)$. This is an oriented graph with the vertices
$1,2,\ldots,m$ and $[b_{ij}]_+$ arrows from  $i$ to $j$, where $1\leq
i,j\leq m$. We call $\tQ$ an \emph{ice quiver}, its vertices $1,\ldots, n$
the exchangeable vertices, and its vertices $n+1,\ldots,
m$ the frozen vertices. Its principal part is defined to be the full
subquiver $Q=Q(B)$ on the vertices $1,2,\ldots,n$. 

Notice that we can deduce the matrix $\tB$ from its ice quiver $\tQ$. Therefore, we say that the corresponding cluster algebra is of quiver type.

For any given arrow $h$
in a quiver, we let $s(h)$ and $t(h)$ to denote its start and terminal respectively. The quiver $Q$ is called\emph{ acyclic} if it contains no oriented
cycles, and \emph{bipartite }if any of its vertices is either a source point or a sink point.

\begin{Eg}
Figure \ref{fig:A_3_injective} provides an example
  of an ice quiver, where we use the diamond nodes to denote the frozen vertices.
\end{Eg}

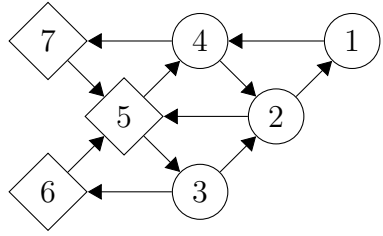
\begin{figure}[htb!]
 \centering
\beginpgfgraphicnamed{fig:A_3_injective}
  \begin{tikzpicture}
\node [shape=circle, draw] (v1) at (7,4) {1}; 
    \node  [shape=circle, draw] (v2) at (6,3) {2}; 
    \node [shape=circle,  draw] (v3) at (5,2) {3};

\node  [shape=circle, draw] (v4) at (5,4) {4}; 
    \node [shape=diamond,  draw] (v5) at (4,3) {5};

\node  [shape=diamond, draw] (v6) at (3,2) {6};
    \node  [shape=diamond, draw] (v7) at (3,4) {7};

    \draw[-triangle 60] (v3) edge (v2); 
    \draw[-triangle 60] (v2) edge (v1);

    \draw[-triangle 60] (v5) edge (v4);

    \draw[-triangle 60] (v7) edge (v5); 
    \draw[-triangle 60] (v5) edge (v3);
    
    \draw[-triangle 60] (v4) edge (v2);

    \draw[-triangle 60] (v1) edge (v4); 
    \draw[-triangle 60] (v4) edge (v7); 

    \draw[-triangle 60] (v2) edge (v5); 
\draw[-triangle 60](v3)edge (v6);

\draw[-triangle 60](v6)edge (v5);
  \end{tikzpicture}
\endpgfgraphicnamed
\caption{The ice quiver associated the adaptable word $(1
  3 2 1 3 2 1)$ and Cartan type $A_3$}
\label{fig:A_3_injective}
\end{figure}

We fix the base field to be the complex field $\C$. Following
\cite{DerksenWeymanZelevinsky07}, choose a generic
potential $\tW$ for the ice quiver $\tQ$, namely, a generic linear combination of oriented cycles of $\tQ$ which is not included in some hypersurfaces (\cite[Corollary 7.4]{DerksenWeymanZelevinsky07}). Associate the Ginzburg algebra
$\Gamma=\Gamma(\tQ,\tW)$ to the quiver with potential $(\tQ,\tW)$. The restriction of
  $(\tQ,\tW)$ to the vertices $[1,n]$ is called the principal
  part and denoted by $(Q,W)$.

Claire Amiot introduced the generalized cluster category $\cC$ in
\cite{Amiot09} as
the quotient category
\begin{align*}
  \cC=\per \Gamma/\cD_{fd}\Gamma,
\end{align*}
where $\per \Gamma$ is the perfect derived category of dg-modules of $\Gamma$ and
$\cD_{fd}\Gamma$ its full subcategory whose objects have finite
dimensional homology.

The natural functor $\pi:\per \Gamma\ra \cC$ gives us
\begin{align*}
T_i=&\pi (e_i\Gamma),\ \forall\ 1\leq i\leq m,\\
  T=&\oplus T_i.
\end{align*}

By Pierre-Guy Plamondon's work \cite{Plamondon10a}, the presentable
cluster category $\cD$ is defined as the full subcategory of
$\cC$ which consists of objects $M$ such that there exist triangles
\begin{align*}
  M_1\ra M_0\ra M\ra S M_1,\\
  M^0\ra M^1\ra S^{-1}M\ra S M^0,
\end{align*}
with $M_1,M_0,M^0,M^1$ in $\add T$, and $S$ is the shift functor in $\cC$. Following \cite{Palu08a}, we define the
index and coindex of $M$ as the following elements in $K_0(\add T)$:
\begin{align*}
  \Ind^TM=[M_0]-[M_1],\\
\Coind^TM=[M^0]-[M^1].
\end{align*}
It is natural to identify $K_0(\add T)$ with $\Z^m$ by using the indecomposable summands $T_i$, $1\leq i\leq m$, of $T$ as the natural basis elements. Then we have
\begin{align}\label{eq:define_index_coindex}
\begin{split}
  \Ind^TM=[M_0:T]-[M_1:T],\\
\Coind^TM=[M^0:T]-[M^1:T],
\end{split}
\end{align}
where $[\ :T]$ denote the multiplicities of the indecomposable summands of $T$ appearing in an element of $K_0(\add T)$. Notice that
$\Coind^TM=-\Ind^{ T}S^{-1}M$.

Let $F$ denote\footnote{Our functor $F$ is the composition of $S$ with the functor $F$ in \cite{Plamondon10a}.} the functor $\Ext^1_\cC(T,\ )$ from $\cD$ to the category
of right modules of $\End_\cC T$. Notice that the module $FM$ is supported on $m$-vertices the quiver $\End_\cC (\oplus_{1\leq k\leq m}T_i)$ whose dimension is given by $\Ext^1_\cC(\oplus_{1\leq i\leq m}T_i,S^{-1}M)$.

If $\Ext(T_{n+i},M)=0$, $\forall 1\leq i\leq m-n$, by \cite[Lemma
2.1 2.3]{Palu08a}, \cite[Lemma 3.6,
Notation 3.7]{Plamondon10a}, we have
\begin{align}\label{eq:ind_coind}
\Ind^{S^{-1} T} M+\Ind^T M=-\tB\cdot \dim \Ext^1_\cC(\oplus_{1\leq i\leq n}T_i,M).
\end{align}

An object $M$ in $\cC$ is said to be \emph{coefficient-free} if
\begin{enumerate}[(i)]
\item it contains no direct summand $T_j$ for $j>n$, and
\item the module $FM$ is supported at the principal quiver $Q$.
\end{enumerate}

By \cite{DerksenWeymanZelevinsky09} \cite{Plamondon10a}, we define the
cluster characters for coefficient-free objects $M$ in $\cD$ as
\begin{align*}
  x_M=x^{\Ind^T M}(\sum_{e\in\N^n} \chi(\Gr_eFM)Y^e),
\end{align*}
where $\Gr_eFM$ denotes the variety of the $e$-dimensional submodules of $FM$ and $\chi$ the Euler-Poincar\'e characteristic.

By \cite{Plamondon10a}, there exists a unique way of associating an
object $T(t)=\oplus_{1\leq i\leq m}T_i(t)$ in $\cD$ with each seed
$t$ such that
\begin{enumerate}
  \item $T(t_0)=T$,
\item if two vertices $t$ and $t'$ are related by an edge labeled $k$,
  then the object $T(t')$ is obtained from $T(t)$ by replacing one
  summand in a unique way.
\end{enumerate}

\begin{Thm}[\cite{DerksenWeymanZelevinsky09}\cite{Plamondon10a}\cite{Nagao10}]
For any seed $t$ and vertex $1\leq i\leq m$, the commutative cluster variable $X_i(t)|_{q^\Hf\mapsto 1}$ equals the cluster character $x_{T_i(t)}$ of the object $T_i(t)$.
\end{Thm}

\subsection{Calabi-Yau reduction}
\label{sec:CY_reduction}

In this subsection, we assume that the \emph{Jacobian algebra} of the
quiver with potential $(\tQ,\tW)$ is finite dimensional. Consider the
full subcategory of $\cU$ consisting of the coefficient-free
objects. Take the ideal $(\Gamma_F)$ consisting of the morphisms factoring through some
$T_j$, $j>n$. Then, by \cite{IyamaYoshino08}\cite{plamondon2013generic}, $\underline{\cU}=\cU/(\Gamma_F)$ is equivalent
to the generalized cluster category
of $(Q,W)$.

For any
coefficient-free object $X$,
denote its image in $\underline{\cU}$ by
$\underline{X}$. Let $\Sigma$ denote the shift functor in
$\cU/(T_F)$. Any given $\Sigma^d
\underline{T_k}$, $1\leq k\leq n$, $d\in \Z$, has the lift $\Sigma^d T_k$ by
\cite{IyamaYoshino08}. Define $\Sigma^d T=(\oplus_{k=1}^n \Sigma^d T_k)\oplus (\oplus_{n<j\leq m} T_j)$.


For any given $1\leq k\leq n$, by taking $M=\Sigma^{-1}T_k$ in
\eqref{eq:ind_coind}, we get
\begin{align*}
  \Ind^{S^{-1}T}\Sigma^{-1}T_k+\Ind^T (\Sigma^{-1}T_k)=-\tB(T)\cdot \dim\Ext^1_\cC(\oplus_{1\leq i\leq n}T_i,\Sigma^{-1}T_k).
\end{align*}

\begin{Lem}
  We have $\Ind^{S^{-1}T}\Sigma^{-1}T_k=-\Ind^{\Sigma ^{-1} T}T_k$.
\end{Lem}
\begin{proof}
$T_k$ fits into a triangle of the following type (\cite[Lemma 3.15]{plamondon2013generic}):
\begin{align*}
  \Sigma^{-1}T_k\ra T_F\ra T_k\ra  S\Sigma^{-1}T_k,
\end{align*}
where $T_F$ is some direct sum of $T_j$, $j>n$. Rewrite it as the
triangle
\begin{align*}
 S^{-1}T_F\ra S^{-1}T_k\ra \Sigma^{-1}T_k\ra T_F.
\end{align*}

Recall by \eqref{eq:define_index_coindex} that $\Ind^{S^{-1}T}\Sigma^{-1}T_k$ (resp.$\Ind^{\Sigma ^{-1} T}T_k$) is defined as the multiplicities of the summands of $S^{-1}T$ (resp. $\Sigma^{-1}T$) appearing in the isoclass of $\Sigma^{-1}T_k$ in $K_0(\add S^{-1}T)$ (resp. $K_0(\Sigma^{-1}T)$). Also, we have 
\begin{align*}
S^{-1}T=&(\oplus_{1\leq i\leq n}S^{-1}T_i)\oplus (\oplus_{n<j\leq m}S^{-1}T_j)\\
\Sigma^{-1}T=&(\oplus_{1\leq i\leq n}\Sigma^{-1}T_i)\oplus (\oplus_{n< j\leq m}T_j)
\end{align*}
We deduce that
\begin{align*}
\Ind^{S^{-1}T}\Sigma^{-1}T_k&=[S^{-1}T_k:S^{-1}T]-[S^{-1}T_F:S^{-1}T]\\
&=[\Sigma^{-1}T_k:\Sigma^{-1}T]-[T_F:\Sigma^{-1}T]\\
&=-\Ind^{\Sigma^{-1}T}T_k.
\end{align*}
\end{proof}

The following equality follows as a consequence
\begin{align}\label{eq:quiver_coindex}
\Ind^T (\Sigma^{-1}T_k)+\tB(T)\cdot
\dim\Ext^1_\cC(\oplus_{1\leq i\leq n}T_i,\Sigma^{-1}T_k)=\Ind^{\Sigma ^{-1}T}T_k.
\end{align}

Notice that $\Ext^1_\cC(T,\Sigma^{-1}T_k)$ is the $k$-th projective
right-module of the Jacobian algebra principal of $(Q,W)$. We will
later discuss this equality in Example \ref{eg:compare_inj_proj} with more details.

\section{Dominance order and pointed sets}
\label{sec:dominant}
\subsection{Constructions}
\label{sec:construction}

Let $t$ be any seed. We associate with it a rank $m$ lattice $\degL(t)$, called the \emph{degree lattice}. While $\degL(t)$ is always isomorphic to $\Z^m$, throughout the paper, we shall distinguish $\degL(t)$ for different $t$.
\begin{Def}[Dominance order]\label{def:dominant_order}
For any given $\tg',\tg\in\degL(t)$, we say $\tg'\prec_t\tg$ (or $\tg$
dominates $\tg'$) if we have $\tg'=\tg+\tB(t)v$ for some $0\neq
v\in\N^n$. We denote $\tg'\preceq_t\tg$ if $\tg'\prec_t \tg$ or $\tg'=\tg$.
\end{Def}

\begin{Lem}\label{lem:interval_bounded}
  For any $\tg^1\prec_t\tg^0$, there exists finitely many $\tg^2$ such
  that $\tg^1\prec_t\tg^2\prec_t\tg^0$.
\end{Lem}
\begin{proof}
Notice that $\tB(t)$ is of full rank. The claim follows from definition.
\end{proof}

Recall that the quantum torus $\cT(t)$ is the Laurent polynomial ring generated by $X_i(t),1\leq i\leq m$.

\begin{Def}[Leading term]\label{def:leading_term}
For any Laurent polynomial $Z\in\cT(t)$, the leading terms are those Laurent
monomials of $Z$ such that, among all the Laurent monomials of $Z$ with non-zero
coefficients, they have maximal multidegrees with respect
to the order $\prec_t$.
\end{Def}

\begin{Def}[Pointed element]\label{def:pointed_element}
If a Laurent polynomial
$Z\in\cT(t)$ has a unique leading term whose degree is $\eta$, we
say its leading (or maximal) degree is $\deg^t Z=\tg$. If its leading term has coefficient
$q^\alpha$ for some $\alpha\in\Hf\Z$, we define the
normalization of $Z$ in the quantum torus $\cT(t)$ to be $[Z]^t=q^{-\alpha}Z$.

We say $Z$ is pointed at $\tg$ (or $\tg$-pointed) if $Z$ has a
unique leading term, and this term has degree $\tg$ with coefficient $1$.
\end{Def}
We remark that the normalized element $[Z]^t$ is pointed by construction. For
simplicity, we often denote $\deg^t Z$ and $[Z]^t$ by $\deg Z$ and
$[Z]$ respectively when the context is clear. 

By Theorem \ref{thm:quantum_cluster_expansion}, the Laurent expansions of quantum cluster variables in any quantum seed $\cT(t)$ are pointed elements.

\begin{Notation*}
	The set of Laurent polynomials in $\cT(t)$ pointed in degree $\tg\in\degL(t)$ is denoted by $\ptSet_\tg(t)$. The set of all pointed elements in $\cT(t)$ is denoted by $\ptSet(t)$.
\end{Notation*}
Apparently, we have $\ptSet(t)=\sqcup_{\tg\in\degL(t)}\ptSet_\tg(t)$.

\begin{Def}[Pointed set]
  Let $W$ be any given subset of $\degL(t)$. A subset $\gen $ of $\ptSet(t)$ is
  said to be pointed at $W$ (or $W$-pointed) if the degree map induces a bijection $$\deg^t:\gen \simeq W.$$ 

 In this case, we use $\gen(\tg)$ to denote the preimage of $\tg\in W$ in $\gen$. Namely, we use $\gen(\ )$ to denote the inverse of the degree map.
\end{Def}
 It
immediately follows from the definition that any pointed set is $\Z[q^\Hf]$-linearly independent.

Given any $W$-pointed subset $\gen$ of $\ptSet(t)$ and a pointed element $Z\in\ptSet(t)$, we define the set \begin{align}
[Z*\gen]^t=\{[Z*\gen(\tg)]^t|\tg\in W]\}.
\end{align}
Then $[Z*\gen]^t$ is a $(\deg^t Z+W)$-pointed set, where $\deg^t Z+W$ denote the degree set $\{\deg^t Z+\tg|\tg\in W\}$.

\begin{Eg}
  By \cite{plamondon2013generic}, for any commutative cluster algebra studied in
  \cite{GeissLeclercSchroeer10}, the dual semicanonical basis, after
  localization at the frozen variables, is a
  $\degL(t)$-pointed set for any seed $t$.
\end{Eg}

Let $\{Z_i\}$ be any (possibly infinite) collection of elements in some finitely generated submodule $\sum_{j=1,\ldots,s} X^{\eta^{(j)} }\hat{\mathbb{A}}_B$ of $\hat{\cT}$. Assume that they have distinct unique maximal leading degrees. Then, by Lemma \ref{lem:interval_bounded}, for any Laurent degree $\eta$, only finitely many $Z_i$ have leading degree superior than $\eta$. It follows that the (possibly infinite) sum $\sum_i Z_i$ is well defined by \eqref{eq:infinite_sum}.

\begin{Def}[Unitriangularity]\label{def:triangular}
Let $\gen $ be a given $W$-pointed subset of $\ptSet(t)$ for some $W\subset \degL(t)$. Any
pointed element $Z\in\ptSet(t)$ is said to be $\prec_t$-unitriangular to
$\gen $, if it has the following (possibly infinite) expansion into $\gen$:
\begin{align}\label{eq:triangular}
  Z=\gen(\deg Z)+\sum_{\tg'\prec_t \deg Z}b_{\tg'}\gen(\tg'), b_{\tg'}\in\qBase.
\end{align}
If the coefficients $b_{\tg'}$ are further contained in $\bm=q^\mHf\Z[q^\mHf]$, we say $Z$ is $(\prec_t,\bm)$-unitriangular to $\gen $.

A subset of $\ptSet(t)$ is said to be $\prec_t$-unitriangular
(resp. $(\prec_t,\bm)$-unitriangular) to $\gen $ if all its elements
have this property.
\end{Def}
We might write the notation $\prec_t$ as $\prec$ or simply omit it, when this
order is clear from the context.

\begin{Rem}\label{rem:expansion_algorithm}
It follows from Lemma \ref{lem:interval_bounded} that the coefficients
$b_{\tg'}$ in \eqref{eq:triangular} are uniquely determined by $Z$. In
fact, denote the coefficient of the Laurent expansion of $Z$ at any
given degree $\tg''\prec_t \deg^t Z$ by $z_{\tg''}$ and that of
$\gen(\tg')$ by $a_{\tg',\tg''}$. We must have the following equation:
\begin{align*}
  z_{\tg''}=1\cdot a_{\deg^t Z, \tg''}+\sum_{\tg''\prec_t\tg'\prec_t \deg^t
    Z}b_{\tg'}a_{\tg',\tg''}+b_{\tg''}\cdot 1.
\end{align*}
It follows that the coefficient $b_{\tg''}$ is determined by those
$b_{\tg'}$ appearing. By Lemma \ref{lem:interval_bounded}, this is a
finite algorithm for any given $\tg''$.

In general, there could exist many infinite expansion of $Z$ into
$\gen $ if we don't require them to take the unitriangular form \eqref{eq:triangular}.
\end{Rem}
\begin{Eg}
Take a seed $t$ with $\tB=
\begin{pmatrix}
0\\
-1
\end{pmatrix}$ and, for simplicity, $q=1$. The quantum torus becomes $\cT=\Z[X_1^\pm,X_2^\pm]$. We have $Y_1=X_2^{-1}$, $\mathbb{A}_B=\Z[Y_1]=\Z[X_2^{-1}]$ and the completion $\hat{\cT}=\cT\otimes_{\mathbb{A}_B} \hat{\mathbb{A}}_B=\Z[X_1^\pm,X_2][[X_2^{-1}]]$. 

Let us take the $\Z^2$-pointed set to be $\{X_2^{d_2} X_1^{d_1}|d_2\in\Z,d_1\in\N\}\cup \{X_2^{d_2}(X_1')^{d_1'}|d_2\in\Z,d_1'\in\N\}$, where $X_1'=X_1^{-1}X_2+X_1^{-1}=X_1^{-1}X_2(1+Y_1)$. By the algorithm in Remark \ref{rem:expansion_algorithm}, we obtain the unique $\prec_t$-unitriangular decomposition of the pointed element $Z=X_1^{-1}$ as:
$$
Z=X_2^{-1}X_1'-X_2^{-2}X_1'+X_2^{-3}X_1'-X_2^{-4}X_1'\cdots.
$$
\end{Eg}

\begin{Lem}[Expansion properties]\label{lem:triangular}
Let $\gen $ be any $\degL(t)$-pointed subset of $\cT(t)$ and $Z$ be any pointed element in $\cT(t)$. Then the following statements are true.

(i) $Z$ is $\prec_t$-unitriangular to $\gen $.

(ii) If $Z$ and $\gen $ are bar-invariant and $Z$ is $(\prec_t,\bm)$-unitriangular to $\gen $, then $Z$ equals $\gen(\deg Z)$.

(iii) If $Z$ is a $\qBase$-linear combination of finitely many elements
of $\gen $, then such a combination is unique and takes the
unitriangular form \eqref{eq:triangular}.
\end{Lem}
\begin{proof}
  (i)The statement follows from the algorithm in Remark \ref{rem:expansion_algorithm}.
  
  (ii) We have a $(\prec_t,\bm)$-unitriangular decomposition:
  $$ Z=b+\sum_{\deg^t b'\prec_t \deg^t b} c'b',$$
  where $b,b'\in\gen$, coefficients $c'\in\bm$.
  Because $Z,b,b'$ are invariant under the bar involution, we have all $c'=0$ and, consequently, $Z=b$.

(iii) Suppose that some $\gen^{t}(\eta)$, $\eta\in D(t)$, appears in the
expansion of $Z$ into $\gen $ with non-zero coefficient $b(\eta)$, such that $\eta\nprec_t \deg^t
Z$. Because the expansion is finite, we can always find a degree
$\eta$ such that it is maximal among those having this property. Then
the coefficient of $X(t)^{\eta}$ of $Z$ equals $b(\eta)\neq 0$. But the coefficient of this Laurent monomial in the pointed
element $Z$ must vanish. This contradiction shows that such $\eta$ cannot exist.
\end{proof}

Let us consider triangular matrices $P=(P_{g^1 g^2})_{g^1,g^2\in\degL(t)}$ such that $P_{g^1 g^2}\in \Z[q^{\pm\Hf}]$ and $P_{g^1 g^2}=0$ unless $g^2\preceq_t g^1$. Then the product of two such such infinite matrices is well defined, because its entries are given by finite sum by Lemma \ref{lem:interval_bounded}, and the result remains triangular.

\begin{Lem}[Inverse transition]\label{lem:inverse_transition}
Let $\stdMod$, $\can$ be any two $\degL(t)$-pointed subsets of $\cT(t)$.
If $\stdMod$ is $(\prec_{t},\bm)$-unitriangular to $\can$, then $\can$ is $(\prec_{t},\bm)$-unitriangular to $\stdMod$ as well.
\end{Lem}
\begin{proof}
For any $\tg\in\degL(t)$, let us denote
\begin{align*}
  \stdMod(\tg)&=\can(\tg)+\sum_{\tg^1\prec_t \tg}c_{\tg,
    \tg^1}\can(\tg^1),\ c_{\tg,\tg^1}\in\bm,\\
  \can(\tg)&=\stdMod(\tg)+\sum_{\tg^1\prec_t \tg}b_{\tg,\tg^1}\stdMod(\tg^1).
\end{align*}
Expand $\RHS$ of the second equation by using the first equation. Then
the coefficient of any factor $\can(\tg^2)$, $\tg^2\prec_t\tg$,
appearing in $\RHS$ is the following three terms' sum
\begin{align*}
  c_{\tg,\tg^2}+b_{\tg,\tg^2}+\sum_{\tg^2\prec_t \tg^1\prec_t \tg}b_{\tg,\tg^1}c_{\tg^1,\tg^2},
\end{align*}
which must vanish.

We verify the claim by induction on the order of the degrees $\tg^2$ such that
$\tg^2\prec_t \tg$. If $\tg^2$ is maximal among such degrees, the last term
vanishes and we have $b_{\tg,\tg^2}\in\bm$. Assume $b_{\tg,(\tg^2)'}\in\bm$ for any $(\tg^2)'$ such that $\tg^2\prec_t
(\tg^2)'$. Then the first and last term belong to $\bm$. Consequently,
$b_{\tg,\tg^2}\in\bm$ belong to $\bm$ as well.
\end{proof}

\subsection{Tropical transformation}

\begin{Def}[Tropical transformation, {\cite{FockGoncharov03}\cite[(7.18)]{FominZelevinsky07}}]
  For any $1\leq k\leq n$, the tropical transformation
  $\phi_{\mu_k t, t}:\degL(t)\ra \degL(\mu_k t)$ is the piecewise linear map such that, $\forall \tg=(\tg_j)_{1\leq j\leq m}\in\degL(t)$, the image
  $\tg'=(\tg'_j)_{1\leq j\leq m}=\phi_{\mu_k t,t}(\tg)$ is given by
  \begin{align}\label{eq:degree_mutation}
    \begin{split}
      \tg'_k&=-\tg_k,\\
      \tg'_i&=\tg_i+b_{ik}(t)[\tg_k]_+,\ \mathrm{if}\ 1\leq i\leq
      m,\ i\neq k,\ b_{ik}\geq 0,\\
      \tg'_j&=\tg_j+b_{jk}(t)[-\tg_k]_+,\ \mathrm{if}\ 1\leq j\leq m,\
      j\neq k,\ b_{jk}\leq 0.
    \end{split}
  \end{align}
\end{Def}
\begin{Rem}
The piecewise linear map \eqref{eq:degree_mutation} is the
\emph{tropicalization} of the rule \cite[(13)]{FockGoncharov03}. It describes
the mutation of the \emph{tropical $\Z$-points}
$\mathcal{X}(\Z^m)=\degL(t)$ of the
$\mathcal{X}$-variety (the variety of the \emph{tropical $Y$-variables} for the
cluster algebra defined via $\tB(t)^T$). Fock and Goncharov conjectured that these
tropical points parametrize a ``canonical'' basis, \cite{FockGoncharov03conj}\cite[Section
5]{FockGoncharov03}.

On the cluster algebra side, Fomin and Zelevinsky wrote this
transformation in
\cite[(7.18)]{FominZelevinsky07} and conjectured that it describes
the transformation of the $g$-vectors of cluster variables. When the cluster algebra admits a categorification by a cluster
category, this formula \eqref{eq:degree_mutation} can be interpreted as the
transformation rule of the indices of object in the cluster category,
\cf \cite{DehyKeller07}\cite{KellerYang09}\cite{plamondon2013generic}. We refer the reader
to \cite[Remark 7.15]{FominZelevinsky07} \cite[Section
3.5]{plamondon2013generic}  \cite{FockGoncharov03}for more details.
\end{Rem}

The following observation is obvious.
\begin{Lem}[Sign coherent transformation]\label{lem:additivity}
  Assume that two vectors $\eta^1$, $\eta^2$ in $\degL(t)$ are sign
  coherent at the $k$-th component, namely, we have simultaneously
  $\eta^1_k,\eta^2_k\geq 0$ or $\eta^1_k,\eta^2_k\leq 0$. Then the tropical transformation
\eqref{eq:degree_mutation} is additive on them:
  \begin{align*}
    \phi_{\mu_k t,t}\eta^1+\phi_{\mu_k t,t}\eta^2= \phi_{\mu_k t,t}(\eta^1+\eta^2).
  \end{align*}
\end{Lem}

Given two seeds $t_1,t_2$ which are related by a mutation sequence $\overleftarrow{\mu}$ from $t_1$ to $t_2$, we define $\phi_{t_2,t_1}$ to be the composition of tropical transformations given by \eqref{eq:degree_mutation} along $\overleftarrow{\mu}$.
\begin{Thm}[{\cite[Theorem 1.1 1.3]{plamondon2013generic} \cite{gross2014canonical}}]\label{thm:change_index_compatible}
	We have
	
(i) The transformation $\phi_{t_2,t_1}$ is independent of the choice of the mutation sequence from $t_1$ to $t_2$.

(ii)For any quantum cluster variable $X_i(t)$, we have
  \begin{align}\label{eq:index_change_compatible}
  \deg^{t^2}X_i(t)=\phi_{t^2,t^1}\deg^{t^1}X_i(t).
  \end{align}
\end{Thm}
\begin{proof}
For cluster algebras arise from quivers, the claims were verified in terms of the indices of generic objects in the cluster category, cf \cite[Theorem 1.1 1.3]{plamondon2013generic}. For an arbitrary cluster algebra, the claims can be interpreted as properties of the leading degrees of the theta basis $\{\vartheta_\tg\}$, cf. \cite{gross2014canonical}.
\end{proof}
\begin{Prop}\label{prop:change_Lambda}
	For any seed $t,t'$, the quantization matrices satisfy
	\begin{align*}
	\Lambda(t')(\deg^{t'}X_i(t),\deg^{t'}X_j(t))=\Lambda(t)_{ij}.
	\end{align*}
\end{Prop}
\begin{proof}
	For quantum cluster algebras arising from quivers, the claim was proved in \cite[Proposition 2.4.3(bilinear form)]{Qin10} by using bilinear forms in triangulated categories.
	
By translating
\eqref{eq:degree_mutation} into multiplication by $E_\epsilon$ in
Definition \ref{def:mutation}, we see
that the claim can be deduced from the sign coherence of
$g$-vectors, which was verified in \cite{gross2014canonical} for all quantum cluster algebras.
\end{proof}

Let $\pr_n:\Z^m\ra \Z^n$ denote the projection onto the first
$n$-components. For any $\tg$ in $D(t)$, we denote the projection
$\pr_n \tg$ by $g$.

Notice that $\phi_{t',t}$ has the inverse $\phi_{t,t'}$. As the coefficient-free version of the piecewise linear map $\phi_{t',t}$, we have the following restriction map.

\begin{Lem}
  The restriction $\pr_n\phi_{t',t}|_{\Z^n}:\Z^n\ra \Z^n$ is a
  bijection of sets.
\end{Lem}
\begin{proof}
It suffices to verify the statement for the case $t'=\mu_k t$, where
we have
  $(\pr_n\phi_{t',t})(\pr_n\phi_{t,t'})=\id.$
\end{proof}

We are interested in the pointed sets whose leading degree parametrization is compatible with the tropical transformation rules.
\begin{Def}[Compatible pointed sets]\label{def:compatible_pointed_set}
  Let $t^i$ be a seed and $\gen^{i}$ a $\degL(t^i)$-pointed set, $i=1,2$. We say that $\gen^{1}$ is compatible with $\gen^{2}$ if, for any
  $\eta\in\degL(t^1)$, we have
  \begin{align*}
    \gen^2(\phi_{t^2,t^1}\eta)=\gen^1(\eta).
  \end{align*}
  In particular, we have $\gen^1=\gen^2$ as sets.
\end{Def}

By Theorem \ref{thm:change_index_compatible}, this compatibility
relation is transitive.

\begin{Conj}[Fock-Goncharov basis conjecture]\label{conj:FG_conj}
	$\qClAlg$ has a basis $\gen$ which is in bijection with $\degL(t)$ for any $t$ such that, for any seeds $t,t'$, the following diagram commutes
		$$\qquad\qquad\begin{array}{cccc}
		\gen & \simeq& \degL(t)\\
		|| &  & \downarrow & \phi_{t',t}\\
		\gen & \simeq& \degL(t')
		\end{array}.$$
\end{Conj}

\begin{Rem}\label{rem:progress_FG_conj}
	By \cite{gross2014canonical}, for general cluster algebras, the above Fock-Goncharov basis conjecture does not hold and modification is needed.
	
	There could exist different bases verifying this conjecture. Examples include the generic basis (related to dual semicanonical basis) \cite{plamondon2013generic}\cite{GLSbasis}, the theta basis \cite{gross2014canonical}, and the triangular basis (related to dual canonical basis) in this paper.
\end{Rem}

\subsection{Dominant degree}

We introduce the following cone.
\begin{Def}[Dominant degree cone]
For any given seed $t$, define the dominant degree cone
$\domDegL(t)\subset \degL(t)$ to be
  the cone generated by the leading degrees $\deg^{t}(X_i(t'))$ of all cluster variables, for any $1\leq i\leq
  m$ and seed $t'$. Its elements are called the dominant degrees of the lattice $\degL(t)$.
\end{Def}
By \eqref{eq:index_change_compatible}, we have
$\domDegL(t)=\phi_{t,t_0}\domDegL(t_0)$, $\forall\ t,t_0$.

The following observation follows from the definition of $\domDegL(t)$.
\begin{Lem}
  Let $Z$ be any element in the quantum cluster algebra $\clAlg^\dagger$. Then, for any seed $t$, the $\prec_t$-maximal degrees
  of the Laurent expansion of $Z$ in $\cT(t)$ are contained in $\domDegL(t)$.
\end{Lem}

\begin{Def}[Bounded from below]
  We say that $(\prec_t,\domDegL(t))$ is bounded from below if for any $\tg\in\domDegL(t)$, there
  exists finitely many $\tg'\in\domDegL(t)$ such that $\tg'\prec_t \tg$.
\end{Def}

\begin{Rem}
  For $\cA$ of type (i) or (ii) (Section
  \ref{sec:intro}), let $t_0$ denote the canonical initial seed. In
  Section \ref{sec:verify_initial_condition}, we shall
  see that $(\prec_{t_0},\domDegL(t_0))$ is bounded from below.
\end{Rem}

The following question arises naturally. While a positive answer would be a strong and useful property, the question remains open.
\begin{Quest}
 If $(\prec_{t_0},\domDegL(t_0))$ is bounded from below, is it true that $(\prec_t,\domDegL(t))$ is bounded
 from below for all $t$?
\end{Quest}

\section{Correction technique}
\label{sec:correction}

In \cite[Section 9]{Qin12}, the author introduced the correction technique which tells us how to modify bases data by adding correction terms when the quantization and the coefficient pattern of a
quantum cluster algebra change. In this section, we recall and
reformulate this
correction technique. It will help us simplify many arguments by
coefficient changes.
\subsection{Similarity and variation}

Let $\var^*:\ex^{(2)}\iso \ex^{(1)}$ be any isomorphism of finite $n$-elements sets. We
associate to it an isomorphism $\var:\Z^{\ex^{(1)}}\iso\Z^{\ex^{(2)}}$ such
that we have $(\var v)_{k}=v_{var^* (k)}$, $\forall v=(v_k)\in\Z^{\ex^{(1)}}$. The isomorphism $\var:\Z^{\ex^{(1)}\times \ex^{(1)}}\iso\Z^{\ex^{(2)}\times \ex^{(2)}}$ is defined similarly.

Let $\tB^{(i)}=(b^{(i)}_{jk})$ be an $J^{(i)}\times \ex^{(i)}$ matrix,
$i=1,2$, where $J^{(i)}$ has $m^{(i)}$ elements and $\ex^{(i)}\subset J^{(i)}$ has $n$ elements. The matrices $\tB^{(1)}$ and $\tB^{(2)}$ are called \emph{similar}, if
we have $B^{(2)}=\var B^{(1)}$ for some choice of $\var^*:\ex^{(2)}\iso \ex^{(1)}$.

\begin{Eg}
Consider the matrices $B^{(1)}=\left(
\begin{array}{ccc}
0&1&0\\
-1&0&1\\
0&-1&0
\end{array}
\right)$
and $B^{(2)}=\left(
\begin{array}{ccc}
0&1&-1\\
-1&0&0\\
1&0&0
\end{array}
\right)$. The index sets are $J^{(i)}=\ex^{(i)}=\set{1,2,3}$, where $i=1,2$. Choose 
$\var^*:\ex^{(2)}\ra\ex^{(1)}$ to be the permutation $\left(
  \begin{array}{ccc}
    1&2&3\\
2&3&1\\
  \end{array}
\right)$. Then we have $B^{(2)}=\var B^{(1)}$.
\end{Eg}

\begin{Def}[Similar compatible pair]
  For $i=1,2$, let $(\tB^{(i)},\Lambda^{(i)})$ be compatible pairs
  such that $(\tB^{(i)})^T \Lambda^{(i)}=D^{(i)}\oplus 0$, where
  $D^{(i)}$ is a diagonal matrix
  $\Diag(\diag_k^{(i)})_{k\in \ex^{(i)}}$ with diagonal entries in
  $\Z_{>0}$. We say these two compatible pairs are similar, if there
  exists an isomorphism $\var^*:\ex^{(2)}\ra \ex^{(1)}$, such that
  the principal $B$-matrices $B^{(2)}=\var B^{(1)}$, and, in addition, there
  exists a positive number $\delta$ such that $D^{(2)}=\delta\var D^{(1)}$.
\end{Def}

\begin{Def}[Similar seeds]
	Two seeds $t^1$, $t^2$ are said to be similar if the corresponding compatible pairs $(\tB(t^{(1)}),\Lambda(t^{(1)}))$ and $(\tB(t^{(2)}),\Lambda(t^{(2)}))$ are similar.
\end{Def}

For $i=1,2$, consider the quantum
torus
$\cT^{(i)}=\Z[q^{\pm\Hf}][X_j^\pm]_{j\in J^{(i)}}$ associated
with similar compatible pairs $(\tB^{(i)},\Lambda^{(i)})$. The sets of pointed elements are denoted by $\ptSet^{(i)}$ as before. We denote by $\pr_{\ex^{(i)}}:\Z^{J^{(i)}}\ra \Z^{\ex^{(i)}}$ the
natural projection onto the coordinates at $\ex^{(i)}$.

\begin{Def}[Similar pointed elements]
  We say the pointed elements
  $M^{(i)}=X^{\eta^{(i)}}(\sum_{v\in\N^{\ex^{(i)}}}c^{(i)}_v(q^\Hf)(Y^{(i)})^v)$,
  $c_0=1$, $c^{(i)}_v(q^\Hf)\in \Z[q^\Hf]$, in $\ptSet^{(i)}$, $i=1,2$, are similar if
  \begin{align*}
    \pr_{\ex^{(2)}} \eta^{(2)}=\var(\pr_{\ex^{(1)}}  \eta^{(1)}),\\
    c^{(2)}_v(q^\Hf)=c^{(1)}_{\var^{-1}v}(q^{\Hf\delta}),\ \forall\
    v\in\N^{\ex^{(2)}}.
  \end{align*}
\end{Def}

Assume that there is an embedding from $J^{(1)}$ to $J^{(2)}$ such that its restriction is the isomorphism $(\var^*)^{-1}:\ex^{(1)}\simeq\ex^{(2)}$. Then this embedding induces an embedding $\Z^{J^{(1)}}\ra \Z^{J^{(2)}}$ by adding $0$ on the extra coordinates, which we still denote by $\var$.
\begin{Def}[variation map]\label{def:variation_map}
 The variation map $\Var:\ptSet^{(1)}\rightarrow\ptSet^{(2)}$ between sets of pointed elements is defined such that, for any $M^{(1)}\in\ptSet^{(1)}$, its image $M^{(2)}=\Var(M^{(1)})$ is the unique pointed element in $\ptSet^{(2)}$ similar to $M^{(1)}$ such that $$\var(\eta^{(1)})=\eta^{(2)}.$$
\end{Def}

\begin{Eg}
Let us take two quantum seeds $t_0^{(1)}$, $t_0^{(2)}$ with the following two compatible pairs respectively
\begin{align*}
(B^{(1)},\Lambda^{(1)})&=(\left(
    \begin{array}{cc}
      0&-2\\
2&0\\
1&0\\
0&1
    \end{array}\right),\left(
    \begin{array}{cccc}
      0&0&-1&0\\
0&0&0&-1\\
1&0&0&2\\
0&1&-2&0
    \end{array}\right))\\
(B^{(2)},\Lambda^{(2)})&=(\left(
    \begin{array}{cc}
      0&-2\\
2&0\\
-1&2\\
0&-1
    \end{array}\right),\left(
    \begin{array}{cccc}
      0&0&-1&2\\
0&0&0&-1\\
1&0&0&2\\
-2&1&-2&0
    \end{array}\right))
\end{align*}
These compatible pairs are similar via the identification of vertices $\var^*(i)=i$, $i=1,2$. 
The following quantum cluster variables obtained by the same mutation sequence are similar:
\begin{align*}
X_2(\mu_2\mu_1 t_0^{(1)})&=X^{-e_2}(1+Y_2+(q^\Hf+q^{-\Hf})Y_1Y_2+Y_1^2Y_2)\in\cT(t^{(1)})\\
X_2(\mu_2\mu_1 t_0^{(2)})&=X^{-e_2+e_4}(1+Y_2+(q^\Hf+q^{-\Hf})Y_1Y_2+Y_1^2Y_2)\in\cT(t^{(2)} )
\end{align*} 
Notice that, we have $$\Var(X_2(\mu_2\mu_1 t_0^{(1)}))=X^{-e_2}(1+Y_2+(q^\Hf+q^{-\Hf})Y_1Y_2+Y_1^2Y_2)\in\cT(t^{(2)}),$$
which differs from $X_2(\mu_2\mu_1 t_0^{(2)})$ by a coefficient factor $X^{e_4}$.
\end{Eg}
\subsection{Correction}
We consider algebraic equations involving the pointed
elements.

Let $(\tB^{(i)},\Lambda^{(i)})$, $i=1,2$, be two compatible pairs. We use $[\ ]^{(i)}$ to denote the
normalization in $\cT^{(i)}$ and $\deg^{(i)}(\ )$
denote the leading degree of a pointed element in $\ptSet^{(i)}$.

\begin{Thm}[{Correction Technique, \cite[Thm 9.2]{Qin12}}]\label{thm:correction}
Let $M^{(1)}$, $Z^{(1)}$, $M_1^{(1)}$, $\ldots$, $M_s^{(1)}$, and (possibly infinitely many)
$Z_j^{(1)}$ be pointed elements in $\ptSet^{(1)}$. Assume that they satisfy algebraic equations
\begin{align}\label{eq:old_equation}
  \begin{split}
    M^{(1)}=&[M_1^{(1)}*M_2^{(1)}*\cdots *M_s^{(1)}]^{(1)},\\
    Z^{(1)}=&\sum_{j\geq 0}c_j(q^\Hf)Z_j^{(1)},
  \end{split}
\end{align}
 such that $c_0=1$,
$c_j(q^\Hf)\in\Z[q^{\pm\Hf}]$, and
\begin{align*}
  \deg^{(1)}Z_j^{(1)}=\deg^{(1)}Z^{(1)}+\tB^{(1)}u_j,\ \mathrm{where}\ u_j\in\N^{\ex^{(1)}}, u_0=0.
\end{align*}

Then, for any $M^{(2)},Z^{(2)},M_i^{(2)},Z_j^{(2)}$ in $\ptSet^{(2)}$ similar to
$M^{(1)},Z^{(1)},M_i^{(1)},Z_j^{(1)}$, we have the following equations:
\begin{align}\label{eq:variation_equation}
  \begin{split}
M^{(2)} =& f_M [ M_1^{(2)}* M_2^{(2)}*\cdots M_s^{(2)}]^{(2)}\\
Z^{(2)} =&\sum_{j\geq 0}c_j(q^{\Hf\delta})f_j Z_j^{(2)}
\end{split}
\end{align}
where $f_M$, $f_j$ are the Laurent monomials in the frozen variables $X_i$,
$n<i\leq m^{(2)}$, given by
\begin{align*}
f_M&=X^{\deg^{(2)} M^{(2)}- \sum_{i=1}^s\deg^{(2)} M_i^{(2)}}\\
f_j&=X^{\deg^{(2)}Z^{(2)}+\tB^{(2)}\var(u_j)-\deg^{(2)}Z_j^{(2)}}.
\end{align*}
In particular, if $M^{(2)}=\Var( M^{(1)}),Z^{(2)}=\Var (Z^{(1)}),M_i^{(2)}=\Var ({M_i^{(1)}}),Z_j^{(2)}=\Var (Z_j^{(1)})$, we have
\begin{align*}
f_M&=1,\\
f_j&=X^{\tB^{(2)}\var(u_j)-\var(\tB^{(1)}u_j)}.
\end{align*}
\end{Thm}

\begin{proof}
 We
  consider Laurent expansions of \eqref{eq:old_equation}
  \eqref{eq:variation_equation} and compare their coefficients term by
  term. Details are given in \cite[Theorem 9.2]{Qin12}.
\end{proof}


The following observations are immediate consequences of the correction
technique (Theorem \ref{thm:correction}) and the existence of (quantum) $F$-polynomials \cite[Theorem 6.1]{Tran09}\cite[Corollary 6.3]{FominZelevinsky07}\cite{gross2014canonical}.
\begin{Lem}\label{lem:variation}
Given two seeds $t,t'$ similar via an isomorphism $\var^*:\ex'\simeq \ex$ between the sets of unfrozen vertices. Let
$\overleftarrow{\mu}$ be any given sequence of
mutations on $\ex$ and $
\overleftarrow{\mu}'$ the
corresponding sequence on $\ex'$ induced by $\var^*$. Then we have the following results.

(i) The seeds $\overleftarrow{\mu}t$ and $\overleftarrow{\mu}'t'$ are similar.

(ii) The quantum cluster variable $X_i(\overleftarrow{\mu}t)$ and $X_{(\var^*)^{-1}(i)}(\overleftarrow{\mu}'t')$, $\forall i\in \ex$, viewed as pointed elements in $\ptSet(t)$ and $\ptSet(t')$ respectively, are similar.

(iii) Let $Z$ be a pointed element in $\ptSet(t)$ with a $\prec_t$-unitriangular expansion
\begin{align*}
  Z=\sum_{j\geq 0}c_jZ_j,
\end{align*}
where $Z_j$ are normalized
twisted products of quantum cluster variables of $\qClAlg(t)$.
Then, for any pointed element $Z'\in\ptSet(t')$ similar to $Z$, we have
\begin{align*}
 Z'=\sum_{j\geq 0}c_jf_j Z_j',  
\end{align*}
where $Z_j'$ are normalized products of the corresponding quantum cluster variables of $\qClAlg(t')$, the coefficients $f_j$ are
given by Theorem \ref{thm:correction}.
\end{Lem}

\section{Injective pointed sets}
\label{sec:pointed_sets}

\subsection{Injective-reachable}\label{sec:injective_reachable}
Recall that a quantum seed $t$ is a collection $(\tB(t),\Lambda(t),\underline{X}(t))$, cf. Section \ref{sec:quantum_cluster_algebra}.

\begin{Def}[Injective-reachable]

A seed $t$ of a given cluster algebra is said to be \emph{injective-reachable} via $(\Sigma,\sigma)$,
for some sequence of
mutations $\Sigma$ and a permutation $\sigma$ of $\{1,\ldots,n\}$,
if we have the following equalities
\begin{align*}
  \pr_n\deg^t(X_{\sigma(i)}(\Sigma t))&=-e_i,\ 1\leq i\leq n,\\
b_{\sigma(i)\sigma(j)}(\Sigma t)&=b_{ij}(t),\ \forall\ 1\leq
i,j\leq n.
\end{align*}

In this situation, we denote the seed $\Sigma t$ by $t[1]$.

A cluster algebra is called injective-reachable if all its
seeds are injective-reachable.
\end{Def}
Recall that we always read a mutation sequence $\overleftarrow{\mu}=\mu_{i_s}\cdots \mu_{i_2}\mu_{i_1}$
from right to left. Define $\sigma(\overleftarrow{\mu})$ to be the sequence $\mu_{\sigma i_s}\cdots \mu_{\sigma i_2}\mu_{\sigma i_1}$. For simplicity, we denote $\sigma(i)$ and $\sigma(\overleftarrow{\mu})$ by $\sigma i$ and $\sigma \overleftarrow{\mu}$ respectively, when the there is no confusion.

\begin{Prop}[\cite{Plamondon10b}\cite{gross2014canonical}]
	Different choices of the mutation sequence $\Sigma$ produce the same seed $t[1]$ up to renumbering of vertices by different $\sigma$. 
\end{Prop}
\begin{proof}
	The claim is known to be true for cluster algebras arising from quivers, thanks to the theory of cluster categories. In general, it follows from \cite{gross2014canonical}: the leading degrees of the cluster variables in $t[1]$ form the wall of a chamber, which corresponds to a seed.
	\end{proof}

\begin{Def}
For any $1\leq k\leq n$, let us denote $X_{\sigma(k)}(\Sigma t)$ by $I_k(t)$ and call it the $k$-th
\emph{injective} quantum cluster variable for $t$ (or simply the $k$-th injective for $t$).

Similarly, 
we denote $X_{\sigma^{-1}k}(\Sigma^{-1}t)$
by $P_{k}(t)$ and call it the $k$-th \emph{projective} quantum cluster
variable for $t$ (or simply the $k$-th projective for $t$)
\end{Def}

 The meaning of ``injective'' and ``projective'' will be explained in Remark
\ref{rem:injective_projective}.

We extend the action of $\sigma$ to a permutation of $\set{1,\ldots,
  m}$ trivially. Let $\sigma$ act naturally on $\Z^m$ by pull back such that $\sigma
e_i=e_{\sigma^{-1} i}$.

 We always make the following assumption for the cluster algebras studied in this paper.
\begin{Assumption*}[Injective-reachable]
The cluster algebra is injective reachable.
\end{Assumption*}
Injective-reachable Assumption does not hold for all cluster algebras, but it holds for many interesting cluster algebras including the type (i)(ii) introduced in Section \ref{sec:main_result}. 

\begin{Prop}[{\cite[Theorem 3.2.1]{Muller15}}]\label{prop:one_step_injective_mutation}
  If $t$ is injective-reachable via $(\Sigma,\sigma)$, then for any $1\leq k\leq n$, the vertex $\mu_k t$ is injective-reachable via
  $(\mu_{\sigma(k)}\Sigma\mu_k,\sigma)$.  
\end{Prop}
\begin{proof}
It is well known that the claim is true for any cluster algebra arising from a quiver, thanks to the cluster categories provided in \cite{Plamondon10a}. By the mutation of scattering diagrams in the recent work \cite{gross2014canonical}, we now know that the claim holds for all cluster algebras, cf. \cite{Muller15} for details.
\end{proof}

\begin{Lem}\label{lem:long_injective_mutation}
If $t$ is injective-reachable via $(\Sigma,\sigma)$, then for any sequence of mutations $\mu_{\uk}=\mu_{k_s}\ldots \mu_{k_2}\mu_{k_1}$, $1\leq k_1,\ldots, k_s\leq n$, the seed $\mu_\uk t$ is injective-reachable via
$(\mu_{\sigma(\uk)}\Sigma\mu_\uk^{-1},\sigma)$, where
$\mu_{\sigma(\uk)}$ is defined to be $\mu_{\sigma(k_s)}\ldots \mu_{\sigma(k_2)}\mu_{\sigma(k_1)}$.
\end{Lem}

\begin{Prop}
Assume that $t$ is injective reachable via $(\Sigma,\sigma)$. Then, for any $1\leq k\leq n$, the following claims are equivalent.

(i) $\mu_k t$ is injective reachable via $(\mu_{\sigma(k)}\Sigma\mu_k,\sigma)$. 

(ii) We have
\begin{align}\label{eq:mutated_injective_degree}
  \deg^t X_{\sigma k}(\mu_{\sigma k}\Sigma t)=-\deg^tI_k(t)+\sum_{1\leq i\leq n}[b_{ik}(t)]_+\deg^t
  I_i(t)+\sum_{0<j\leq m-n}[b_{n+j,\sigma k}(\Sigma t)]_+ e_{n+j}.
\end{align}
\end{Prop}
\begin{proof}
First assume that \eqref{eq:mutated_injective_degree} holds. Omitting the
symbol $(t)$, we can write
\begin{align*}
  \pr_n \deg^t X_{\sigma k}(\mu_{\sigma k}\Sigma t)=e_k+\sum_{1\leq i\leq n:b_{ik}\geq 0}(-b_{ik})e_i.
\end{align*}
And Claim (i) follows
from the following straightforward calculation:
\begin{align*}
\pr_n \deg^{\mu_k t}X_{\sigma k}(\mu_{\sigma k}\Sigma t)=&\pr_n\phi_{\mu_k t,t}\deg^t X_{\sigma k}(\mu_{\sigma k}\Sigma t)\\
=&\pr_n\phi_{\mu_k t,t}\pr_n\deg^t X_{\sigma k}(\mu_{\sigma k}\Sigma t)\\
=&\pr_n(\phi_{\mu_k t,t} e_k +\phi_{\mu_k t,t}\sum_{1\leq i\leq n:b_{ik}\geq 0}(-b_{ik})e_i)\\
=&\pr_n(-e_k+\sum_{1\leq i\leq m:b_{ik}\geq 0}b_{ik}e_i + \sum_{1\leq i\leq n:b_{ik}\geq 0}(-b_{ik})e_i)\\
=&-e_k.
\end{align*}

Conversely, assume that Claim (i)
is true. Because the map
$\pr_n\phi_{\mu_k t,t}$ is invertible on $\Z^n$, the last two steps of above calculation implies
\begin{align}\label{eq:restricted_mutated_injective_degree}
  \pr_n\deg^t X_{\sigma k}(\mu_{\sigma k}\Sigma t)=\pr_n(-\deg^tI_k(t)+\sum_{1\leq i\leq n}[b_{ik}]_+\deg^t I_i(t)).
\end{align}
On the other hand, consider the exchange rule in the seed $\Sigma t$:
\begin{align*}
  X_{\sigma k}(\mu_{\sigma k}\Sigma t)=\frac{\prod_{1\leq i\leq m}X_{\sigma i}(\Sigma
    t)^{[b_{\sigma i,\sigma k}(\Sigma t)]_+}+\prod_{1\leq j\leq m}X_{\sigma j}(\Sigma
    t)^{[b_{\sigma k,\sigma j}(\Sigma t)]_+}}{X_{\sigma k}(\Sigma t)}.
\end{align*}
View this equation in $\cT(t)$, by which we mean consider the Laurent expansions of all terms of the equation in the quantum torus $\cT(t)$. Because $X_{\sigma k}(\mu_{\sigma k}\Sigma t)$ is a
pointed element whose unique leading term has coefficient $1$, its
leading degree must be chosen from either $$\deg^t(\prod_{1\leq i\leq m}X_{\sigma i}(\Sigma
    t)^{[b_{\sigma i,\sigma k}(\Sigma t)]_+})-\deg^tX_{\sigma k}(\Sigma t)$$ or $$\deg^t (\prod_{1\leq j\leq m}X_{\sigma j}(\Sigma
    t)^{[b_{\sigma k,\sigma j}(\Sigma t)]_+})-\deg^t X_{\sigma k}(\Sigma
    t).$$
Their images in $\Z^n$ under the projection $\pr_n$ are different except the trivial case when $\sigma k$ is an isolated point in $Q(\Sigma t)$. \eqref{eq:restricted_mutated_injective_degree} implies that we should choose the former one. This choice gives us
    \eqref{eq:mutated_injective_degree}.
\end{proof}

\begin{Rem}
Notice that the definition of injective-reachable is not affected by quantization
or frozen
variables. 


Let us impose the principal coefficients on $t$, namely $$\tB(t)=\left(
  \begin{array}{c}
    B(t)\\
\mathrm{Id}_n
  \end{array}
\right)$$
 \cf \cite{FominZelevinsky07}.
We can calculate the leading degrees of the cluster variables appearing in \eqref{eq:mutated_injective_degree} following the mutation rule of (extended) $g$-vectors
\cite[(6.12)]{FominZelevinsky07}\cite[(3.32)]{Tran09}. Then we deduce that \eqref{eq:mutated_injective_degree} and Proposition \ref{prop:one_step_injective_mutation} are equivalent to the following property of the matrix
$\tB(\Sigma t)$:
$$b_{n+\sigma(i),\sigma(j)}(\Sigma
t)=-\delta_{ij},\ 1\leq i,j\leq n.$$
\end{Rem}

\subsection{Injective-reachable chain}
\label{sec:chain_seed}

Let $t$ be any seed which is injective-reachable via an associated
pair $(\Sigma_t, \sigma_t)$. We
denote the seed $t$ by $t[0]$.


For all $d\geq 0$, we recursively define the seeds $t[d+1]$ to be $\Sigma_{t[d]}t[d]$ where $\Sigma_{t[d]}=\sigma^d\Sigma_t$. Then $t[d]$ is injective
reachable via the pair $(\Sigma_{t[d]},\sigma_{t[d]})=(\sigma^d\Sigma_t,\sigma_t)$ by Lemma \ref{lem:long_injective_mutation}.

Let
$\ord\sigma_t$ denote the order of the permutation $\sigma_t$. So we have $\sigma_t^{\ord\sigma-1}=\sigma_t^{-1}$. Then
$t[\ord\sigma_t-1]$ is injective-reachable via
\begin{align*}
(\Sigma_{t[\ord\sigma_t-1]},\sigma_{t[\ord\sigma_t-1]})=&(\sigma_t^{\ord\sigma-1}\Sigma_t,\sigma_t)\\
=&(\sigma_t^{-1}\Sigma_t,\sigma_t).
\end{align*}
Moreover,
the principal part of the quiver of the seed $t[\ord\sigma_t]$
is the same as that of $t$. Since injective-reachable is independent
of the coefficient part, the seed
$(\sigma_t^{-1}\Sigma_t)^{-1}t$ is injective-reachable via
$(\sigma_t^{-1}\Sigma_t,\sigma_t)$. In general, for any integer $d\geq 1$, we recursively define the
the seed $t[-d]=(\sigma_t^{-d}\Sigma_t)^{-1} t[-d+1]$, which is injective-reachable
via $(\sigma_t^{-d}\Sigma_t,\sigma_t)$.

Notice that we have $(\sigma_t^{-d}\Sigma_t)^{-1}=\sigma_t^{-d}(\Sigma_t^{-1})$ by the action of $\sigma$ on mutation sequences, which is also denoted by $\sigma_t^{-d}\Sigma_t^{-1}$.

\begin{Def}[Injective-reachable chain]
	The chain of seeds $(t[d])_{d\in \Z}$ is called an
	injective-reachable chain.
\end{Def}

\subsection{Cluster expansions of injective and projectives}

We
make the following assumption on the Laurent expansions of $I_i(t)$ and $P_{\sigma i}(\Sigma
t)$ for the rest of this paper. Its meaning will be clear in Remark \ref{rem:injective_projective}.

Recall that $\sigma^{-1} e_i=e_{\sigma i}$, $1\leq i\leq m$. The following assumption holds for all cluster algebras arising from quivers, \cf Remark \ref{rem:injective_projective}. We refer the reader to Example \ref{eg:compare_inj_proj} for an explicit example.
\begin{Assumption*}[Cluster Expansion]
Given any seed $t$ and vertex $1\leq i\leq n$. When $q^\Hf$ specializes to $1$, the Laurent expansion of
$I_i(t)|_{q^\Hf \mapsto 1}$ in $\cT(t)$ takes the following form:
\begin{align}\label{eq:expand_injective}
  \begin{split}
    I_i(t)|_{q^\Hf \mapsto 1}&=X_{\sigma{i}}(\Sigma t)|_{q^\Hf \mapsto 1}\\
&=X(t)^{\deg^t
      I_i(t)}(1+Y(t)^{e_i}+\sum_{d\in\N^n : d>e_i}\alpha_d
    Y(t)^{d}),
  \end{split}
\end{align}
and the Laurent expansion of
$P_i(\Sigma t)|_{q^\Hf \mapsto 1}$ in $\cT(\Sigma t)$ takes the following form:
\begin{align}\label{eq:expand_projective}
  \begin{split}
    P_{i}(\Sigma t)|_{q^\Hf \mapsto 1}=& X_{\sigma^{-1}i}(t)|_{q^\Hf \mapsto 1}\\
    =&X(\Sigma t)^{\deg^{\Sigma t}P_i(\Sigma t)}Y(\Sigma t)^{p(i,\Sigma
      t)}\\&\qquad\cdot(\sum_{d\in\N^n:e_{i}<d\leq
      p(i,\Sigma t)}\beta_d Y(\Sigma t)^{-d}+Y(\Sigma t)^{-e_{i}}+1),
  \end{split}
\end{align} 
for some dimension vector $p(i, \Sigma t)\in\N^n$, coefficients $\alpha_d,\beta_d\in\Z$.

Moreover, the degrees satisfy
\begin{align}\label{eq:index_coindex}
  \deg^{\Sigma t}(P_i( \Sigma t)Y(\Sigma t)^{p( i,\Sigma
    t)})=\sigma^{-1} \deg^tI_{\sigma^{-1}i}(t).
\end{align}
\end{Assumption*}

\begin{Rem}[Injectives and projectives]\label{rem:injective_projective}
Assume that the cluster algebra arises from a quiver, in other words, the matrix $\tB(t)$ is skew-symmetric. We explain the Cluster Expansion Assumption by the theory of cluster categories. An example will be given in Example \ref{eg:compare_inj_proj}.

We use the cluster category $\cC$ in Section \ref{sec:ice_quiver} to categorify the cluster algebra. Then the quantum cluster variable $X_i(t)$ correspond to an indecomposable rigid object $T_i$ for any $1\leq i\leq m$.

On the one hand, the object $\Sigma T_k$, $1\leq k\leq n$, corresponds
to an injective module: $\Ext^1_{\cU/(\Gamma_F)}(\oplus_{1\leq i\leq n} T_i,\Sigma T_k)$ is the $k$-th injective right module of $\End_{\cU/(\Gamma_F)}(\oplus_{1\leq i\leq n}T_i)$, \cf Section \ref{sec:CY_reduction} and \cite[Section 3.1]{Plamondon10a}. On the other hand, because the object $\Sigma T_k$ has index $-e_k$ in the quotient cluster category $\cU/(\Gamma_F)$, it corresponds to the quantum cluster variable $I_k(t)=X_{\sigma k}(t[1])$. Furthermore, the object $\Sigma T_k$ corresponds to the $\sigma k$-th cluster variable in the seed $t[1]$ for some permutation $\sigma$.

Similarly, the object $T_k$ corresponds to a projective module: $\Ext^1(\oplus_i \Sigma T_i,T_k)$ is the indecomposable projective right module of $\End_{\cU/(\Gamma_F)}(\oplus_{1\leq i\leq n}\Sigma T_i)$ on the vertex corresponding to $\Sigma T_k$, which is labeled as $\sigma k$. Let us denote its dimension vector by $$p(\sigma k,t[1])=\dim \Ext^1_{\cU/(\Gamma_F)}(\Sigma T,T_k)(=\dim \Ext^1_{\cC}(\Sigma T,T_k)).$$

Now \eqref{eq:expand_injective} and \eqref{eq:expand_projective} are just the cluster characters of injective and projective modules respectively (\cite{DerksenWeymanZelevinsky09}).

To see \eqref{eq:index_coindex}, we
rewrite \eqref{eq:quiver_coindex} as
\begin{align*}
  \Ind^{\Sigma T}T_k+\tB(\Sigma T)\cdot \Ext_\cC(\Sigma T,T_k)=\Ind^T \Sigma T_k,
\end{align*}
where the indices count the multiplicities of the summands of $\Sigma T$ and $T$ on each side. Noticing that $\Sigma T_i$ corresponds to the $\sigma i$-th cluster variable in $t[1]$, $T_i$ corresponds to the $i$-th cluster variable in $t$, and $\sigma e_{\sigma i}=e_i$, we rewrite the above equation as
\begin{align*}
\sigma(  \deg^{t[1]}P_{\sigma k}+\tB(t[1])\cdot p(\sigma k,t[1]))=\deg^t I_k(t).
\end{align*}
This give us \eqref{eq:index_coindex} by taking $k=\sigma^{-1}i$.
\end{Rem}

\begin{Eg}[Compare injectives and projectives]\label{eg:compare_inj_proj}
Consider a seed $t$ whose ice quiver $\tQ(t)$ is given in
Figure \ref{fig:A_3_principal_injective}. It is of principal coefficients
with the canonical quantization matrix
\begin{align*}
  \Lambda(t)=\left(
    \begin{array}{cccccccc}
      0&0&0&0&-1&0&0&0\\
0&0&0&0&0&-1&0&0\\
0&0&0&0&0&0&-1&0\\
0&0&0&0&0&0&0&-1\\
1&0&0&0&0&1&0&-1\\
0&1&0&0&-1&0&1&1\\
0&0&1&0&0&-1&0&0\\
0&0&0&1&1&-1&0&0\\
    \end{array}
\right).
\end{align*}

Take the sequence $\Sigma=\mu_1\mu_3\mu_2\mu_4\mu_1$ and the permutation $$\sigma=\left(
  \begin{array}{cccc}
    1&2&3&4\\
4&2&3&1
  \end{array}
\right).$$

The seed $t[-1]$ is obtained from $t$ by applying
$\sigma^{-1}\Sigma^{-1}=\mu_4\mu_1\mu_2\mu_3\mu_4$. Its ice quiver is given in Figure
\ref{fig:A_3_principal_projective}. $t[-1]$ is injective-reachable via $(\sigma^{-1}\Sigma,\sigma)=(\mu_4\mu_3\mu_2\mu_1\mu_4,\sigma)$.

In the quantum torus $\cT(t[-1])$, we have
  \begin{align*}
I_1(t[-1])=X_4(t)=&X^{-e_1+e_5+e_8}(1 + Y_1 + Y_1Y_2),\\
I_2(t[-1])=X_2(t)=&X^{-e_2+e_6+e_7+e_8}(1 + Y_2 + Y_2Y_4),\\
I_3(t[-1])=X_3(t)=&X^{-e_3+e_7}(1 + Y_3 + Y_2Y_3 + Y_2Y_3Y_4),\\
I_4(t[-1])=X_1(t)=&X^{-e_4+e_5+e_6+e_7}(1 + Y_4 + Y_1Y_4).
  \end{align*}

It is straightforward to check that the projective right modules of the
Jacobian algebra of $(Q(t),W(t))$ have dimensions
\begin{align*}
p(1, t)=&e_1+e_2+e_3,\\
p(2, t)=&e_2+e_3+e_4,\\
  p(3, t)=&e_3,\\
p(4, t)=&e_1+e_4.
\end{align*}
The
$Y$-variables of the seed $t$ in the quantum torus $\cT(t)$ are given by
\begin{align*}
  Y_1=&X^{e_2-e_4+e_5}\\
Y_2=&X^{-e_1+e_3+e_4+e_6}\\
Y_3=&X^{-e_2+e_7}\\
Y_4=&X^{e_1-e_2+e_8}.
\end{align*}
The projectives for $t$, obtained by applying the sequence
$(\sigma\Sigma)^{-1}$, take the following form in the quantum torus
$\cT(t)$:
\begin{align*}
P_1( t)=X_4(t[-1])=&X^{-e_3}(1 + Y_3 + Y_2Y_3 + Y_1Y_2Y_3)\\
=&X^{-e_1+e_5+e_6+e_7}(Y^{-e_1-e_2-e_3}+Y^{-e_1-e_2}+Y^{-e_1}+1),\\
P_2( t)=X_2(t[-1])=&X^{e_2-e_3-e_4}(1 + Y_3 + Y_4+Y_3Y_4 + Y_2Y_3Y_4)\\
=&X^{-e_2+e_6+e_7+e_8}(Y^{-e_2-e_3-e_4}+Y^{-e_2-e_4}+Y^{-e_2-e_3}+Y^{-e_2}+1),\\
P_3( t)=X_3(t[-1])=&X^{e_2-e_3}(1 + Y_3)\\
=&X^{-e_3+e_7}(Y_3^{-1}+1),\\
P_4( t)=X_1(t[-1])=&X^{-e_1}(1 + Y_1 + Y_1Y_4)\\
=&X^{-e_4+e_5+e_8}(Y^{-e_1-e_4}+Y^{-e_4}+1).
\end{align*}
Cluster Expansion Assumption holds for this example. Notice that, this example is an accidental case where the Laurent expansions of quantum cluster variables do not involve $q$-coefficients.
\end{Eg}

\begin{figure}[htb!]
 \centering
\beginpgfgraphicnamed{fig:A_3_principal_injective}
  \begin{tikzpicture}
\node [shape=circle, draw] (v1) at (7,4) {1}; 
    \node  [shape=circle, draw] (v2) at (6,3) {2}; 
    \node [shape=circle,  draw] (v3) at (5,2) {3};

\node  [shape=circle, draw] (v4) at (5,4) {4};

    \node [shape=diamond,  draw] (v8) at (3,4) {8};
\node  [shape=diamond, draw] (v6) at (4,3) {6};
    \node  [shape=diamond, draw] (v7) at (3,2) {7}; 
    \node [shape=diamond,  draw] (v5) at (9,4) {5};

    \draw[-triangle 60] (v3) edge (v2); 
    \draw[-triangle 60] (v2) edge (v1);
     \draw[-triangle 60] (v4) edge (v2);
    \draw[-triangle 60] (v1) edge (v4);

    \draw[-triangle 60] (v5) edge (v1); 
    \draw[-triangle 60] (v6) edge (v2); 
    \draw[-triangle 60] (v7) edge (v3); 
    \draw[-triangle 60] (v8) edge (v4); 
     
  \end{tikzpicture}
\endpgfgraphicnamed
\caption{An ice quiver $\tQ(t)$ of principal coefficients}
\label{fig:A_3_principal_injective}
\end{figure}
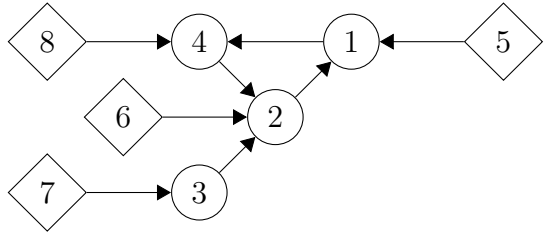

\begin{figure}[htb!]
 \centering
\beginpgfgraphicnamed{fig:A_3_principal_projective}
\begin{tikzpicture}

\node [shape=circle, draw] (v1) at (7,4) {4}; 
    \node  [shape=circle, draw] (v2) at (6,3) {2}; 
    \node [shape=circle,  draw] (v3) at (5,2) {3};

\node  [shape=circle, draw] (v4) at (5,4) {1};

    \node [shape=diamond,  draw] (v8) at (3,4) {8};
\node  [shape=diamond, draw] (v6) at (8,2.5) {6};
    \node  [shape=diamond, draw] (v7) at (3,2) {7}; 
    \node [shape=diamond,  draw] (v5) at (9,3) {5};

    \draw[-triangle 60] (v3) edge (v2); 
    \draw[-triangle 60] (v2) edge (v1);
     \draw[-triangle 60] (v4) edge (v2);
    \draw[-triangle 60] (v1) edge (v4);

    \draw[-triangle 60] (v1) edge (v6);
    \draw[-triangle 60] (v1) edge (v7);
 
    \draw[-triangle 60] (v4) edge (v5);
 
    \draw[-triangle 60] (v2) edge (v8); 
    \draw[-triangle 60] (v8) edge (v3); 

    \draw[-triangle 60] (v6) edge (v3); 
     
\end{tikzpicture}\endpgfgraphicnamed
\caption{The quiver $\tQ(t[-1])$ obtained from Figure {\ref{fig:A_3_principal_injective}}
  after mutations at $(4, 1, 2, 3, 4)$}
\label{fig:A_3_principal_projective}
\end{figure}
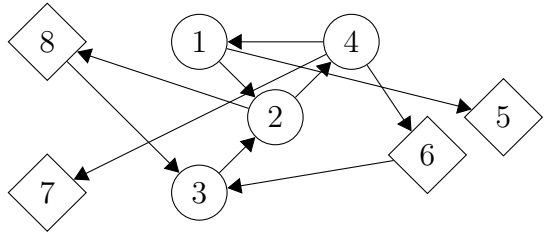

\subsection{Injective pointed set}\label{sec:injective_pointed}

For any $g\in\Z^n$, denote $g_+=([g_i]_+)_{1\leq i\leq n}$
and $g_-=([-g_i]_+)_{1\leq i\leq n}$. We will also view $g$,
$g_\pm\in\Z^n$ as elements of $\Z^m$ by putting $0$ at the last $m-n$ components.

For any $(f,d_X,d_I)\in \Z^{m-n}\oplus\N^n\oplus \N^n$, we define the
following pointed element of $\cT(t)$
\begin{align}
  \inj^t(f,d_X,d_I)=[\prod_{1\leq i\leq m-n }X_{n+i}^{f_i}*X(t)^{d_X}*I(t)^{d_I}]^t,
\end{align}
where $I(t)^{d_I}$ denotes $[\prod_{1\leq k\leq n} I_k(t)^{(d_I)_k}]^t$.

Notice that $\deg^t I_k(t)=-e_k+f^{(k)}$ for some $f^{(k)}\in\Z^{\{n+1,\ldots,m\}}$.

\begin{Def}[Injective pointed set]\label{def:injective_pointed_element}
For any $\tg\in\degL(t)$, put $g=\pr_n \tg\in\Z^n$ and $g_f=\tg-g-\sum_{1\leq k\leq n}[-g_k]_+\cdot f^{(k)}\in\Z^{\{n+1,\ldots,m\}}$. Define the pointed element $\inj^t(\tg)=\inj^t(g_f,g_+,g_-)$.

The injective pointed set $\inj^t$ is the set $\{\inj^t(\tg)|\tg\in\degL(t)\}$.
\end{Def}

By definition, $\inj^t$ is a $\degL(t)$-pointed $\Z[q^{\pm\Hf}]$-linearly independent subset of
$\cT(t)$. The following lemma suggests that its elements behave well under one-step mutations.
\begin{Lem}[Neighboring injective pointed set]\label{lem:neighbouring_pointed}
For any given seed $t$ and $1\leq k\leq n$, denote $t'=\mu_kt$. Then the set $\inj^{t'}$ is
$\degL(t)$-pointed in $\cT(t)$ with respect to the order
$\prec_{t}$. Furthermore, for any $\tg\in\degL(t')$, we have
\begin{align*}
\deg^t  \inj^{t'}(\tg)=\phi_{t,t'}(\tg).
\end{align*}
\end{Lem}
\begin{proof}
The degrees in $\degL(t')$ change by the piecewise linear formula
\eqref{eq:degree_mutation} which depends on the sign of the $k$-th
components. In our situation, the degrees of the factors of $\inj^{t'}(\tg)$
are $\deg^{t'}X(t')^{g_f}$, $\deg^{t'}X(t')^{g_+}$,
$\deg^{t'}I(t')^{g_-}$. They are sign coherent at the $k$-th
component. Therefore, we can use Lemma \ref{lem:additivity} and deduce that
\begin{align*}
\phi_{t,t'}(\tg)=&\phi_{t,t'}\deg^{t'}(X(t')^{g_f}*X(t')^{g_+}*I(t')^{g_-})\\
=&\phi_{t,t'}\deg^{t'}(X(t')^{g_f})+\phi_{t,t'}\deg^{t'}(X(t')^{g_+})+\phi_{t,t'}\deg^{t'}(I(t')^{g_-})\\
=&\deg^t(X(t')^{g_f})+\deg^t(X(t')^{g_+})+\deg^t(I(t')^{g_-})\ \ (\mathrm{Theorem\ \ref{thm:change_index_compatible}})\\
=&\deg^t(X(t')^{g_f}*X(t')^{g_+}*I(t')^{g_-})\\
=&\deg^t \inj^{t'}(\tg).
\end{align*}

It remains to check that the coefficient of the leading term remains
to be $1$. We have to verify the following equality:
\begin{align*}
\Lambda(t)(\phi_{t, t'}\deg^{t'}X(t')^{g_f},\phi_{t,t'}X(t')^{g_+})=&\Lambda(t')(\deg^{t'}X(t')^{g_f},\deg^{t'}X(t')^{g_+}),\\
\Lambda(t)(\phi_{t,t'}\deg^{t'}X(t')^{g_f},\phi_{t,t'}I(t')^{g_-})=&\Lambda(t')(\deg^{t'}X(t')^{g_f},\deg^{t'}I(t')^{g_-}),\\
\Lambda(t)(\phi_{t,t'}\deg^{t'}X(t')^{g_+},\phi_{t,t'}I(t')^{g_-})=&\Lambda(t')(\deg^{t'}X(t')^{g_+},\deg^{t'}I(t')^{g_-})
\end{align*}
Decompose the degrees appearing above into sum of
unit vectors by Lemma \ref{lem:additivity}. Then it suffices to verify the
following equations for any pair of unit vectors $(e_i,e_j)$, $1\leq i,j\leq m$,
sign-coherent at the $k$-th component:
\begin{align*}
\Lambda(t)(\phi_{t,t'}e_i,\phi_{t,t'}(\pm e_j))=&\Lambda(t')(e_i,\pm e_j),\\
\end{align*}
These equations follow from the mutation rule of quantization
matrices, $\forall i,j\neq k$:
\begin{align*}
  \Lambda(t)_{ij}&=\Lambda(t')_{ij},\\
\Lambda(t)(e_i,-e_k+[-b_{jk}(t')]_+e_j)&=e_i^T(E_{-1}(t')^{T}\Lambda(t')E_{-1}(t'))E_{-1}(t')e_k\\
&=e_i^T\Lambda(t')e_k\\
\Lambda(t)(e_i,e_k-[-b_{jk}(t')]_+e_j)&=e_i^T(E_{1}(t')^{T}\Lambda(t')E_{1}(t'))E_{1}(t')(-e_k)\\
&=e_i^T\Lambda(t')(-e_k).
\end{align*}
\end{proof}

\begin{Rem}
  The notion of leading terms depends on the choice of seed and so does
  the normalization factor. Different pointed elements
  $Z_i\in\cT(t^1)$ might have identical leading
  degrees in the quantum torus $\cT(t^2)$ of the new seed $t_2$, and a
  pointed element $Z$ in $\cT(t^1)$ might not
  be pointed in the $\cT(t^2)$. The normalization factor of the twisted
  product $Z_1*Z_2$
  might also change if they are not quantum cluster variables in the same cluster, cf. Example \ref{eg:A_2}.
\end{Rem}

\begin{Eg}[Injective pointed sets]\label{eg:A_2}
  We consider the initial acyclic quiver $Q=Q(t_0)$ of type $A_2$ in Figure
  \ref{fig:A_2}. Then $\tB(t_0)=B=
  \left(\begin{array}{cc}
    0&-1\\
1&0
  \end{array}\right).
$ Take the quantization matrix $\Lambda(t_0)$ to be $-B^{-1}=B$.

Then $t_0$ is injective-reachable via
$(\Sigma,\sigma)=(\mu_2\mu_1,\id)$. We have 
\begin{align*}
X_i(t_0)&=X_i,\ i=1,2,\\
I_1(t_0)&=X_1^{-1} + X_1^{-1}X_2=X^{-e_1}(1+Y_1)\\
  I_2(t_0)&=X_1^{-1}X_2^{-1} + X_2^{-1} + X_1^{-1}=X^{-e_2}(1+Y_2+Y_1Y_2)
\end{align*}
The set $\inj^{t_0}$ consists of the normalized ordered products of $X_i(t_0)$ and $I_j(t_0)$
such that $X_i(t_0)$ and $I_i(t_0)$ do not appear simultaneously. It is apparently $\degL(t_0)=\Z^2$ pointed. 

Consider the neighboring seed $\mu_1 t_0$, its quantization matrix is
$\Lambda(\mu_1 t_0)=-B$. In the
quantum torus $\cT(\mu_1 t_0)$, we have the following Laurent expansions
\begin{align*}
  X_1(t_0)&=X^{-e_1+e_2}(1+Y_1)\\
X_2(t_0)&=X_2\\
I_1(t_0)&=X_1\\
I_2(t_0)&=X^{-e_2}(1+Y_2).
\end{align*}
It is straightforward to check that $\inj^{t_0}$ is a $\degL(\mu_1 t_0)$-pointed set.

If we proceed to the seed $\mu_2\mu_1t_0=t_0[1]$. The quantization
matrix $\Lambda(t_0[1])=B$. In the
quantum torus $\cT(t_0[1])$, we have the following Laurent expansions
\begin{align*}
X_1(t_0)&=P_1(t_0[1])=X^{-e_2}(1+Y_2+Y_1Y_2)\\
X_2(t_0)&=P_2(t_0[1])=X^{e_1-e_2}(1+Y_2)\\
I_1(t_0)&=X_1\\
I_2(t_0)&=X_2.
\end{align*}

It is straightforward to check that the set of leading degrees of
$\inj^{t_0}$ in this seed is $\N\times \Z$. In fact, the different
pointed elements $[X_1(t_0)*I_2(t_0)]^{t_0}$ and $1$
in $\cT(t_0)$ will have the same leading degree
$0$ in $\cT(t_0[1])$. The normalization factors of $X_1*I_2$ are also
different in $\cT(t_0)$ and $\cT(t_0[1])$.
\end{Eg}

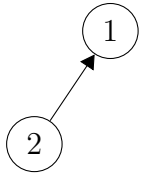
\begin{figure}[htb!]
 \centering
\beginpgfgraphicnamed{fig:A_21}
\begin{tikzpicture}
\node [shape=circle, draw] (v1) at (2,4) {1};
    \node [shape=circle, draw] (v2) at (1,2.5) {2};
    \draw[-triangle 60] (v2) edge (v1); 
\end{tikzpicture}
\endpgfgraphicnamed
\caption{A quiver $Q$ of type $A_2$.}
\label{fig:A_2}
\end{figure}

\section{Triangular bases}
\label{sec:triangular_basis}

In this section, we introduce the key tool of this paper: the common triangular basis of a quantum cluster algebra.

\subsection{Common triangular bases in brief}
\label{sec:triangular_basis_brief}
In this subsection, we define the common triangular bases and present theorems that guarantee their existence. We give a brief strategy of how we will treat them for the rest of this paper.

\subsubsection*{Definitions}

Let $\qClAlg$ be a given injective-reachable quantum cluster algebras. We will assume that Cluster Expansion Assumption is satisfied, which is true for cluster algebras arising from quivers.

Let $t$ be any given seed of $\qClAlg$. Recall that we have defined the dominance order (partial order) $\prec_t$ on the degree lattice $\degL(t)$. Also, we have the degree map $\deg^t$ from the set of pointed elements $\ptSet(t)$ to $\degL(t)$. In addition, the normalization operator $[\ ]^t$ gives $[q^s Z]^t=Z$ for any pointed element $Z\in\ptSet(t)$, $s\in\Hf\Z$.

Recall that the injectives $I_k(t)$, $1\leq k\leq n$, are the quantum cluster variables whose leading degrees $\deg^t I_k(t)$ take the form $-e_k+f^{(k)}$ for some $f^{(k)}\in\Z^{\{n+1,\ldots,m\}}$. They are quantum cluster variables of the seed $t[1]$.

\begin{Def}[Triangular basis]\label{def:triangularBasis}
Let $t$ be any seed of a given quantum cluster algebra $\qClAlg$. A $\Z[q^{\pm\Hf}]$-basis of $\qClAlg$ is called a triangular basis with respect to $t$, if it satisfies the following properties.
\begin{itemize}
\item It contains the quantum cluster monomials in the seeds $t$ and $t[1]$.
\item (bar-invariance) All of its elements are invariant under the bar involution.
\item (parameterization) It is $\degL(t)$-pointed: it is in bijection with $\degL(t)$ via the degree map $\deg^t$.
\item (triangularity)  For any quantum cluster variable $X_i(t)$ in the seed $t$ and any basis element $S$, the normalized twisted product $[X_i(t)*S]^t$ is $(\prec_t,\bm)$-unitriangular to the basis:
$$
[X_i(t)*S]^t=b+\sum_{\deg^t b'\prec_t \deg^t b}c_{b'}b',
$$
where $c_{b'}\in \bm$, $\deg^t b=\deg^t X_i(t)+\deg^t S$, and $b,b'$ are basis elements.
\end{itemize}
\end{Def}

We can lessen the triangularity condition in the definition of a triangular basis and obtain the notion of a weakly triangular basis (Definition \ref{def:weakly_triangular_basis}). A weakly triangular basis \wrt a seed $t$, if exists, is unique (Lemma \ref{lem:basis_property}(i)), which we denote by $\can^t$. A triangular basis is the weakly triangular basis (Lemma \ref{lem:basis_property}(ii)) and, correspondingly, we still use $\can^t$ to denote it and indicate that it is the triangular basis.

We will expect the basis to have the following property.
\begin{Def}[Positive basis]\label{def:positive_basis}
	A basis of a $\Z[q^{\pm\Hf}]$-algebra said to be positive, if its structure constants with respect to multiplication belong to $\N[q^{\pm\Hf}]$.
\end{Def}

Notice that Definition \ref{def:triangularBasis} depends on the seed $t$ chosen. The parametrization is demanded by Fock-Goncharov conjecture (Conjecture \ref{conj:FG_conj}), and the triangularity is similar to Leclerc's conjecture (Remark \ref{rem:Leclerc_conjecture}). By the following definition, we expect the parametrization to be compatible \wrt different seeds (Definition \ref{def:compatible_pointed_set}).

\begin{Def}[Common triangular basis]\label{def:common_triangular_basis}
	Let $T$ be a set of seeds. A basis of the quantum cluster algebra $\qClAlg$ is called the common triangular basis \wrt $T$, if it is the triangular basis $\can^t$ \wrt each $t\in V$ and, moreover, $\can^t$, $\can^{t'}$ are compatible for any $t,t'\in T$.
	
	When $T$ consists of all seeds of $\qClAlg$, the corresponding common triangular basis is called the common triangular basis of $\qClAlg$.
\end{Def}

It follows from its definition that the common triangular basis, if exists, contains all the quantum cluster monomials (Lemma \ref{lem:basis_property}(iii)).

\begin{Lem}\label{lem:FG_conj}
	If the common triangular basis of $\qClAlg$ exists, then it implies the Fock-Goncharov basis conjecture (Conjecture \ref{conj:FG_conj}). 
\end{Lem}

\subsubsection*{Existence Theorems}

We now present various theorems guaranteeing the existence of the common triangular basis, whose proofs will be postponed to the end of this section.

\begin{Thm}[Existence Theorem]\label{thm:induction}
	Let $\qClAlg$ be an injective-reachable quantum cluster algebra which satisfies Cluster Expansion Assumption. Let $(t[d])_{d\in \Z}$ be an injective-reachable
	chain of seeds. Assume that the common triangular basis with respect to $\{t[d]\}$ exists and is positive. Then, the common triangular basis of $\qClAlg$ exists.
\end{Thm}

The correction technique (Theorem \ref{thm:correction}) allows us to reduce the conditions in Existence theorem to a finite criterion.
\begin{Thm}[Reduced Existence Theorem]\label{thm:reduction}
	Let $\qClAlg$ be an injective-reachable quantum cluster algebra which satisfies Cluster Expansion Assumption. Let $t$ be a seed
	injective-reachable via $(\Sigma, \sigma)$. Assume that the triangular basis $\can^t$ exists, is positive, and 
	contains the quantum cluster variables
	obtained along the mutation sequences $\Sigma$ from $t$ to $t[1]$ and along the sequence $\sigma^{-1}\Sigma^{-1}$ from $t$ to $t[-1]$. Then the common triangular basis of $\qClAlg$ exists.
\end{Thm}

As a consequence of the correction technique (Theorem \ref{thm:correction}), a change of the coefficient pattern or quantization will not affect the existence of the common triangular basis.
\begin{Thm}\label{thm:change_coefficient}
	Let $t$, $t'$ be two similar seeds in the sense of Section \ref{sec:correction} and $\clAlg(t)$,
	$\clAlg(t')$ the corresponding quantum cluster algebras respectively. If the common triangular basis of $\qClAlg(t)$ exists, then the common triangular basis of $\qClAlg(t')$ exists as well.
\end{Thm}

\subsubsection*{Strategy}

The construction of the (common) triangular basis is not clear from the definition, because Lusztig's Lemma does not apply. This situation is very different from that of Berenstein-Zelevinsky's triangular basis \cite{BerensteinZelevinsky12}. In applications (Section \ref{sec:application}), we will adopt the following approach:

\begin{enumerate}
\item Start with some known basis $\can$ of $\qClAlg$, prove that it is the triangular basis \wrt some chosen initial seed.
\item  Show that this basis is the common triangular basis \wrt all seeds $t$ by induction on $t$.
\end{enumerate}

Most arguments in Section \ref{sec:triangular_basis} are not difficult. When verifying the $\bm$-unitriangularity, we do not rely on sensitive calculation of $q$-powers, but use the following basic properties:
\begin{enumerate}
\item $\bm$-unitriangularity is preserved under composition and under taking inverse (Lemma \ref{lem:inverse_transition}).
\item By Lemma \ref{lem:preserve_triangular}, if we have a unitriangular decomposition of pointed Laurent polynomials $Z,Z_j$:
$$Z=Z_0+\sum_{j\geq 1:\deg Z_j\prec \deg Z_0=\deg Z} b_j Z_j,\ b_j\in\bm,$$
then it remains unitriangular after left multiplying $X_i$:
\begin{align}\label{eq:triangle_preservation}
[X_i*Z]&=[X_i*Z_0]+\sum b_j q^{\alpha_j}[X_i*Z_j],\ \alpha_j\leq 0,.
\end{align}

\end{enumerate}
The right hand side can be further decomposed by the triangularity property of the triangular basis in Definition \ref{def:triangularBasis}.

Finally, we point out the difficult part in the inductive verification of the common triangular basis: after one step mutation $\mu_k$, prove that the new cluster variable $X_k(\mu_k t)$ belongs to the triangular basis $\can^t$ (Proposition \ref{prop:prove_X}). Our proof can be sketched as follows.
\begin{enumerate}
\item Assume that $\can^t$ equals the triangular basis $\can^{t[-1]}$ as sets. We consider the desired basis element $b$ whose leading degree in $\cT(t[-1])$ equals $\deg^{t[-1]}X_k(\mu_k t)$. It must appear as the leading term in some product of elements of $\can^{t[-1]}$:
\begin{align*}
q^{-\alpha} b_1*b_2=b+(\prec_{t[-1]}-\mathrm{lower}\ \mathrm{terms}),
\end{align*}
where $\alpha\in\frac{\Z}{2}$, $b_1,b_2\in\can^{t[-1]}=\can^t$ with $$\deg^{t[-1]}b_1+\deg^{t[-1]} b_2=\deg^{t[-1]}b.$$

\item Assume compatibility of degrees: $\deg^t b=\phi_{t,t[-1]}\deg^{t[-1]}b$. Show that, with respect to $t$, the above equation becomes
\begin{align*}
q^{-\alpha}b_1*b_2=b+(\prec_{t}-\mathrm{higher}\ \mathrm{terms}).
\end{align*} 

\item
Cluster Expansion Assumption gives the lowest several Laurent monomials of the product $b_1*b_2$ in $\cT(t)$. Then, by assuming the positivity of $\can^t$, we can determine the Laurent expansion of the lowest term $b$, which equals that of $X_k(\mu_k t)$ in $\cT(t)$.
\end{enumerate}
This proof is based on good degree parametrization of basis elements \wrt different seeds $t$, $t[-1]$, which is expected by Fock-Goncharov conjecture (Conjecture \ref{conj:FG_conj}).

\subsection{Basic properties}\label{sec:basic_property}

\begin{Lem}\label{lem:triangular_basis_property}
	Let $\can^t$ be the triangular basis \wrt a given seed $t$. Then the following claims hold.
\begin{enumerate}[(i)]
	\item $\can^t$ factors through the frozen variables, namely, we have
	 \begin{align}
	 \label{eq:factorization_basis}
	 [X_{n+i}(t)^{f_i}*\can^t(\tg)]^t\in\can^t,\ \forall 1\leq i\leq m-n,\tg\in\degL(t),f_i\in\Z.
	 \end{align}
	 \item We have $\can^t(0)=1$.
	 \item The injective pointed set $\inj^t$ is $(\prec_t,\bm)$-unitriangular to $\can^t$.
	\end{enumerate}
\end{Lem}
\begin{proof}
	Omit the symbol $t$ for simplicity.
	
(i)	First assume $f_i\geq 0$. By the triangularity of $\can$, the normalized product $[X_{n+i}^{f_i}*\can(\tg)]$ has a $(\prec,\bm)$-unitriangular decomposition into $\can$. But $[X_{n+i}^{f_i}*\can(\tg)]$ itself is bar-invariant. It follows that $[X_{n+i}^{f_i}*\can(\tg)]=\can(\tg+f_i e_{n+i})$ when $f_i\geq 0$. Finally, for any $f_i<0$, since $[X_{n+i}^{-f_i}*\can(\tg+f_i e_{n+i})]=\can(\tg)$, we have $[X_{n+i}^{f_i}*\can(\tg)]=[X_{n+i}^{f_i}*[X_{n+i}^{-f_i}*\can(\tg+f_i e_{n+i})]]=\can(\tg+f_i e_{n+i})$.
	
	(ii) Since $1$ is trivially a cluster monomial in the seed $t$, the claim follows from definition of triangular basis.
	
(iii) The claim follows from the triangularity of the triangular basis. 
\end{proof}

\begin{Rem}
	\eqref{eq:factorization_basis} is an analogue of the factorization
	property of the dual canonical bases of quantum unipotent subgroups, cf. \cite[Section 6.3]{Kimura10}.
	
	Claim (iii) is an analog of the transition property between
	dual PBW bases and dual canonical bases. If the seed $t$ has acyclic principal part with appropriate
	coefficient pattern, up to localization at frozen variables,  $\inj^t(f,d_X,d_I)$ are dual PBW basis elements, and we can
	take $\can^t$ to be the dual canonical basis. \cite{KimuraQin11} proved that such $\can^t$ contains all quantum cluster monomials by. In fact, this phenomenon motivates our definition of weakly triangular bases (Definition \ref{def:weakly_triangular_basis}).
\end{Rem}


The following important lemma suggests that the $(\prec_t,\bm)$-unitriangularity of an expansion is preserved under multiplication by $X_k$ from the left or by $I_k$ from the right.
\begin{Lem}[Triangularity Preservation]\label{lem:preserve_triangular}
For any $\tg'\preceq_t\tg\in\degL(t)$, $1\leq i\leq m$, $1\leq k\leq n$, we have
\begin{align}
\Lambda(t)(\deg^t X_i(t),\deg^t Y_k(t))=\Lambda(t)(\deg^t Y_k(t),\deg^t I_i(t))=-\delta_{ik}\mathbf{d}_i,\\
\Lambda(t)(\deg^t X_k(t),\tg')\leq \Lambda(t)(\deg^t X_k(t),\tg),  \\
\Lambda(t)(\tg',\deg^t I_k(t))\leq\Lambda(t) (\tg,\deg^t I_k(t)).
\end{align}
\end{Lem}
\begin{proof}
The first equation follows from the definition of compatible
pairs, and it implies the remaining equations.
\end{proof}

The following lemma is a direct consequence of Lemma \ref{lem:preserve_triangular}. Notice that its assumption is satisfied when the triangular basis $\can^t$ exists.
\begin{Lem}[Substitution]\label{lem:substitute_expansion}
	Let $t$ be any given seed. Assume that $\inj^t(f,d_X,d_I)$ is unitriangular to $\inj^t$, $\forall f\in\Z^{m-n}$, $d_X,d_I\in\N^n$. Let $Z$ be an element in $\ptSet(t)$.
	
	If a pointed element $Z\in\ptSet(t)$ is
	$(\prec_{t},\bm)$-unitriangular to $\inj^{t}$, then the
	normalized products $[\prod_{i}X_{n+i}(t)^{f_i}*X(t)^{d_X}*Z*I(t)^{d_I}]^{t}$ are $(\prec_{t},\bm)$-unitriangular to $\inj^{t}$ too.
\end{Lem}
\begin{proof}
	Substitute the factor $Z$ by its expansion in $\inj^t$. The claim follows from Lemma
	\ref{lem:preserve_triangular}.
\end{proof}

\begin{Lem}[Positive Laurent expansion]\label{lem:positive_expansion}
Let $t$ be a given seed and $\gen$ any positive basis of the quantum cluster algebra. If $\gen$ contains the quantum cluster monomials of $t$, then the Laurent expansions of its elements in $\cT(t)$ have coefficients in $\N[q^{\pm\Hf}]$.
\end{Lem}
\begin{proof}
  The expansion coefficients can be interpreted as $q^\Hf$-shifted of
  the positive structure constants of $\gen$, \cf\cite[Proposition
  2.2]{HernandezLeclerc09}.
\end{proof}

\subsection{Weakly triangular bases}\label{sec:weakly_triangular_basis}

\begin{Def}[weakly triangular basis]\label{def:weakly_triangular_basis}
	A weakly triangular basis with respect to a seed $t$ is a bar-invariant $\degL(t)$-pointed basis of $\qClAlg$, such that it factors through the frozen variables (\eqref{eq:factorization_basis}), and $\inj^t$ is $(\prec_t,\bm)$-unitriangular to it. 
\end{Def}
By Lemma \ref{lem:basis_property}(i), the weakly triangular basis with respect to a seed $t$ is unique, which we still denote by $\can^t$. 
By Lemma \ref{lem:basis_property}, the triangular basis and the weakly triangular basis only differ at the triangularity properties.

\begin{Lem}\label{lem:basis_property}
Let $t$ be any given seed. We have the following:
\begin{enumerate}[(i)]
\item a weakly triangular basis with respect to $t$, if exists, is unique;

\item the triangular basis with respect to $t$, if exists, equals the weakly triangular basis;

\item if the weakly triangular basis $\can^t$ exists, then it contains the quantum cluster monomials
$X(t)^{g+f}$, $I(t)^g \cdot X(t)^f$, $\forall g\in\N^n, f\in\Z^{\{n+1,\ldots ,m\}}$. In particular, $\can^t(0)=1$.

\end{enumerate}
\end{Lem}
\begin{proof}
	The statements (i)(ii) easily follow from properties of unitriangular transitions:
	
	(i) Let $\can^{(i)}$, $i=1,2$, be any two weakly triangular bases \wrt $t$. Because $\inj^t$ is $(\prec_t,\bm)$-unitriangular to $\can^{(1)}$, conversely, $\can^{(1)}$ is also $(\prec_t,\bm)$-unitriangular to $\inj^t$ (Lemma \ref{lem:inverse_transition}). Now composed with the unitriangular transition from $\inj^t$ to $\can^{(2)}$, we get that $\can^{(1)}$ is $(\prec_t,\bm)$-unitriangular to $\can^{(2)}$. Because all basis elements in $\can^{(i)}$ are bar-invariant, we have $\can^{(1)}(\tg)=\can^{(2)}(\tg)$ for all $\tg\in \degL(t)$, \cf Lemma \ref{lem:triangular}(ii).
	
	(ii) The claim follows from definition and Lemma \ref{lem:triangular_basis_property}.
	
	(iii) The bar-invariant elements $X(t)^d$ and $I(t)^d$ belong to $\inj^t$ and, consequently, are $(\prec_t,\bm)$-unitriangular to $\can^t$. Lemma \ref{lem:triangular}(ii) implies that they belong to $\can^t$.
\end{proof}

The following lemma easily follows from properties of unitriangular transitions. It helps simplify many arguments in later sections.

\begin{Lem}\label{lem:condition_equivalence}
Let $\can^t$ be the weakly triangular basis with respect to a seed $t$ and  $k\in[1,n]$ any vertex.
  \begin{enumerate}[(i)]
  \item The pointed set $[X_k(t)*\can^t]^t$ is $(\prec_t,\bm)$-unitriangular to $\can^t$ if and only if $\inj^t$ is $(\prec_t,\bm)$-unitriangular to the pointed set $[X_k(t)*\inj^t]^t$.
  
  \item $[X_k(t)*\inj^t]^t$ is $(\prec_t,\bm)$-unitriangular to $\inj^t$ for all $1\leq k\leq n$ if and only if $\inj^t(f,d_X,d_I)$ is $(\prec_t,\bm)$-unitriangular to $\inj^t$, $\forall f\in\Z^{\{n+1,\ldots,m\}},d_X,d_I\in\N^n$.
  
  \item For any $d\in\N$, $X_k(\mu_k t)^d$ belongs to $\can^t$ if and only if $X_k(\mu_k t)^d$ is $(\prec_t,\bm)$-unitriangular to $\inj^t$.
  
\item For any $d\in\N$, $I_k(\mu_k t)^d$ belongs to $\can^t$ if and only if $I_k(\mu_k t)^d$ is $(\prec_t,\bm)$-unitriangular to $\inj^t$.
  \end{enumerate}
  
\end{Lem}

\begin{proof}
We call $(\prec_t,\bm)$-unitriangular by unitriangular for simplicity.

(i) Triangularity Preservation Lemma (Lemma \ref{lem:preserve_triangular}) implies that
$[X_k*\inj^t]^t$ is unitriangular to $[X_k(t)*\can^t]^t$ and vice versus.

First, assume that $[X_k(t)*\can^t]^t$ is unitriangular to $\can^t$.  Because $[X_k*\inj^t]^t$ is unitriangular to $[X_k(t)*\can^t]^t$, it is  unitriangular to $\can^t$ and, consequently, unitriangular to $\inj^t$. Conversely, $\inj^t$ is
unitriangular to $[X_k*\inj^t]^t$.

Similarly, assume that $\inj^t$ is
unitriangular to $[X_k*\inj^t]^t$. Because $[X_k*\can^t]^t$ is unitriangular to $[X_k(t)*\inj^t]^t$, it is unitriangular to $\inj^t$ and, consequently, unitriangular to $\can^t$. Conversely, $\can^t$ is
unitriangular to $[X_k*\can^t]^t$.

(ii) Notice that all elements in $[X_k(t)*\inj^t]^t$ takes the form $\inj^t(f,d_X,d_I)$. Therefore, if all $\inj^t(f,d_X,d_I)$ are unitriangular to $\inj^t$, so does the pointed set $\inj^t$.

Conversely, assume $[X_k(t)*\inj^t]^t$ is unitriangular to $\inj^t$ $\forall 1\leq k\leq n$. Notice that any $\inj^t(f,d_X,d_I)$ can be obtained from the triangular basis element $\inj^t(f,0,d_I)\in\inj^t$ by repeatedly left multiplying the $X$-variables. It follows from Triangularity Preservation Lemma (Lemma \ref{lem:preserve_triangular}) that $\inj^t(f,d_X,d_I)$ remains unitriangular to $\inj^t$.

(iii)(iv) The claims are obvious by the bar-invariance and Lemma \ref{lem:basis_property}(i).
\end{proof}


The following easy lemma compares weakly triangular bases in similar seeds.
\begin{Lem}\label{lem:variation_triangular_basis}
Let $t $ and $t '$ be two seeds similar via an isomorphism between sets of unfrozen vertices $\var^*:\ex'\simeq \ex$. Assume we have the weakly triangular basis $\can^t$ of
$\qClAlg(t)$ with respect to $t$. Let $\can^{t'}$ denote the set of the elements in $\ptSet(t')$ which are similar
to the elements of $\can^t$.

(i) $\can^{t'}$ is the weakly triangular basis of $\qClAlg(t')$.

(ii) If $\can^t$ is positive, so does $\can^{t'}$. If $\can^t$ is the triangular basis with respect to $t$, then $\can^{t'}$ is the triangular basis with respect to $t'$. 

(iii) Let $\overleftarrow{\mu}$ be a mutation sequence on $\ex$ and $\overleftarrow{\mu}'$ the corresponding sequence on $\ex'$. Then, for any $i\in\ex$, the quantum cluster variable $X_i(\overleftarrow{\mu}t)$ belongs to $\can^{t}$ if and only if the corresponding quantum cluster variable $X_{(\var^*)^{-1}i}(\overleftarrow{\mu}'t')$ belongs to $\can^{t'}$.
\end{Lem}

\begin{proof}
  (i) By construction, $\can^{t'}$ factors through the frozen variables, is $D(t')$-pointed and, therefore, linearly
  independent. 
  
  Any given basis element $b\in \can^t$ is a polynomial of quantum
  cluster variables of $\qClAlg(t)$ with coefficients in
  $\Z[q^{\pm\Hf}]P$, in other words, a finite sum of products of quantum cluster variables. We can view the finite sum as $\prec_t$-unitriangular decomposition by Lemma \ref{lem:triangular}(iii) and, by Lemma \ref{lem:variation}(iii), deduce that any $b'\in\can^{t'}$
  similar to $b$ is a polynomial of quantum
  cluster variables of $\qClAlg(t')$ with coefficients in $\Z[q^{\pm\Hf}]P'$. Therefore, $b'$ is contained in $\qClAlg(t')$. 
  
  We see that $\can^{t'}$ spans
  $\qClAlg(t')$ by Theorem \ref{thm:correction}. It is $(\prec_t,\bm)$-unitriangular to $\inj^{t'}$ by Lemma \ref{lem:variation}(iii).

(ii) The claims follow from Theorem \ref{thm:correction} and Lemma
\ref{lem:variation}(iii) respectively.

(iii) Notice that $X_{i}(\overleftarrow{\mu}t)$ is similar to $X_{(\var^*)^{-1}i}(\overleftarrow{\mu}'t')$ by Lemma
\ref{lem:variation}. The claim follows from the construction of $\can^{t'}$.
\end{proof}

\subsection{admissibility}
\begin{Def}[admissible]
	Let $k\in[1,n]$ be a given vertex.	A basis $\gen$ of the quantum cluster algebra $\qClAlg$ is said to be admissible in direction $k$ with respect to a seed $t$, if the quantum cluster variables $X_k(\mu_k t)$, $I_k(\mu_k t)$ belongs to $\gen$.
\end{Def}

\begin{Lem}\label{lem:admissible_cluster_monomial}
	Let $\gen$ be a positive basis of $\qClAlg$, such that it contains the quantum cluster monomials in $t$. If $\gen$ contains the quantum cluster variable $X_k(\mu_k t)$ for some $1\leq k\leq n$, then it contains the quantum cluster monomials in $\mu_k t$.
\end{Lem}
\begin{proof}
	Denote $t'=\mu_k t$, $X_k(\mu_k t)=X'_k(t)$, $Y_k(\mu_k t)=Y'_k(t)$.
	
	The statement follows from the positive
	Laurent expansions of $\gen$ in the $\cT(t)$ and the Laurent phenomenon in $t'$.
	
	For simplicity, we specialize $q^\Hf$ to $1$ and omit the symbol $t$. Any quantum cluster monomial $M$ in $t'$ is a product of basis elements in $\gen$
	\begin{align*}
	M=X^{d_{\hat{k}}}(X_k')^{d_k}=\sum_{\tg\preceq \deg M}b(\tg)\gen(\tg),
	\end{align*}
	where $b(\tg)\in
	\N$, $b(\deg M)=1$, $d_k\in\N$,$d_{\hat{k}}\in\N^{[1,m]-\{k\}}$. The Laurent expansion of $\LHS$ is $X^{d_{\hat{k}}}\prod_j X_j^{d_k[-b_{jk}]_+}X_k^{-d_k}(1+Y_k)^{d_k}$. By Lemma \ref{lem:positive_expansion}, the positive Laurent expansion of each basis element $\gen(\tg)$ appearing must take the form
	\begin{align*}
	X^{d_{\hat{k}}}\prod_{1\leq j\leq m} X_j^{d_k[-b_{jk}]_+}X_k^{-d_k}F_{\tg}(Y_k), 
	\end{align*}
	where $F_{\tg}(Y_k)$ is a polynomial in $Y_k$ whose coefficients are non-negative but not larger than those of $(1+Y_k)^{d_k}$ at all degrees. It follows that, if we let $F^\circ _\tg(Y_k)$ denote the polynomial $F_{\tg}(Y_k^{-1})Y_k^{d_k}$. Then its coefficients are also non-negative but not larger than those of $(1+Y_k^{-1})^{d_k} Y_k^{d_k}=(1+Y_k)^{d_k}$.
	
	Consider the Laurent expansion of $\gen(\tg)$ in $\cT(t')$. Using $Y_k=Y_k'^{-1}$, we obtain
	\begin{align*}
	\gen(\tg)&=X^{d_{\hat{k}}}X_k'^{d_k}\cdot \frac{1}{(1+Y_k)^{d_k}}F_{\tg}(Y_k)\\
	&=X^{d_{\hat{k}}}X_k'^{d_k}\cdot \frac{Y_k'^{d_k}}{(1+Y_k')^{d_k}}\frac{F^\circ_{\tg}(Y_k')}{Y_k'^{d_k}}.
	\end{align*}
	This is an Laurent polynomial in $\cT(t')$ if and only if $\frac{F^\circ_{\tg}(Y_k')}{(1+Y_k')^{d_k}}$ is a Laurent polynomial in $\cT(t')$. But $F^\circ_{\tg}(Y_k')$ has non-negative coefficients no larger than those of $(1+Y_k')^{d_k}$. It follows that $F^\circ_{\tg}(Y_k')=(1+Y_k')^{d_k}$ and, consequently, $X^{d_{\hat{k}}}(X_k')^{d_k}=\gen(\tg)$.
\end{proof}
\subsection{One step mutations}\label{sec:one_step_mutation}

As the most difficult in the verification of the common triangular basis, we shall show in Proposition \ref{prop:prove_X} that the expected
  transformation rule \eqref{eq:degree_mutation} of the leading degrees of basis elements provides a crucial step towards the admissibility.

We will control the
Laurent expansions of triangular basis elements with the
help of the positivity
of the basis $\can^t$ and Cluster Expansion Assumption. To warm up, let us start with an
unsuccessful but inspiring calculation. 

Fix a given vertex $k\in[1,n]$, denote $\mu_k t=t'$, $X_k(\mu_k t)=X'_k(t)$, $I_k(\mu_k t)=I'_k(t)$. Recall that $\deg^t I_k=-e_k+f^{(k)}$ for some $f^{(k)}\in\Z^{\{n+1,\ldots,m\}}$.
\begin{Prop}\label{prop:contain_X}
Let $\can^t$ be the weakly triangular basis with respect to a seed $t$ and assume it to be positive. Then we have the following Laurent expansion in $\cT(t)$:
\begin{align*}
\can^t(\deg^t X'_k(t))=X'_k(t)+\sum_{d\in\N^n, d>e_k}u_d X(t)^{\deg^t X'_k(t)}Y(t)^d,
\end{align*}
where $u_d\in\N[q^\pm]$.
\end{Prop}
\begin{proof}
Omit the symbol $t$ for simplicity. Denote $\tb^{-}=\sum_{j}[-b_{jk}]_+e_j$, $X_-=X^{\tb^{-}}$. Then we have $\deg X'_k=\tb^{-}-e_k$
and $X'_k=X_k^{-1}X_-(1+Y_k)$.

We consider the normalized twisted
product $[X^{-f^{(k)}}*X_-*I_k]$. By Lemma
\ref{lem:basis_property}(iii), all factors belong to the
basis $\can$. By Lemma \ref{lem:positive_expansion}, the Laurent
expansion of $I_k$ has non-negative coefficients. Further applying Cluster Expansion Assumption, we see that the Laurent expansion of $I_k$
must take the form
$$
I_k=X^{-e_k+f^{(k)}}(1+\alpha_{e_k} Y_k+\sum_{d\in \N^n: d>e_k}\alpha_dY^d),\ \alpha_d\in \N[q^{\pm\Hf}],
$$
where the bar-invariant $q$-coefficient $\alpha_{e_k}$ of $I_k$ at the term $X^{-e_k+f^{(k)}} Y_k$ belongs to $\N[q^{\pm\Hf}]$ and satisfies $\alpha_{e_k}(1)=1$ and, consequently, equals $1$.

Therefore, the Laurent expansion of $[X^{-f^{(k)}}*X_-*I_k]$ becomes
\begin{align}\label{eq:expand_XI_assumption} 
  \begin{split}
    [X^{-f^{(k)}}*X_-*I_k]=&X^{-f^{(k)}}X_-X^{-e_k+f^{(k)}}(1+Y_k+\sum_{d>e_k}\alpha_dq^{\Hf\Lambda(\tb^{-},\deg
      Y^d)}Y^d)\\
    =&X_k^{-1}X_-+X_k^{-1}X_-Y_k+\sum_{d:d>e_k}\alpha_dq^{\Hf\Lambda(\tb^{-},\deg Y^d)}X_k^{-1}X_-Y^d.
  \end{split}
\end{align}
On the other hand, by the triangularity of $\can$, we have
\begin{align}\label{eq:expand_XI_condition}
  [X^{-f^{(k)}}*X_-*I_k]=\can(\tb^{-}-e_k)+\sum_{\eta\prec
    \tb^{-}-e_k}c_\eta \can(\eta),
\end{align}
where the coefficients $c_\eta\in \N[q^{\pm\Hf}]\cap \bm$.

Compare \eqref{eq:expand_XI_assumption} with the Laurent expansion
of \eqref{eq:expand_XI_condition}. They all have non-negative
coefficients in $\N[q^{\pm\Hf}]$. By analyzing these coefficients, we easily deduce that the monomial $X_k^{-1}X_-Y_k$ can only be the contribution of $\can(\tb^{-}-e_k)$.

Therefore, the Laurent expansion of $\can(\tb^{-}-e_k)$ takes the form
\begin{align*}
  \can(\tb^{-}-e_k)=X_k^{-1}X_-+X_k^{-1}X_-Y_k+\sum_{d\in\N^n,d>e_k}u_dX_k^{-1}X_-Y^d.
\end{align*}
The claim follows.
\end{proof}

\begin{Rem}
  In the proof of Proposition \ref{prop:contain_X}, the calculation of
  the desired basis element $\can^t(\deg^t X'_k(t))$ is based on the
  following two steps:

  \begin{enumerate}[(i)]
  \item put the desired basis element $\can^t(\deg^t X'_k(t))$ in
    the product of cluster variables contained in $\can^t$
    (cf. \eqref{eq:expand_XI_condition});
\item analyze its Laurent
    expansion by the positivity and the estimation from Cluster Expansion Assumption.
  \end{enumerate}

However, Proposition \ref{prop:contain_X} can not precisely
  control $\can^t(\deg^t X'_k(t))$, because the desired basis element
  appears as the leading term in \eqref{eq:expand_XI_condition}. To
  fix this, in Proposition \ref{prop:prove_X}, we put it as the
  last term in a similar but less intuitive equation. 
  \end{Rem}


 \begin{Prop}[One step mutation]\label{prop:prove_X}
Given any seed $t$ and any vertex $k\in [1, n]$. Let $\can^t$ and $\can^{t[-1]}$ be the weakly triangular bases with respect to $t$, $t[-1]$, respectively and assume them to be positive. Furthermore, assume $\can^t=\can^{ t[-1]}$ as sets such that we have the following equality of basis elements:
$$\can^{t}(\deg^{t} X'_{k}( t))=\can^{t[-1]}(\deg^{t[-1]} X'_k(t)).$$
Then the following quantum cluster variable is contained in $\can^t$:
  $$X'_{k}(t)=\can^{t}(\deg^{t} X'_{k}( t)).$$
\end{Prop}
This proposition is the crucial step in verifying the admissibility of a basis.
\begin{proof}
	We provide a tricky proof, whose idea was sketched in the end of Section \ref{sec:triangular_basis_brief}.
	
	(i) We want to find out the leading degree of the Laurent expansion of $X_k'(t)$ in the quantum torus $\cT(t[-1])$.

	The quantum cluster variable $X'_k(t)$ is given by the exchange relation $$X'_k(t)=X_k(t)^{-1}X(t)^{f^{+}+b^{+}}
	+X_k(t)^{-1}X(t)^{f^{-}+b^{-}},$$
	where we denote the multiplicities over the exchangeable vertices by
	\begin{align*}
	{b^{+}}&=\sum_{1\leq i\leq n}[b_{ik}(t)]_+e_i,\\
	b^{-}&=\sum_{1\leq j\leq
		n}[-b_{jk}(t)]_+e_j,
	\end{align*}
	and the multiplicities over the frozen vertices (coefficients) by
	\begin{align*}
	f^{+}&=\sum_{1\leq i\leq
		m-n}[b_{n+i,k}(t)]_+e_{n+i},\\
	f^{-}&=\sum_{1\leq j\leq
		m-n}[-b_{j+n,k}(t)]_+e_{j+n}.
	\end{align*}
	
	Recall that we have a mutation sequence $\sigma^{-1}(\Sigma)$ from $t[-1]$ to $t$ together with a permutation $\sigma$ of $[1,n]$ such that $I_i(t[-1])=X_{\sigma i}(t)$ for any $i\in[1,n]$ (Section \ref{sec:chain_seed}). Also, recall that we denote $\sigma e_i=e_{\sigma^{-1}i}$. Then, in $\cT(t[-1])$, the above exchange equation becomes
	\begin{align}\label{eq:mutation_in_I}
	\begin{split}
	X'_k(t)=&I_{\sigma^{-1}
		k}(t[-1])^{-1}X(t[-1])^{f^{+}} I(t[-1])^{\sigma b^{+}}\\
	&+I_{\sigma^{-1}
		k}(t[-1])^{-1}X(t[-1])^{f^{-}}I(t[-1])^{\sigma b^{-}}.
	\end{split}
	\end{align}
	
	By
	\eqref{eq:mutated_injective_degree}, the desired leading degree is given by
	\begin{align*}
	\deg^{t[-1]} X'_k(t)&=-\deg^{t[-1]}I_{\sigma^{-1} k}(t[-1])+f^{+}+\deg^{t[-1]}I(t[-1])^{\sigma b^{+}}\\
	&=e_{\sigma^{-1}k}-f^{(\sigma^{-1} k)}+f^{+}+\deg^{t[-1]}I(t[-1])^{\sigma b^{+}}.
	\end{align*}

	(ii) Consider the desired basis element $\can^{t[-1]}(\deg^{t[-1]}X_k'(t))$ of the triangular basis $\can^{t[-1]}$, which shares the desired leading degree $\deg^{t[-1]}X_k'(t)$ with $X_k'(t)$ in the quantum torus $\cT(t[-1])$. Notice that the basis $\can^t=\can^{t[-1]}$ is positive. By step (i), this basis element must appear as the leading term of the following product of elements in $\can^{t[-1]}$:
	\begin{align}\label{eq:expand_product_I}
	\begin{split}
	q^{-\alpha}X(t[-1])^{e_{\sigma^{-1}k}-f^{(\sigma^{-1} k)}}*&X(t[-1])^{f^{+}} I(t[-1])^{\sigma b^{+}}\\
	&=\can^{t[-1]}(\deg^{t[-1]} X'_k(t))+\sum_{\eta}
	p_\eta\can^{t[-1]}(\eta),
	\end{split}
	\end{align}
	where $\eta\prec_{t[-1]}\deg^{t[-1]} X'_k(t)$, the non-negative coefficients $p_\eta$ belongs to $\bm$ by
	the triangularity of $\can^{t[-1]}$, and the normalization factor $q^{-\alpha}$ given by
	\begin{align*}
	\alpha &=\Hf\Lambda(t[-1])(-\deg^{t[-1]}I_{\sigma^{-1}k}(t[-1]),f^{+}+\deg^{t[-1]} I(t[-1])^{\sigma b^{+}}).
	\end{align*}

	(iii) We try to locate the desired basis element as the last term in the new quantum torus $\cT(t)$ with respect to the new order $\prec_t$.
	
	View equation \eqref{eq:expand_product_I} in the quantum torus
	$\cT(t)$. It becomes
	\begin{align}\label{eq:expand_product_P}
	\begin{split}
	q^{-\alpha}P_k(t)X(t[-1])^{-f^{(\sigma^{-1} k)}}*X(t)^{f^{+}+ b^{+}}&=\can^{t[-1]}(\deg^{t[-1]}
	X'_k(t))+\sum_{\eta} p_\eta\can^{t[-1]}(\eta)\\
	&=\can^{t}(\deg^{t}
	X'_k(t))+\sum_{\eta} p_\eta\can^{t[-1]}(\eta).
	\end{split}
	\end{align}

	Let us compute the desired basis element's leading degree:
	\begin{align*}
	\deg^{t}X'_k(t)=&b^{-}+ f^{-}-e_k\\
	=&-e_k+f^{(\sigma^{-1} k)}-( b^{+}+f^{+}-b^{-}- f^{-})+ b^{+}+f^{+}-f^{(\sigma^{-1} k)}\\
	=&\sigma^{-1}\deg^{t[-1]}I_{\sigma^{-1}k}(t[-1])-\deg^tY_k(t)+ b^{+}+f^{+}-f^{(\sigma^{-1} k)}\\
	\stackrel{\eqref{eq:index_coindex}}{=}&\deg^tP_k(t)Y(t)^{p(k,t)-e_k}X(t)^{ b^{+}+f^{+}-f^{(\sigma^{-1} k)}}.
	\end{align*}
	
	At this moment, we haven't shown that the remaining basis elements $\can^{t[-1])}(\eta)$ in \eqref{eq:expand_product_P} have degrees higher than that of $X'_k(t)$ in the quantum torus $\cT(t)$. Nevertheless, let us proceed by looking at Laurent expansions.
	
	(iv) We want to take Laurent expansions of \eqref{eq:expand_product_P} and explicitly determine the desired basis element as the last few terms.
	
	By Lemma \ref{lem:positive_expansion}, the Laurent expansion of $P_k(t)$ have coefficients in $\N[q^{\pm\Hf}]$. By \eqref{eq:expand_projective}, it must take the form
	\begin{align*}
	P_k(t)=X(t)^{\deg^{t}P_k(t)}Y(t)^{p(k,t)}\cdot(\sum_{e_k<d\leq p(k,t)}\beta_dY(t)^{-d}+Y_k(t)^{-1}+1),
	\end{align*}
	where $\beta_d\in\N[q^{\pm\Hf}]$. Notice that the coefficients $X(t)^{-f^{(\sigma^{-1} k)}}=X(t[-1])^{-f^{(\sigma^{-1} k)}}$. Consequently, the
	Laurent expansion of $\LHS$ of \eqref{eq:expand_product_P} becomes
	\begin{align}\label{eq:pre_calculate_X}
	\begin{split}
	&q^{-\alpha}P_k(t)X(t)^{-f^{(\sigma^{-1} k)}}*X(t)^{f^{+}+ b^{+}}\\
	=&q^{-\alpha}
	(X(t)^{\deg^{t}P_k(t)-f^{(\sigma^{-1} k)}}Y(t)^{p(k,t)}\cdot(\sum_{e_k<d\leq p(k,t)}\beta_dY(t)^{-d}+Y_k(t)^{-1}+1))*X(t)^{f^{+}+ b^{+}}\\
		\stackrel{\eqref{eq:index_coindex}}{=}&q^{-\alpha}
		(X(t)^{-e_k}\cdot(\sum_{e_k<d\leq p(k,t)}\beta_dY(t)^{-d}+Y_k(t)^{-1}+1))*X(t)^{f^{+}+ b^{+}}.
	\end{split}
	\end{align}
	
	We observe that $q^{-\alpha}$ is the normalization $q$-factor for the last two Laurent monomials in \eqref{eq:pre_calculate_X}: by Proposition
	\ref{prop:change_Lambda}, we have
	\begin{align*}
	\alpha=&\Hf\Lambda(t[-1])(-\deg^{t[-1]}I_{\sigma^{-1}k}(t[-1]),f^{+}+\deg^{t[-1]}
	I(t[-1])^{\sigma b^{+}})\\
	=&-\Hf\Lambda(t[-1])(\deg^{t[-1]}I_{\sigma^{-1}k}(t[-1]),f^{+}+\deg^{t[-1]} I(t[-1])^{\sigma b^{+}})\\
	=&-\Hf\Lambda(t)(\deg^t X_k(t),\deg^{t}X(t)^{f^{+}+ b^{+}})\\
	=&\Hf\Lambda(t)(-e_k,\deg^{t}X(t)^{f^{+}+ b^{+}}).
	\end{align*}
Thus the last Laurent monomial has coefficient $1$. The same holds for the second last Laurent monomial because $\Lambda(t)(\deg^t Y_k(t)^{-1},f^{+}+ b^{+})=0$. Now the \eqref{eq:pre_calculate_X} becomes
	\begin{align}\label{eq:calculate_X}
	\begin{split}  
	&q^{-\alpha}P_k(t)X(t)^{-f^{(\sigma^{-1} k)}}*X(t)^{f^{+}+ b^{+}}\\
	=& \sum_{e_k<d\leq p(k,t)}q^{\alpha_d}\beta_d
	X(t)^{b^{+}+f^{+}-e_k}Y(t)^{-d}
	+X(t)^{b^{+}+f^{+}-e_k}Y(t)^{-e_k}+X(t)^{b^{+}+f^{+}-e_k}\\
	=&\sum_{e_k<d\leq p(k,t)}q^{\alpha_d}\beta_d
	X(t)^{b^{+}+f^{+}-e_k}Y(t)^{-d}+X(t)^{\deg^{t}X_k'(t)}+X(t)^{\deg^{t}X_k'(t)}Y_k(t),
	\end{split}
	\end{align}
	where $\alpha_d=\Hf\Lambda(t)(\deg^t Y(t)^{-d}, b^{+})\leq 0$.
	
	Notice that the last two terms sum to $X'_k(t)$. Similar to the discussion in Proposition \ref{prop:contain_X}, we compare this Laurent
	expansion to the Laurent expansion of the $\RHS$ of
	\eqref{eq:expand_product_P}, which has non-negative coefficients by Lemma \ref{lem:positive_expansion}.
	
	First, $\can^t(\deg X'_k(t))$ contains $X^{\deg^{t}X_k'(t)}$ because it
	needs to have the leading degree $\deg^t X_k'(t)$. Notice that the Laurent
	expansion of $\can^t(\deg X'_k(t))$ has non-negative coefficients less than those in
	\eqref{eq:calculate_X}, and $X(t)^{\deg^tX_k'(t)}Y_k(t)$ is the only term whose degree is dominated by
	$\deg^t X_k'(t)$. Moreover, if the term $X^{\deg^{t}X_k'(t)}Y_k(t)$ is
	the contribution from some other basis elements $\can^{t[-1]}(\eta)$ in
	\eqref{eq:expand_product_P}, this will contradicts to the properties
	that $p_\eta\in\bm$ and the bar-invariance of
	$\can^{t[-1]}(\eta)$. Therefore, the desired basis element
	$\can^t(\deg X'_k(t))$ equals $X^{\deg^{t}X_k'(t)}+X^{\deg^{t}X_k'(t)}Y_k(t)=X'_k(t)$.
\end{proof}
\begin{Eg}[Control Laurent expansion]\label{eg:control_expansion}
	We continue Example \ref{eg:compare_inj_proj}. Take $k=4$. We want to control the Laurent expansion of $\can^t(\deg^t
	X_k'(t))$ by locating it as the last few terms of a Laurent
	polynomial in $\cT(t)$ and deduce that it is the cluster variable $X_k'(t)$.
	
	We have the following quantum cluster variables in $\cT(t)$:
	\begin{align*}
	X_4'(t)=&X^{e_2-e_4}(1+Y_4)=X^{e_1-e_4+e_8}(Y_4^{-1}+1),\\
	P_4(t)=&X^{-e_4+e_5+e_8}(Y^{-e_1-e_4}+Y^{-e_4}+1).
	\end{align*}
	Denote the coefficient part of the degree
	$\deg^{t[-1]}\inj_{\sigma^{-1}4}(t[-1])=-e_1+e_5+e_8$ by $f^{(\sigma^{-1} k)}=e_5+e_8$. Denote $ b^{+}=\sum_{1\leq i\leq n}[b_{ik}(t)]_+e_i=e_1$,
	$f^{+}=\sum_{i>n}[b_{ik}(t)]_+e_i=e_8$. Then $\sigma  b^{+}=e_4$.
	
	The quantization matrix $\Lambda(t[-1])$ is given by
	\begin{align*}
	\Lambda(t[-1])=\left(
	\begin{array}{cccccccc}
	0&0&0&0&1&0&0&0\\
	0&0&0&0&0&-1&1&1\\
	0&0&0&0&0&-1&1&0\\
	0&0&0&0&0&0&1&0\\
	-1&0&0&0&0&1&0&-1\\
	0&1&1&0&-1&0&1&1\\
	0&-1&-1&-1&0&-1&0&0\\
	0&-1&0&0&1&-1&0&0\\
	\end{array}
	\right).
	\end{align*}
	
	Assume that $\can^t=\can^{t[-1]}$ with positive structure
	constants and
	$\can^{t}(\deg^tX_4'(t))=\can^{t[-1]}(\deg^{t[-1]}X_4'(t))$. Compute the degree
	\begin{align*}
	\deg^{t[-1]}X_4'(t)&=e_1-e_4+e_6+e_7\\
	&=e_{\sigma^{-1}4}-f^{(\sigma^{-1} k)}+f^{+}+\deg^{t[-1]}I(t[-1])^{\sigma b^{+}}.
	\end{align*}
	Then the desired basis element $\can^{t}(\deg^tX_4'(t))$ appears in
	the $(\prec_{t[-1]},\bm)$-unitriangular expansion of the following normalized product in $\cT(t[-1])$:
	\begin{align}\label{eq:eg_locate_simple}
	\begin{split}
	q^{-\alpha}
	X(t[-1])^{e_{\sigma^{-1}4}-f^{(\sigma^{-1} k)}}&*X(t[-1])^{f^{+}}I(t[-1])^{\sigma b^{+}}\\
	&=q^{-\alpha}P_k(t)X(t[-1])^{-f^{(\sigma^{-1} k)}}*X(t)^{ b^{+}+f^{+}}.
	\end{split}
	\end{align}
	Compute the quantization degree
	\begin{align*}
	\alpha=&\Hf\Lambda(t[-1])(e_1-f^{(\sigma^{-1} k)},f^{+}+\deg^{t[-1]}I(t[-1])^{\sigma b^{+}})\\
	&=\Hf\Lambda(t[-1])(e_1-e_5-e_8,-e_4+e_5+e_6+e_7+e_8)\\
	&=\Hf\\
	&=\Hf\Lambda(t)(-e_4,e_1+e_8)\\
	&=\Hf\Lambda(t)(\deg^tP_4(t)Y_1Y_4-f^{(\sigma^{-1} k)},f^{+}+ b^{+}).
	\end{align*}
	Now, view the twisted product \eqref{eq:eg_locate_simple} in $\cT(t)$, we obtain
	\begin{align*}
	q^{-\alpha}P_4(t)X^{-f^{(\sigma^{-1} k)}}*X^{f^{+}+ b^{+}}=&q^{-\Hf}X^{e_1-e_4+e_8}Y_1^{-1}Y_4^{-1}+X^{e_1-e_4+e_8}(Y_4^{-1}+1)\\
	=&q^{-\Hf}X^{-e_5}+X_4'(t).
	\end{align*}
	
	Recall that the expansion of the normalized product into $\can^{t[-1]}$ takes the form
	\begin{align*}
	q^{-\alpha}P_4(t)X(t)^{-f^{(\sigma^{-1} k)}}*X(t)^{f^{+}+ b^{+}}&=\can^{t[-1]}(\deg^{t[-1]}X_4'(t))+\sum_{\eta\prec_{t[-1]}\deg^{t[-1]}X_4'(t)}p_\eta\can^{t[-1]}(\eta)\\
	&=\can^{t}(\deg^{t}X_4'(t))+\sum_{\eta\prec_{t[-1]}\deg^{t[-1]}X_4'(t)}p_\eta\can^{t[-1]}(\eta)
	\end{align*}
	with non-negative coefficients $p_\eta\in\bm$.
	
	By using bar-invariance of basis elements (\cf proof(iv) of Proposition \ref{prop:prove_X}), we deduce that the Laurent expansion $\can^{t}(\deg^tX_4'(t))$ in $\cT(t)$ consists of the lowest two terms of the Laurent expansion of the left hand side which is given by Cluster Expansion Assumption. Consequently, $\can^{t}(\deg^tX_4'(t))$ equals $X_k'(t)$.
\end{Eg}

\subsection{Triangular bases \wrt new seeds}\label{sec:change_seed}

In this subsection, we show that if the triangular basis
$\can^t$ is admissible in direction $k$, then it lifts to the common triangular basis \wrt $\{t,t'\}$, where $t'=\mu_k t$ (Proposition \ref{prop:condition_change_seed}).

Most arguments in this subsection easily follow from basic properties of unitriangular transitions, except the nontrivial part (Case ii-b) in the proof of Proposition \ref{prop:condition_change_seed}, where the unitriangularity comes from the exchange relation of a quantum cluster variable.

The following easy lemma tells us that unitriangularity remains unchanged even when the seeds and dominance orders change, if the leading degrees follow tropical transformations (\cf Definition \ref{def:compatible_pointed_set}).
 \begin{Lem}[Dominance order change]\label{lem:change_order} Let $t^{(i)}$, $i=1,2$, be two seeds of a given quantum cluster algebra and $\gen^{(i)}$  bar-invariant $\degL(t^{(i)})$-pointed bases. Assume that $\gen^{(1)}$ and $\gen^{(2)}$ are compatible. Let $Z$ be any element in $\qClAlg$ such that it has unique leading degrees in both $\cT(t^{(1)})$ and $\cT(t^{(2)})$ and $\deg^{t^{(2)}} Z=\phi_{t^{(2)},t^{(1)}}(\deg^{t^{(1)}} Z)$.

If $Z$ belongs to $\ptSet(t^{(1)})$ and is $(\prec_{t^{(1)}},\bm)$-unitriangular to $\gen^{(1)}$, then it belongs to $\ptSet(t^{(2)})$ and is $(\prec_{t^{(2)}},\bm)$-unitriangular to $\gen^{(2)}$ as well.
 \end{Lem}
 \begin{proof}
Denote $\deg^{t^{(1)}} Z=\eta_Z$ and $\deg^{t^{(2)}} Z=\tg_Z$ for simplicity. We have
\begin{align*}
  Z=\gen^{(1)}(\eta_Z)+\sum_{\eta'\prec_{t^{(1)}}\eta_Z} c_{\eta'}\gen^{(1)}(\eta'),\ c_{\eta'}\in \bm.
\end{align*}
Because $Z\in\qClAlg$, $\RHS$ is a finite sum. Since $\gen^{(1)}$ and $\gen^{(2)}$ are compatible, we can rewrite $\RHS$ and obtain
\begin{align*}
 Z=\gen^{(2)}(\tg_Z)+\sum_{\eta'\prec_{t^{(1)}}\eta_Z} c_{\eta'}\gen^{(2)}(\phi_{t^{(2)},t^{(1)}}(\eta')),\ c_{\eta'}\in \bm.
\end{align*}
Notice that $\phi_{t^{(2)},t^{(1)}}\eta'\neq \phi_{t^{(2)},t^{(1)}}\eta_Z=\tg_Z$. By this
equation, $Z$ has coefficient $1$ at the degree $\tg_Z$. Since $\tg_Z$ is the leading degree of $Z$ by our hypothesis, $Z$ is pointed. Moreover, this
expansion is a finite sum. By applying Lemma
\ref{lem:triangular}(iii), we see it must take the form
\begin{align*}
Z=\gen^{(2)}(\tg_Z)+\sum_{\tg'\prec_{t^{(2)}}\tg_Z} c_{\tg'}\gen^{(2)}(\tg'),\ c_{\tg'}\in\bm.
\end{align*}
\end{proof}

\begin{Lem}[Weakly triangular basis \wrt new seed]\label{lem:basis_change_seed}
Let $t,t'$ be two seeds related by a mutation at some vertex
$k$. Let $\can^t$ be the weakly triangular basis with respect to $t$ and assume it to be positive and admissible in direction $k$. Furthermore, assume that, the injective pointed set $\inj^{t'}$ is
$(\prec_t,\bm)$-unitriangular to $\inj^t$ in the quantum torus $\cT(t)$. Then the following claims hold.
\begin{enumerate}[(i)]
\item The weakly triangular basis $\can^{t'}$ with respect to $t'$ exists and is compatible with $\can^t$.
\item $\inj^t$ is $(\prec_{t'},\bm)$-unitriangular to $\inj^{t'}$ in $\cT(t')$.
\end{enumerate}
\end{Lem}
\begin{proof}

(i) Since the positive basis $\can^t$ is admissible in direction $k$, Lemma \ref{lem:admissible_cluster_monomial} implies that it contains the quantum cluster monomials in $t'$ and $t'[1]$.

It follows from Lemma \ref{lem:neighbouring_pointed} that $\inj^{t'}(\tg)$ is pointed at degree $\phi_{t,t'}\tg$
in $\cT(t)$, $\forall \tg\in\degL(t')$. By the hypothesis in the Proposition, it is $(\prec_t,\bm)$-unitriangular to $\inj^t$ and, consequently, to $\can^t$. So we have a $(\prec_t,\bm)$-unitriangular decomposition.
\begin{align}\label{eq:bad_expansion}
  \inj^{t'}(\tg)=\can^t(\eta)+\sum_{\eta'\prec_t \eta} c_{\eta \eta'}
  \can^t(\eta'),\ \eta,\eta'\in\degL(t),\ \tg=\phi_{t',t}(\eta),\ c_{\eta \eta'}\in\bm.
\end{align}
View \eqref{eq:bad_expansion} in $\cT(t')$ from now on, in which $\can^t$ consists of positive Laurent polynomials (Lemma \ref{lem:positive_expansion}). Since $\inj^{t'}(\tg)$ belong to $\qClAlg$, this linear combination is a
finite sum. Furthermore, because $X_k(t')$ and $I_k(t')$ belong to
$\can^t$, and $\can^t$ is positive, $c_{\eta\eta'}$ and the Laurent expansions of all terms appearing are non-negative. Compare both sides: $\LHS$ has a unique
leading Laurent monomial of degree $\tg$ with coefficient $1$. Therefore, $\RHS$ has a unique
leading Laurent monomial of degree $\tg$ with coefficient $1$, which must be
the contribution from $\can^t(\eta)$. Consequently, $\can^t(\eta)$ is
pointed at $\tg$ in $\cT(t')$.

Thus we have a $\degL(t')$-pointed basis $\can^{t'}$ defined by $\can^{t'}(\tg)=\can^t({\phi_{t,t'}\tg})$, $\forall \tg\in\degL(t')$. The bases $\can^t$ and $\can^{t'}$ are compatible. By Lemma \ref{lem:change_order}, $\inj^{t'}$ is $(\prec_{t'},\bm)$-unitriangular to $\can^{t'}$. The claim follows.

(ii) $\inj^t$ is $(\prec_{t'},\bm)$-unitriangular
to $\can^{t'}$ by Lemma \ref{lem:change_order} (use Lemma \ref{lem:neighbouring_pointed}). $\can^{t'}$ is $(\prec_{t'},\bm)$-unitriangular to $\inj^{t'}$
by its triangularity. The claim follows as a composition.
\end{proof}

\begin{Prop}[Triangular basis \wrt new seed]\label{prop:condition_change_seed}
Let $t$, $t'$ be two seeds related by a mutation at some given vertex $k$. Let $\can^t$ be the triangular basis with respect to $t$. If $\can^t$ is positive and admissible in direction $k$ with respect to $t$, then the triangular basis with respect to $t'$ exists and is compatible with $\can^t$.
\end{Prop}
\begin{proof}
Our proof is based on basic properties of unitriangular transitions except the nontrivial last step (Case-iib), where the triangularity arises from the exchange relation of a mutation. We prove the following two claims.

  \begin{enumerate}[{Claim} (i)]
  	\item If $\can^t$ is positive and admissible in direction $k$, then the weakly triangular basis $\can^{t'}$ exists and is compatible with $\can^t$.
  	
  	\item If the weakly triangular basis $\can^{t'}$ exists and is compatible with the triangular basis $\can^t$, then it is the triangular basis with respect to $t'$.
  \end{enumerate}

  For simplicity, we omit the notation $(t)$ inside $X_i(t)$, $X'_k(t)=X_k(t')$, $I'_k(t)=I_k(t')$, $Y_k(t)$, $Y'_k(t)=Y_k(t')$ in the following.

\emph{
Proof of claim (i):}

By Definition \ref{def:injective_pointed_element}, for any given $\tg\in\degL(t')$, we have, in $\cT(t')$,
\begin{align*}
  \inj^{t'}(\tg)=[X^{g_f}X^{[g_{\hk}]_+}*(X_k')^{[g_k]_+}*(I_k')^{[-g_k]_+}*I_k ^{[-g_{\hk}]_+}]^{t'},
\end{align*}
where $g_{\hk}\in\Z^n$ is obtained from $g$ by imposing $g_k=0$. By
Lemma \ref{lem:neighbouring_pointed}, it is normalized in $\cT(t)$ as well:
\begin{align*}
  \inj^{t'}(\tg)=[X^{g_f}X^{[g_{\hk}]_+}*(X_k')^{[g_k]_+}*(I_k')^{[-g_k]_+}*I_k ^{[-g_{\hk}]_+}]^{t}.
\end{align*}

Because that the factors $(X_k')^{[g_k]_+}$ or $(I_k')^{[-g_k]_+}$ belong to $\can^t$ by Lemma \ref{lem:admissible_cluster_monomial}, they are $(\prec_t,\bm)$-unitriangular to $\inj^t$ by the triangularity of $\can^t$. Then Lemma
\ref{lem:substitute_expansion} implies that $\inj^{t'}$ is
$(\prec_t,\bm)$-unitriangular to $\inj^t$. Then we can apply Lemma \ref{lem:basis_change_seed} and the claim follows.

\emph{Proof of claim (ii):
}
We want to show that $[X_i(t')*S]^{t'}$ is $(\prec_{t'},\bm)$-unitriangular to $\can^{t'}$, for any $S\in\can^{t'}$, $1\leq i\leq m$.

First, take any $i\neq k$. We have $X_i(t')=X_i(t)$. 
Then the normalized twisted product $[X_i(t')*S]^t=[X_i(t)*S]^t$ is $(\prec_t,\bm)$-unitriangular to $\can^t$.

In $\cT(t')$, the twisted product $[X_i(t)*S]^t$ has the unique maximal degree
$\deg^{t'}(X_i(t))+\deg^{t'}S$. This maximal degree turns out to be $\phi_{t',
  t}\deg^tX_i(t)+\phi_{t',t}\deg^t S=\phi_{t',
  t}\deg^t(X_i(t)*S)$, thanks to the sign coherence at the $k$-th
components (Lemma \ref{lem:additivity}). Therefore, we can apply Lemma \ref{lem:change_order} and
deduce that $[X_i(t')*S]^t$ is pointed and $(\prec_{t'},\bm)$-unitriangular to
$\can^{t'}$ in $\cT(t')$.

Then, it remains to prove that the normalized twisted product $[X_k'*S]^{t'}$ is
$(\prec_{t'},\bm)$-unitriangular to $\can^{t'}$.

By Lemma \ref{lem:neighbouring_pointed}, $\inj^t$ is $\degL(t')$-pointed in $\cT(t')$ with the expected leading degree change via tropical transformation. Since $\can^t$ and $\can^{t'}$ are compatible, it is further $(\prec_{t'},\bm)$-unitriangular to $\can^{t'}$ by Lemma \ref{lem:change_order}.  Conversely, $\can^{t'}$ is $(\prec_{t'},\bm)$-unitriangular to $\inj^t$ as well by Lemma \ref{lem:inverse_transition}. Correspondingly, taking the $\prec_{t'}$-unitriangular decomposition of $S$ in $\inj^t$ and applying Lemma
Triangularity Preservation Lemma (Lemma \ref{lem:preserve_triangular}) in $\cT(t')$, we obtain
\begin{align*}
  [X'_k*S]^{t'}&=[X'_k*\sum_{\phi_{t',t}\eta
    \preceq_{t'}\deg^{t'}S}c_{\eta}\inj^t(\eta)]^{t'}\\
&=\sum_{\phi_{t',t}\eta
    \preceq_{t'}\deg^{t'}S}c_{\eta}q^{\alpha_\eta}[X'_k*\inj^t(\eta)]^{t'},
\end{align*}
where $c_{\eta}=1$, $\alpha_\eta=0$, if $\phi_{t',t}\eta=\deg^{t'}S$, and
$c_\eta\in\bm$, $\alpha_\eta\leq 0$ otherwise. It suffices to show that each term
$[X'_k*\inj^t(\eta)]^{t'}$ is $(\prec_{t'},\bm)$-unitriangular to $\can^{t'}$.

We can write
\begin{align*}
  \inj^t(\eta)=[X^{\eta_f}*X^{[\eta_\hk]_+}*X_k^{[\eta_k]_+}*I_k^{[-\eta_k]_+}*I^{[-\eta_\hk]_+}]^t
\end{align*}
where $\eta_\hk$ is obtained from $\pr_n \eta$ by setting the $k$-th
component $\eta_k$ to $0$, $\eta_f\in\Z^{\{n+1,\ldots,m\}}$ (Definition \ref{def:injective_pointed_element}). This element belongs to $\ptSet(t')$ by Lemma \ref{lem:neighbouring_pointed}, namely:
\begin{align*}
  \inj^t(\eta)=[X^{\eta_f}*X^{[\eta_\hk]_+}*X_k^{[\eta_k]_+}*I_k^{[-\eta_k]_+}*I^{[-\eta_\hk]_+}]^{t'}.
\end{align*}

(Case ii-a) If $\eta_k\leq 0$, then $[\eta_k]_+=0$ and $I^t(\eta)$ does not contain any factor $X_k$. Consider the following normalized product in $\cT(t)$
\begin{align*}
  [X'_k*I^t(\eta)]^{t}&=[X^{\eta_f}*X^{[\eta_\hk]_+}*X'_k*I_k^{[-\eta_k]_+}*I^{[-\eta_\hk]_+}]^{t}.
\end{align*}

The leading
degrees $\deg^t X'_k$ and $\deg^t I^t(\eta)$ are sign coherent at the
$k$-th components. The compatibility of degree change follows by Lemma \ref{lem:additivity}:
\begin{align*}
  \deg^{t'}(X'_k*I^t(\eta))=\phi_{t',t}\deg^t(X'_k*I^t(\eta)).
\end{align*}
In addition, $X_k'\in \can^{t'}=\can^t$ is
$(\prec_t,\bm)$-unitriangular to $\inj^t$ and, by Substitution Lemma
(Lemma \ref{lem:substitute_expansion}), $[X'_k*I^t(\eta)]^t$ is $(\prec_t,\bm)$-unitriangular to $\inj^t$. Therefore, we can change the seed from $t$ to $t'$ by
applying Lemma \ref{lem:change_order} and obtain that
$[X'_k*I^t(\eta)]^{t}=[X'_k*I^t(\eta)]^{t'}$ and that $[X'_k*I^t(\eta)]^{t'}$ is $(\prec_{t'},\bm)$-unitriangular to $\can^{t'}$.

(Case ii-b) This last case is nontrivial. If $\eta_k>0$, we have the following normalized product in $\cT(t')$
\begin{align}\label{eq:subtle_product}
  \begin{split}
    [X'_k*I^t(\eta)]^{t'}=&[X^{\eta_f}X^{[\eta_\hk]_+}*X'_k*X_k^{\eta_k}*I^{[-\eta_\hk]_+}]^{t'}\\
    =&[X^{\eta_f}X^{[\eta_\hk]_+}*(X'_k*X_k)*X_k^{\eta_k-1}*I^{[-\eta_\hk]_+}]^{t'}.
\end{split}
  \end{align}

Our crucial ingredient is the following exchange relation of the mutation at vertex $k$ in the quantum torus $\cT(t')$:
\begin{align*}
[X'_k*X_k]^{t'}=&X^{\tb^{+}}+q^{-\frac{\mathbf{d}_k}{2}}X^{\tb^{-}} \\
=&X^{\tb^{+}}+q^{-\frac{\mathbf{d}_k}{2}}X^{\tb^{+}} Y_k',
\end{align*}
where $ \tb^{+}=\sum_{i}[b_{ik}(t)]_+e_i$, $\tb^{-}=\sum_{j}[-b_{jk}(t)]_+e_j$, $Y_k'=Y_k(t')=Y_k(t)^{-1}$.

Notice that the leading term of the pointed element $[X'_k*X_k]^{t'}$ in
$\cT(t')$ is $X^ {\tb^{+}}$ and we have $\deg^{t'}X_k= {\tb^{+}}-e_k$.

By the definition of normalization, \eqref{eq:subtle_product} becomes
\begin{align*}
[X'_k*I^t(\eta)]^{t'}=&[X^{\eta_f}X^{[\eta_\hk]_+}*(X^ {\tb^{+}}+q^{-\frac{\mathbf{d}_k}{2}}X^ {\tb^{+}}
Y_k')*X_k^{\eta_k-1}*I^{[-\eta_\hk]_+}]^{t'}\\
=&[X^{\eta_f}X^{[\eta_\hk]_+}X^ {\tb^{+}}
X_k^{\eta_k-1}*I^{[-\eta_\hk]_+}]^{t'}\\
&+q^{-\frac{\mathbf{d}_k}{2}}q^{\alpha_k}[X^{\eta_f}X^{[\eta_\hk]_+}X^ {\tb^{+}} Y_k' X_k^{\eta_k-1}*I^{[-\eta_\hk]_+}]^{t'}\\
\stackrel{\mathrm{def}}{=}&A+q^{-\frac{\mathbf{d}_k}{2}}q^{\alpha_k}B,
\end{align*}
where $\alpha_k$ is given by
\begin{align*}
\alpha_k=&\Hf\Lambda(t')([\eta_\hk]_+,\deg^{t'}Y_k')+ \Hf\Lambda(t')(\deg^{t'}Y_k',(\eta_k-1)\deg^{t'}X_k)\\
&+\Hf\Lambda(t')(\deg^{t'}Y_k',\deg^{t'}I^{[-\eta_\hk]_+}).
\end{align*}
Since $(B(t'),\Lambda(t'))$ is a compatible pair, one has $\Lambda(t')(e_i,\deg^{t'}Y_k')=-\delta_{ik}\mathbf{d}_k$, $\forall 1\leq i\leq m$. It follows that
\begin{align*}
\alpha_k=&0+ \Hf\Lambda(t')(\deg^{t'}Y_k',(\eta_k-1)( {\tb^{+}}-e_k))+0\\  
=&-(\eta_k-1)\frac{\mathbf{d}_k}{2}\leq 0.
\end{align*}
It remains to show that $A$, $B$ are $(\prec_{t'},\bm)$-unitriangular to $\can^{t'}$.
In $\cT(t)$, the leading degrees of the factors
in both terms $A$, $B$ are sign coherent at the vertex $k$-th components respectively. Thus Lemma \ref{lem:additivity} implies:
\begin{align*}
\deg^{t'}A=\phi_{t',t}\deg^t A,\quad
\deg^{t'}B=\phi_{t',t}\deg^t B.
\end{align*}

Because $[A]^t$ and $[B]^t$ belong to $\inj^t$, they are $(\prec_t,\bm)$-unitriangular to $\can^t$. By Dominance Order Change Lemma (Lemma 
\ref{lem:change_order}), we have $[A]^t=A$, $[B]^t=B$ and they are $(\prec_{t'},\bm)$-unitriangular to $\can^{t'}$.
\end{proof}

  

\subsection{Proofs of theorems}

As a consequence of previous subsections, we can now prove theorems about the existence of the common triangular basis.

\begin{proof}[{Proof of Theorem \ref{thm:induction}}]
For any injective reachable chain $(t'[d])_{d\in\Z}$, if the common triangular basis \wrt $\{t'[d]\}$ exists, Proposition \ref{prop:prove_X} implies that, for any $k\in[1,n]$, $d\in\Z$, it is admissible in direction $\sigma^d k$ \wrt the seed $t'[d]$. Then, Proposition \ref{prop:condition_change_seed} implies that the common triangular basis \wrt to $\{(\mu_k t')[d] |d\in\Z\}$ exists and is compatible with the previous common triangular basis \wrt $\{t'[d]\}$.

We repeatedly apply the above arguments along any mutation sequence starting from the initial injective reachable chain $(t[d])$. The claim follows.
\end{proof}


\begin{Prop}\label{prop:initial_seed_cond}
Let $t$ be any chosen seed and $(t[d])_{d\in\Z}$ an injective-reachable chain.
\begin{enumerate}[(i)]
	\item 
If the triangular basis $\can^t$ \wrt a seed $t$ exists, then the triangular basis $\can^{t[d]}$ \wrt the seed $t[d]$ exists, for any
$d\in\Z$.
 
\item If $\can^t$ is further positive, then so does $\can^{t[d]}$.

\item Let $\overleftarrow{\mu}$ be any mutation sequence and $k$ any integer in $[1,n]$. If the quantum cluster variable $X_k(\overleftarrow{\mu}t)$ belongs to $\can^t$, then the quantum cluster variable $X_{\sigma^d k}(\sigma^d\mu
(t[d]))$ belongs to $\can^{t[d]}$.
\end{enumerate}

\end{Prop}
\begin{proof}
Notice that any seed $t[d]$,
$d\in\Z$, is similar to $t[0]=t$ via an permutation $\sigma^d$ of the unfrozen vertices $[1,n]$ (Section \ref{sec:chain_seed}). The claims follow from Lemma \ref{lem:variation_triangular_basis}.
\end{proof}

The following theorem is a possible reduction of the existence theorem
\ref{thm:induction}, which reduces.

\begin{proof}[{Proof of Theorem \ref{thm:reduction}}]
By Proposition \ref{prop:initial_seed_cond}, the existence of $\can^t$ implies the existence of the
triangular basis $\can^{t[d]}$, $d\in \Z$, which is also positive. We want to show that these
bases are compatible.

Denote the mutation sequence from $t$ to $t[-1]$ by (read from right to left)
 $$\sigma(\Sigma^{-1})=\mu_{i_s}\cdots \mu_{i_2}\mu_{i_1},\ s\in\N.$$
 Then the mutation sequence from $t$ to $t[1]$ is $\Sigma=\mu_{\sigma i_1}\mu_{\sigma i_2}\cdots \mu_{\sigma i_s}$. For any seed $t'=\mu_{i_u}\cdots \mu_{i_2}\mu_{i_1}t$, $1\leq u\leq s$, obtained along the mutation sequence from $t$ to $t[-1]$, Lemma \ref{lem:long_injective_mutation} implies:
$$t'[1]=\mu_{\sigma i_u}\cdots \mu_{\sigma i_1}t[1]\\
=\mu_{\sigma i_{u+1}}\cdots \mu_{\sigma i_s}t.$$
So, when $u<s$, we have $I_{i_{u+1}}(t')=X_{\sigma i_{u+1}}(t'[1])$. Thus the quantum cluster variables $X_{i_{u+1}}(t')$, $I_{i_{u+1}}(t')$ are contained in $\can^t$ by hypothesis. In other words, $\can^t$ is admissible in direction $i_{u+1}$ \wrt the seed $t'$ when $u<s$.

Repeatedly using Proposition \ref{prop:condition_change_seed} along the mutation sequence from $t$ to $t[-1]$, we
obtain that the triangular bases $\can^{t'}$ \wrt the seeds $t'$ obtained along the mutation
sequence exist and are compatible with
$\can^t$. In particular, $\can^{t[-1]}$ and $\can^t$ are compatible. By using Proposition \ref{prop:initial_seed_cond}(iii), we deduce the similar result for
$\can^{t[d]}$, $d\in\Z$. 

Finally, the claim follows from Theorem \ref{thm:induction}. 
\end{proof}

\begin{proof}[Proof of Theorem \ref{thm:change_coefficient}]
	By Lemma \ref{lem:variation_triangular_basis}, the triangular basis $\can^{t'}$ exists and, in addition, it contains the quantum cluster variables obtained along the mutation sequences from $t$ to $t[1]$ and $t[-1]$ respectively. The claim follows from Theorem \ref{thm:reduction}.
\end{proof}

\section{Graded quiver varieties and minuscule modules}
\label{sec:quiver_variety}
In the following, we fix a symmetric generalized Cartan matrix $C$. In this section, we review graded quiver quivers, working in a framework slightly generalized than that of \cite{Nakajima03}. Following the arguments of \cite{Nakajima03}, we compute the
characters of some simple modules (Proposition \ref{prop:KR_module_variant}(i)(ii)). Such calculation would be useful in Section \ref{sec:adaptable_word} and the proof of Theorem \ref{thm:consequence}.

Thanks to this generalized framework, we use graded quiver varieties to give a universal treatment of simply laced Cartan type, including the ADE type extensively studied in previous literature. Readers unfamiliar with graded quiver varieties, quantum affine algebras, or only interested in type $ADE$ might admit the construction in \cite{Nakajima03}\cite{Nakajima04} and skip this section, whose main result will only be used in the proof of Proposition \ref{prop:calculate_variable}.

\subsection{A review of graded quiver varieties}
\label{sec:definition_quiver_variety}

We start with a review of the graded quiver varieties used in
\cite{KimuraQin11}\cite{Qin12}, whose construction follows from that
of \cite{Nakajima01}\cite{Nakajima04}. We choose its grading as in \cite{Qin12}\cite{KimuraQin11}.

\subsubsection*{Graded quiver varieties}
Given any generalized $I\times I$ symmetric Cartan matrix $C=(C_{ij})$, we associate a diagram $\Gamma$ with it such that the set
of vertices is $I$ and $\Gamma$ has $-C_{ij}$ edges between any two different
vertices $i$ and $j$. Define $r=|I|$. We then choose an orientation $\Omega$ and get a quiver $(\Gamma,\Omega)$. For any
arrow $h$ in $\Omega$, we denote its opposite arrow by $\oh$. Let
$\overline{\Omega}$ denote the set of arrows opposite to the arrows in
$\Omega$. Define the $H$ to be the union of $\Omega$ and $\overline{\Omega}$.

Let us choose an acyclic orientation $\Omega$, \ie $(\Gamma,\Omega)$ is acyclic. We follow the torus grading in \cite[Section 4.3]{Qin12}\cite{KimuraQin11}. Choose a map $\xi$ from $I$ to $\set{1,2,\ldots,r}$, whose images
are denoted by $\xi_i$ for $i\in I$, such that
$\xi_i>\xi_j$ whenever there exists a nontrivial path from $i$ to $j$
in the acyclic quiver $(\Gamma,\Omega)$. Define $\epsilon_{ij}\in\set{\pm 1}$ to be the sign of $\xi_i-\xi_j$ and
	\begin{align}\label{eq:grading_dinstiction}
  \xi_{ij}=\frac{\xi_i-\xi_j}{r}.
\end{align}

It follows that $\epsilon_{s(h)t(h)}=1$ if $h\in \Omega$ and
$\epsilon_{s(h)t(h)}=-1$ if $h\in \overline{\Omega}$.

\begin{Eg}
  Consider the acyclic quiver $(\Gamma,\Omega)$ in Figure
  \ref{fig:quiver_A_5}. We have
  $\epsilon_{21}=\epsilon_{23}=\epsilon_{34}=\epsilon_{45}=1$, $\epsilon_{12}=\epsilon_{32}=\epsilon_{43}=\epsilon_{54}=-1$.
\end{Eg}

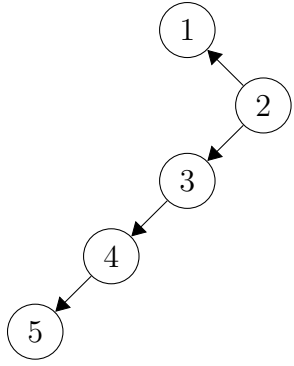
\begin{figure}[htb!]
 \centering
\beginpgfgraphicnamed{fig:quiver_A_5}
\begin{tikzpicture}
\node [shape=circle, draw] (v1) at (1,4) {1};
    \node [shape=circle, draw] (v2) at (2,3) {2};
    \node [shape=circle, draw] (v3) at (1,2) {3};
\node [shape=circle, draw] (v4) at (0,1) {4};
\node [shape=circle, draw] (v5) at (-1,0) {5};

    \draw[-triangle 60] (v2) edge (v1); 
\draw[-triangle 60] (v2) edge (v3);
\draw[-triangle 60] (v3) edge (v4);
\draw[-triangle 60] (v4) edge (v5);
\end{tikzpicture}
\endpgfgraphicnamed
\caption{An acyclic quiver $(\Gamma,\Omega)$ of Cartan type $A_5$}
\label{fig:quiver_A_5}
\end{figure}

Define the set of vertices
\begin{align*}
\sW=I\times 2\Z\\
\sV=I\times (1+2\Z).
\end{align*}
The \emph{$q$-analog of Cartan matrix} $C_q$ is defined to be the linear map
from $\Z^{I\times \Z}$ to $\Z^{I\times \Z}$, such that for any $\eta=(\eta_{i,b})$, we have
\begin{align}
  (C_q \eta)_{i,a}=\eta_{i,a-1}+\eta_{i,a+1}+\sum_{j\neq i}C_{ij}\eta_{j,a+\epsilon_{ij}}.
\end{align}
For any finitely supported $v\in\N^\sV$, $w\in\N^\sW$, we say the pair
$(v,w)$ is \emph{$l$-dominant} if $w-C_qv\geq 0$. 

For any finite
supported $w$, define $C_q^{-1}w$ to be the unique vector such that
$(C_q^{-1}w)_{i,b}=0$ when $b\ll 0$ and $C_q (C_q^{-1}w)=w$. We refer
the reader to Remark \ref{rem:compare_form} for comparing our notations, including $C_q$, with those of  \cite{HernandezLeclerc11}.

\begin{Eg}
  Consider the acyclic quiver $(\Gamma,\Omega)$ in Figure
  \ref{fig:quiver_A_21}. Then we can only choose $\xi_1=3,\xi_2=2,\xi_3=1$. We have
  $\epsilon_{12}=\epsilon_{23}=\epsilon_{13}=1$,
  $\epsilon_{21}=\epsilon_{32}=\epsilon_{31}=-1$, and the following equation
  \begin{align*}
    (C_q \eta)_{2,a}=\eta_{2,a-1}+\eta_{2,a+1}-\eta_{1,a-1}-\eta_{3,a+1}.
  \end{align*}
The component $\eta_{2,a+1}$ contributes to the components of degree $(2,a)$, $(2,a+2)$,
$(1,a+1+\epsilon_{21})$, $(3,a+1+\epsilon_{23})$ via the map $C_q$.
\end{Eg}
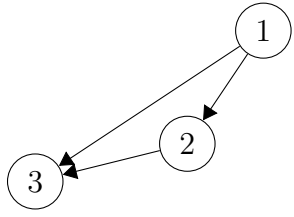
\begin{figure}[htb!]
 \centering
\beginpgfgraphicnamed{fig:quiver_A_21}
\begin{tikzpicture}
\node [shape=circle, draw] (v1) at (2,4) {1};
    \node [shape=circle, draw] (v2) at (1,2.5) {2};
    \node [shape=circle, draw] (v3) at (-1,2) {3};
    \draw[-triangle 60] (v1) edge (v2); 
\draw[-triangle 60] (v2) edge (v3);
\draw[-triangle 60] (v1) edge (v3);
\end{tikzpicture}
\endpgfgraphicnamed
\caption{An acyclic quiver $(\Gamma,\Omega)$ of Cartan type $A_2^{(1)}$}
\label{fig:quiver_A_21}
\end{figure}

For any $d\in \Z$, define the shift operator $[d]$ to be the linear map
from $\Z^{I\times \Z}$ to $\Z^{I\times \Z}$ such that, for any
$\eta\in \Z^{I\times \Z}$, we have
\begin{align*}
(\eta[d])_{i,a}=\eta_{i,a+d}.
\end{align*}

Take any finitely supported dimension vector $v$, $v'$, in $\N^\sV$ and $w$
in $\N^{\sW}$. Denote the associated graded $\C$-vector spaces by $V=\oplus V_{i,b}$, $V'=\oplus V'_{i,b}$ and $W=\oplus W_{i,a}$ respectively. Consider the following vector spaces
\begin{align*}
L(w,v)=&\oplus_{(i,a)\in \sW}\Hom(W_{i,a},V_{i,a-1}),\\
L(v,w)=&\oplus_{(i,b)\in \sV}\Hom(V_{i,b},W_{i,b-1}),\\
 E(\Omega;v,v')=&(\oplus_{h\in H, b\in 1+2\Z
 }\Hom(V_{s(h),b},V'_{t(h),b-1+\epsilon_{s(h)t(h)}}))\\
=&(\oplus_{h\in\Omega, b\in 1+2\Z
 }\Hom(V_{s(h),b},V'_{t(h),b}))
\\&\oplus (\oplus_{\oh\in\overline{\Omega}, b\in 1+2\Z
}\Hom(V_{s(\oh),b},V'_{t(\oh),b-2})).
\end{align*}
Define the vector space $\Rep(\Omega;v,w)=E(\Omega;v,v)\oplus
L(w,v)\oplus L(v,w)$, whose coordinates will be denoted by
\begin{align*}
  (B,\imath,\jmath)=&((B_h)_{h\in H},\imath,\jmath)\\
=&((B_h)_{h\in \Omega},(B_\oh)_{\oh\in\overline{\Omega}},(\imath_i),(\jmath_i))\\
=&((\oplus_bB_{h,b})_h,(\oplus_bB_{\oh,b})_{\oh},(\oplus_a\imath_{i,a}),(\oplus_b
\jmath_{i,b})).
\end{align*}
Notice that $\Rep(\Omega;v,w)$ can be naturally viewed as the vector
space of $(v,w)$-dimensional representations of a quiver with vertices
$\sV\sqcup \sW$. We denote this quiver by $\tilde{\Gamma}_\Omega$ can
call it the \emph{framed repetition quiver} associated with the acyclic
orientation $\Omega$. 

Different acyclic orientations produce isomorphic
framed repetition quivers. Choose and fix such a quiver $\tilde{\Gamma}_\Omega$
from now on.

The \emph{analog of the moment map} $\mu$ is the linear map from
$\Rep(\Omega;v,w)$ to $L(v,v[-1])$ given by
\begin{align*}
  \mu(B,\imath,\jmath)=&\oplus_{(i,b)\in\sV}\mu(B,\imath,\jmath)_{i,b}\\
=&\oplus_{(i,b)\in\sV}(\sum_{h\in\Omega,t(h)=i}b_{h,b-2}b_{\oh,b}-\sum_{h\in\Omega,s(h)=i}b_{\oh,b}b_{h,b}+\imath_{i,b-1}\jmath_{i,b}).
\end{align*}

\begin{Eg}
  Let the acyclic quiver $(\Gamma,\Omega)$ be given by Figure
  \ref{fig:quiver_A_5}. Part of the framed repetition quiver $\tilde{\Gamma}_\Omega$ is shown in Figure \ref{fig:quiver_variety_A_5}
\end{Eg}

\begin{figure}[htb!]
 \centering
\beginpgfgraphicnamed{fig:quiver_variety_A_5}
\begin{tikzpicture}

\node [draw] (w16) at (2,8) {1,a+6}; 
\node [ draw] (w26) at (4,6) {2,a+6};
    \node [ draw] (w36) at (2,4) {3,a+6};
\node [ draw] (w46) at (0,2) {4,a+6};
\node [ draw] (w56) at (-2,0) {5,a+6};

\node [draw] (w14) at (-2,8) {1,a+4}; 
\node [draw] (w24) at (0,6) {2,a+4}; 
\node [draw] (w34) at (-2,4) {3,a+4}; 
\node [draw] (w44) at (-4,2) {4,a+4};

\node [shape=circle,draw] (v15) at (0,8) {1,a+5}; 
\node [shape=circle,draw] (v25) at (2,6) {2,a+5}; 
\node [shape=circle,draw] (v35) at (0,4) {3,a+5}; 
\node [shape=circle,draw] (v45) at (-2,2) {4,a+5}; 
\node [shape=circle,draw] (v55) at (-4,0) {5,a+5};

\node [draw] (w12) at (-6,8) {1,a+2}; 
\node [draw] (w22) at (-4,6) {2,a+2}; 
\node [draw] (w32) at (-6,4) {3,a+2};

\node [shape=circle, draw] (v13) at (-4,8) {1,a+3}; 
\node [shape=circle,draw] (v23) at (-2,6) {2,a+3}; 
\node [shape=circle,draw] (v33) at (-4,4) {3,a+3}; 

\node [draw] (w20) at (-8,6) {2,a}; 

\node [shape=circle, draw] (v21) at (-6,6) {2,a+1};

\draw[-triangle 60](v25)edge (v15);
\draw[-triangle 60](v25)edge (v35);
\draw[-triangle 60](v35)edge (v45);

\draw[-triangle 60](v23)edge (v13);
\draw[-triangle 60](v23)edge (v33);

\draw[-triangle 60](v15)edge (v23);
\draw[-triangle 60](v35)edge (v23);
\draw[-triangle 60](v45)edge (v33);

\draw[-triangle 60](v13)edge (v21);
\draw[-triangle 60](v33)edge (v21);

\draw[-triangle 60](w16)edge (v15);
\draw[-triangle 60](v15)edge node[above]{$\jmath$} (w14);
\draw[-triangle 60](w14)edge node[above]{$\imath$} (v13);
\draw[-triangle 60](v13)edge (w12);

\draw[-triangle 60](w26)edge (v25);
\draw[-triangle 60](v25)edge (w24);
\draw[-triangle 60](w24)edge (v23);
\draw[-triangle 60](v23)edge (w22);
\draw[-triangle 60](w22)edge (v21);
\draw[-triangle 60](v21)edge (w20);

\draw[-triangle 60](w36)edge (v35);
\draw[-triangle 60](v35)edge (w34);
\draw[-triangle 60](w34)edge (v33);
\draw[-triangle 60](v33)edge (w32);

\draw[-triangle 60](w46)edge (v45);
\draw[-triangle 60](v45)edge  (w44);

\draw[-triangle 60](w56)edge  (v55);
\draw[-triangle 60](v45)edge  (v55);

\end{tikzpicture}
\endpgfgraphicnamed
\caption{Part of the framed repetition quiver $\tilde{\Gamma}_\Omega$}
\label{fig:quiver_variety_A_5}
\end{figure}
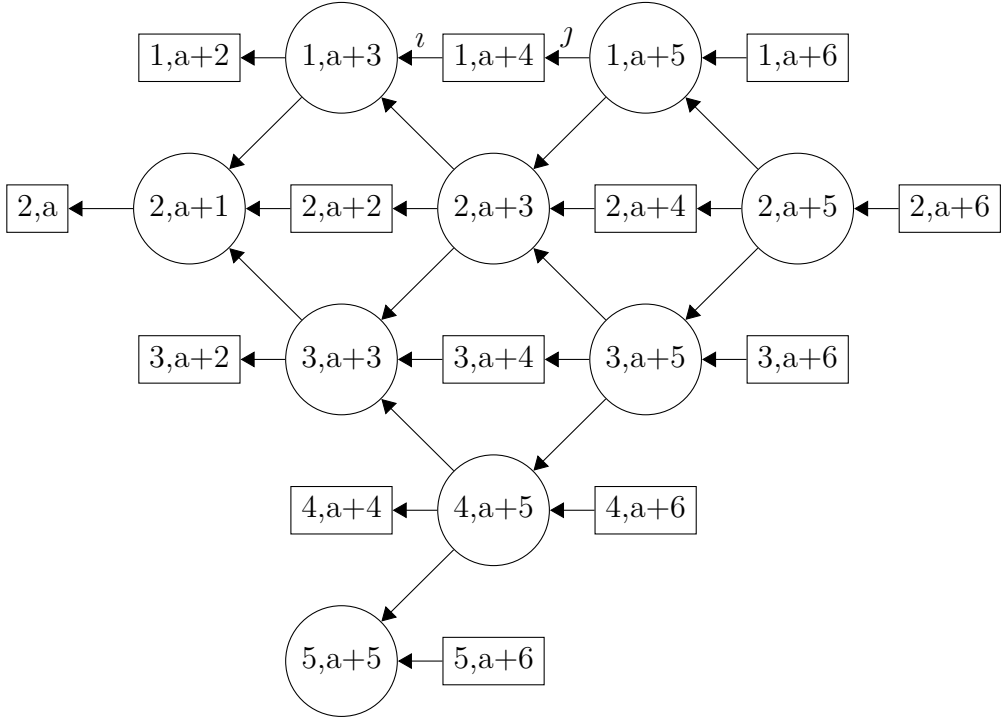

The variety $\mu^{-1}(0)$ has the natural action by the reductive
algebraic group $G_v=\prod_{(i,b)\in \sV}GL(V_{i,b})$ such that, for any given
$g=(g_{i,b})$, the action is given by
\begin{align*}
  g(\imath)&=g\imath,\\
g(\jmath)&=\jmath g^{-1},\\
g(b_{h})&=g_{t(h)}b_{h}g_{s(h)}^{-1},\ \forall h\in H.
\end{align*}
Define the character $\chi_v$ of $G_v$ by $\chi_v(g)=\prod_{i,b}(\det g_{i,b})^{-1}$. Fix this character from now on.

Define $\grProjQuot(v,w)$ to be the \emph{geometric invariant quotient} of
the $G_v$-variety $\mu^{-1}(0)$ associated with $\chi_v$ in the sense of Mumford, \cf \cite{mumford1994geometric}. Denote the
categorical quotient by $\grAffQuot(v,w)$ and the natural projective
morphism from $\grProjQuot(v,w)$ to $\grAffQuot(v,w)$ by $\pi$. Denote
the fiber $\pi^{-1}(0)$ by $\grLag(v,w)$. The varieties
$\grProjQuot(v,w)$, $\grAffQuot(v,w)$ and $\grLag(v,w)$ are called \emph{graded quiver varieties}.

For any $v'<v$, there is a natural embedding of $\grAffQuot(v',w)$
into $\grAffQuot(v,w)$. Moreover, by our construction,
$\grAffQuot(v,w)$ stabilizes when $v$ is large enough. Denote the
union $\cup_v\grAffQuot(v,w)$ by $\grAffQuot(w)$.

\subsubsection*{Grothendieck ring and characters}

We have intersection cohomology sheaves $IC(v,w)$, where $w-C_q v\geq 0$, for
closed subvarieties $\grAffQuot(v,w)$ of $\grAffQuot(w)$. Define the
operation of $\Z[t^\pm]$ on the Grothendieck ring of the derived
category of constructible sheaves over $\C$-vectors spaces on $\grAffQuot(w)$ such that $t$ acts as the shift functor. Let $K_w$
denote the free
abelian group generated by the classes of the sheaves $IC(v,w)$ over
$\Z[t^\pm]$ inside the Grothendieck group. 

Let $R_w=\Hom_{\Z[t^\pm]}(K_w,\Z[t^\pm])$ denote the dual of $K_w$. Define $\quotKGp$ to be the
subspace of $\prod_{w}R_w$ which consists of the basis $\set{S(w)}$,
such that the value of $S(w)$ on $IC(v,w')$ is $\delta_{w,w'-C_q v}$. Inspired by \cite{Nakajima01}, we call $S(w)$ the simple modules.

Define the quantum torus $\YTorus $ to be the Laurent polynomial ring $\Z[t^\pm][Y_{i,a}^\pm]_{(i,a)\in
  \sW }$ equipped with the twisted product $*$ defined by\footnote{Our
  quantum torus is the same as the quantum torus in \cite[(3.21)]{Nakajima09}: identify our form $-\cE$ and operator $[1]$ with the form $\varepsilon$ and the operator $q$ in \cite{Nakajima09}. It differ from that of \cite{Nakajima03} by taking $t$ to $t^{-1}$, cf. Remark \ref{rem:compare_form}.} 
  
   \begin{align}
  \label{eq:twist_prod}
  \begin{split}
    Y^{w^1}*Y^{w^2}=&t^{-\cE(m^1,m^2)}Y^{w^1+w^2},\\
    \cE(m^1,m^2)=&-w^1[1]\cdot C_q^{-1}w^2+w^2[1]\cdot C_q^{-1}w^1,
  \end{split}
\end{align}
for any $w^1,w^2\in\N^\sW$,\cf \cite{KimuraQin11}\cite[(40)]{Qin12}.

As before, we let $[\ ]$ denote the normalization by $t$-factors in $\YTorus $:
\begin{align}\label{eq:Y_torus_normalization}
[t^sY^w]\stackrel{\mathrm{def}}{=}Y^w,\ \forall w\in\N^\sW,s\in \Z.
\end{align}

 There exists a natural multiplication on $\quotKGp$, which is derived from the geometric restriction functor and
the Euler twist $\cE(\ ,\ )$, \cf \cite{Nakajima09}. Let $\quotKGp$
denote the Grothendieck ring equipped
with this multiplication. Furthermore, the structure
constants of its natural basis $\set{S(w)}$ are given by
\begin{align*}
  S(w^1)\otimes S(w^2)=\sum_{v\geq 0, w^1+w^2-C_q v\geq 0}a_v(t)S(w^1+w^2-C_q v),
\end{align*}
such that $a_v(t)\in\N[t^\pm]$, $a_0=1$. Let $\commQuotKGp$ denote the
specialization of $\quotKGp $ by taking $t=1$.

Let $\qtChar$ denote the $t$-analog of the $q$-characters
($q,t$-character for short)
\begin{align}
  \label{eq:qtChar}
  \qtChar :\quotKGp \ra \YTorus =\Z[Y_{i,a}^\pm]_{(i,a)\in I\times \Z}.
\end{align}
$\qtChar $ is an injective ring homomorphism with respect to the
twisted product $*$, \cf
\cite{Nakajima04}\cite{VaragnoloVasserot03}\cite{Hernandez02}. We refer the
reader to \cite{Qin12} for its precise definition in our convention. In particular, define the Laurent monomial introduced by \cite{FrenkelReshetikhin99}
\begin{align*}
  A_{i,b}=Y_{i,b-1}Y_{i,b+1}\prod_{j\neq i}Y_{j,b+\epsilon_{ij}}^{C_{ij}}.
\end{align*}
Then the character of any $S(w)$ is a formal series
\begin{align*}
  \qtChar  S(w)=Y^w(1+\sum_{v\neq 0} b_v(t) A^{-v}),
\end{align*}
such that $b_v(t)\in \N[t^\pm]$. Any monomial $Y^wA^v$ is said to be $l$-dominant if the pair $(v,w)$ is. We say the leading term of $\qtChar  S(w)$
is $Y^w$. The simple module $S(w)$ is called \emph{minuscule} if $Y^w$
is the only one $l$-dominant monomial appearing in its character.

\subsubsection*{Truncated characters}
We consider truncated $q,t$-characters which were first introduced in \cite{HernandezLeclerc09}.

Take a multi-degree $\uc=(c_i)_{i\in I}\in (2\Z)^I$. When $c_i$ takes a constant value $c\in 2\Z$ for any $i\in I$, we
simply denote $\uc=c$. We say $\uc$ is \emph{$\Omega'$-adaptable} (or
adaptable for short), if the
full subquiver of the fixed framed repetition quiver $\tilde{\Gamma}_\Omega$ on the vertices
$(i,c_i-1)$, $i\in I$, is the acyclic quiver
$(\Gamma,\Omega')$. We shall always assume $\uc$ to be adaptable for some $\Omega'$ from
now on.

Define the truncated character $\qtChar_{\leq \uc}$ to be 
$\qtChar$ composed by the projection from $\YTorus$ to its the quotient ring $\YTorus_{\leq \uc}=\Z[t^\pm][Y_{i,a}^\pm]_{(i,a)\in I\times \Z,a\leq
  c_i}$.






\subsection{A review of Kirillov-Reshetikhin modules}
\label{sec:KR_module}
In this section, we review Kirillov-Reshetikhin modules in the language of
graded quiver varieties, where we use the arguments of
\cite{Nakajima03} in our setting.

For any $i\in I$, $k\in \N$, $a\in 2\Z$, define the
dimension vector
\begin{align}\label{eq:KR_dimension}
  w^{(i)}_{k,a}=e_{i,a}+e_{i,a+2}+\ldots+e_{i,a+2(k-1)}.
\end{align}
Inspired
by \cite{Nakajima03}, we call the simple
modules $S(w^{(i)}_{k,a})$ the \emph{Kirillov-Reshetikhin modules} and
denote it by $W^{(i)}_{k,a}$. They were introduced in
\cite{KirillovReshetikhin90} for quantum affine algebras.

\begin{Rem}\label{rem:shift_degree}
  When taking the quiver $(\Gamma,\Omega)$ to be
  bipartite and choosing any $j\in I$, we can identify our graded
  framed quiver with the corresponding quiver in \cite{Nakajima03} by identifying our vertex $(i,s-\epsilon_{ij})$
  with the vertex $(i,s)$ in \cite{Nakajima03}. We can then translate the results of \cite{Nakajima03} to Section \ref{sec:KR_module}.
\end{Rem}

Let us proceed with some important techniques in calculating the characters of Kirillov-Reshetikhin modules, following \cite{Nakajima03}. For any Laurent monomial $m\neq 1$ in $\YTorus$ and $j\in I$, define
  \begin{align*}
    r_j(m)=\max \set{s\in \Q|m\mathrm{\ contains\ the\ factor\ }Y_{i,s+\xi_{ji}}\mathrm{\ for\
        some\ }i\in I}.
  \end{align*}
Here we use $\xi_{ji}$ (defined in \eqref{eq:grading_dinstiction}) to distinguish the grading of the vertices in
$(\Gamma,\Omega)$. It follows that $r_j(m)+\xi_{ji}=r_i(m)$, for any
$i,j$, and that the non-positive criterion in the following definition is independent of the choice of $j$.

\begin{Def}[right negative \cite{frenkel2001combinatorics}]\label{def:right_negative}
The monomial $m$ is said to be right negative if there exists some $j\in I$ such that the factors
$Y_{i,r_j(m)+\xi_{ji}}$ has non-positive powers for any $i\in I$.
\end{Def}

\begin{Eg}\label{eg:dominant_monomial}
  Let us take the quiver $(\Gamma,\Omega)$ in Figure
  \ref{fig:quiver_A_21}. Then the monomial $Y^{w^{(2)}_{k,0}}=Y_{2,0}Y_{2,2}\cdots Y_{2,2k-2}$ is
  $l$-dominant with $r_2(Y^{w^{(2)}_{k,0}})=2k-2$,
  $r_1(Y^{w^{(2)}_{k,0}})=2k-2-\frac{1}{3}$, $r_3(Y^{w^{2}_{k,0}})=2k-2+\frac{1}{3}$. 

For any $1\leq s\leq
  k$, the monomial
  $m_s=Y^{w^{(2)}_{k,0}}A^{-1}_{2,2k-2s+1}\cdots A^{-1}_{2,2k-3}A^{-2}_{2,2k-1}$
  is right negative with $r_2(m_s)=2k$.
\end{Eg}

Notice that the $l$-dominant monomials are not right negative. The
product of two right negative monomials is still right negative.

In \cite{Nakajima03}, Nakajima computed truncated $q,t$-characters of Kirillov-Reshetikhin
modules by studying right negative monomials and using combinatorial
properties of $q,t$-characters defined over graded quiver varieties. His arguments remain effective for our
graded quiver varieties and imply the following results.
\begin{Thm}[\cite{Nakajima03}]\label{thm:KR_module}
  (i) All monomials in $\qtChar W^{(i)}_{k,a}$ are right negative
  except its leading term $Y^{w^{(i)}_{k,a}}$.

(ii) Let $m$ be a right negative monomial appearing in the truncated
$q,t$-character $\qtChar_{\leq 2k}
W^{(i)}_{k,a}$, then we have
\begin{align}
  \label{eq:right_negative_monomial}
  \begin{split}
    m=&\prod^{k-1}_{t=0}Y_{i,a+2t}\prod_{t=s+1}^{k}A^{-1}_{i,a+2t-1}\\
=&\prod^{s-1}_{t=0}Y_{i,a+2t}\prod_{t=s+1}^{k}(Y^{-1}_{i,a+2t}\prod_{j\neq
i}Y^{-C_{ij}}_{j,a+2t-1+\epsilon_{ij}})\\
    =&Y_{i,a}\cdots Y_{i,a+2(s-1)}Y_{i,a+2(s+1)}^{-1}\cdots
    Y_{i,a+2k}^{-1}\prod_{j\neq i}(Y_{j,a+2s+1+\epsilon_{ij}}\cdots
    Y_{j,a+2k-1+\xi_{ij}})^{-C_{ij}},
  \end{split}
\end{align}
where $0\leq s\leq  k-1$.

(iii) We have the following equality in the quantum torus $\YTorus$;
\begin{align}
  \label{eq:exact_sequence}
  [
  \qtChar W^{(i)}_{1,a+2k}*\qtChar W^{(i)}_{k,a}]=\qtChar W^{(i)}_{k+1,a}+t^{-1}\qtChar S(w^{(i)}_{k-1,a}-\sum_{j\neq
  i}C_{ij}w^{(j)}_{1,a+2k-1+\epsilon_{ij}}).
\end{align}
\end{Thm}
Recall that $[\ ]$ denote the normalization (cf. \eqref{eq:Y_torus_normalization}).

Notice that $m$ becomes the leading term $Y^{w^{(i)}_{k,a}}$ of $\qtChar W^{(i)}_{k,a}$ if
we take $s=k$. The following Lemma was used in the proof of Theorem \ref{thm:KR_module}.
\begin{Lem}[\cite{Nakajima03}]\label{lem:right_negative_product}
  Take two monomials $m$ and $m'$ from $\qtChar W^{(i)}_{k,a}$ and
  $\qtChar W^{(i)}_{1,a+2k}$ respectively. If the product $mm'$ is not
  right negative, then the following claims are true.

(i) The monomials take the form
\begin{align}
  m'=&Y_{i,a+2k}\\
m=&Y_{i,a}\cdots
Y_{i,a+2k-2}\cdot\prod_{t=s+1}^{k}(A_{i,a+2t-1}^{-1}),\ 0\leq s\leq k.
\end{align}

(ii) The coefficient of the monomial $mm'$ in $[\qtChar W^{(i)}_{1,a+2k}*\qtChar
W^{(i)}_{k,a}]$ is $1$ if $s=k$ and $t^{-1}$ otherwise.
\end{Lem}

As a consequence of the properties of Kirillov-Reshetikhin modules in Theorem
\ref{thm:KR_module}, we deduce the following $T$-systems (such formula appear as \cite[Proposition 5.6]{HernandezLeclerc11} after the bar involution).
\begin{Thm}[\cite{Nakajima03}]\label{thm:T_system}
  The following equation holds in the quantum torus $\YTorus$:
  \begin{align}
    \label{eq:T_system}
    [\qtChar W^{(i)}_{k,a+2}*\qtChar W^{(i)}_{k,a}]=[\qtChar W^{(i)}_{k-1,a+2}*\qtChar
    W^{(i)}_{k+1,a} ]+t^{-1}[\prod_{j\neq i} (\qtChar
    W^{(j)}_{k,a+1+\epsilon_{ij}})^{-C_{ij}}].
  \end{align}
Moreover, $ W^{(i)}_{k-1,a+2}\otimes W^{(i)}_{k+1,a}$ agrees with the
simple module $S(w^{(i)}_{k+1,a}+ w^{(i)}_{k-1,a+2})$ up to
$t$-power in $\quotKGp$.
\end{Thm}

\subsection{A different deformation}
\label{sec:twist}

In order to compare Grothendieck rings with quantum groups, we will need a different
$t$-deformation (introduced in
\cite{Hernandez02} as an algebraic approach to $q,t$-characters) and
work with the corresponding extended Grothendieck ring
$\extQuotKGp=\quotKGp\otimes \Z[t^{\pm\Hf}]$ as in
\cite{HernandezLeclerc11}. To be more precise, as in
\cite[Section 6.1]{KimuraQin11}, we define the quantum torus
$\YTorus^H_{t^\Hf}$ as the Laurent polynomial ring $\Z[t^{\pm\Hf}][Y_{i,a}^\pm]_{(i,a)\in
  \sW }$ equipped with the twisted product $*$ defined by
\begin{align}
  \label{eq:new_twist_prod}
  \begin{split}
    Y^{w^1}*Y^{w^2}=&t^{\Hf\cN(m^1,m^2)}Y^{w^1+w^2},\\
    \cN(m^1,m^2)=&w^1[1]\cdot C_q^{-1}w^2-w^2[1]\cdot
    C_q^{-1}w^1\\
    &\quad -w^1[-1]\cdot C_q^{-1}w^2+w^2[-1]\cdot C_q^{-1}w^1,
  \end{split}
\end{align}
for any $w^1,w^2\in\N^\sW$,\cf \cite[Section
6]{KimuraQin11}. Replacing the Euler form $-\cE(\ ,\ )$ in the construction of
the twisted multiplication in $\cR_{t^\Hf}=\cR_t\otimes \Z[t^{\pm\Hf}] $
by $\Hf\cN(\ ,\ )$, we obtain the extended 
Grothendieck ring $\cR_{t^\Hf}^H$, which will be simply denoted by $\extQuotKGp$. Denote the corresponding
$q,t$-character from $\extQuotKGp$ to $\YTorus^H_{t^\Hf}$ by
$\qtChar^H$. 

\begin{Rem}\label{rem:compare_form}
  The construction of $\cN(\ ,\ )$ in \eqref{eq:new_twist_prod} as a
  variant of $\cE(\ ,\ )$ is inspired by
  \cite[(6)]{HernandezLeclerc11}.


In \cite{HernandezLeclerc11}, the $q$-analog of Cartan matrix is defined to be $$C^{HL}_q(z)=(z+z^{-1})I-A=\sum_{m\in \Z} C^{HL}_q(m)z^m,$$
where $A=C-2I$ is the adjacent matrix and $z$ an indeterminate. Its inverse is denoted by
$$\tilde{C}(z)=\sum_{k\geq 0}(z+z^{-1})^{-k-1}A^k=\sum_{m\geq 1}\tilde{C}(m)Z^m.$$

Fix any chosen $j\in I$. If we identify our vertices $(i,b-\epsilon_{ij})$, $\forall i\in I,b\in\Z$ with the vertices $(i,b)$ in \cite{HernandezLeclerc11},  
then our $C_q$ and $C_q^{-1}$ agree with theirs:
  \begin{align*}
(C_q^{HL})_{ij}(b-a)&=  C_q e_{j,a}\cdot  e_{i,b-\epsilon_{ij}}=&(C_q e_{j,a})_{i,b-\epsilon_{ij}}
=\left\{
      \begin{array}{ll}
        1&\mathrm{if}\ i=j,b=a\\
        C_{ij}&\mathrm{if}\ i\neq j,b=a\\
        0&\mathrm{else}
      \end{array}\right.  
  \\
 \tilde{C}_{ij}(b-a)&=C_q^{-1}e_{j,a}\cdot e_{i,b-\epsilon_{ij}}.
  \end{align*}
For example, the expression $$\tilde{C}_{21}(z)=z^2-z^8+z^{12}-z^{18}+\cdots$$ in \cite[Example 2.2]{HernandezLeclerc11} corresponds to our expression $$(C_q^{-1}e_{1,a})_{2,\Z}=e_{2,a+2-\epsilon_{21}}-e_{2,a+8-\epsilon_{21}}+e_{2,a+12-\epsilon_{21}}-e_{2,a+18-\epsilon_{21}}+\cdots.$$  
  
Consequently, we have the following translations of bilinear forms:
\begin{align*}
  \cN(e_{i,p-\epsilon_{ij}},e_{j,s})&=\tilde{C}_{ij}(p-1-s)-\tilde{C}_{ji}(s-1-p)-\tilde{C}_{ij}(p+1-s)+\tilde{C}_{ji}(s+1-p),\\
  2\cE(e_{i,p-\epsilon_{ij}},e_{j,s})&=-2\tilde{C}_{ij}(p-1-s)+2\tilde{C}_{ji}(s-1-p),
\end{align*}
which are the same as the bilinear forms in \cite[(6)(10)]{HernandezLeclerc11} because the matrix $\tilde{C}$ is symmetric.
\end{Rem}

 \begin{Lem}\label{lem:compatible_twist}
For any $c\in 2\N$, $d,d'\in \N$, and $i\in I$, we have
   \begin{align*}
     -\cE(Y_{i,c-2d}\cdots Y_{i,c-2}Y_{i,c},A_{j,c-1-2d'}^{-1})=&\left\{
       \begin{array}{ll}
         -1&\mathrm{if\ }i=j,d'=d\\
       0&\mathrm{else}
\end{array}
\right.\\
     \cN(Y_{i,c-2d}\cdots Y_{i,c-2}Y_{i,c},A_{j,c-1-2d'}^{-1})=&\left\{
       \begin{array}{ll}
         -2&\mathrm{if\ }i=j,d'=d\\
       0&\mathrm{else}
\end{array}
\right. .
   \end{align*}
 \end{Lem}
 \begin{proof}
  With the help of \cite[Lemma 4.3]{Qin12}, for any $(i,a)\in \sW$ and
  $(j,b)\in \sV$, we make the following
  computation.
  \begin{align}\label{eq:calculate_Euler_form}
    \begin{split}
      -\cE(e_{i,a},-C_q e_{j,b})&=-e_{i,a}[1]\cdot C_q^{-1}C_q
      e_{j,b}+C_q e_{j,b}[1]\cdot C_q^{-1}e_{i,a}\\
      &=-e_{i,a}[1]\cdot e_{j,b}+e_{j,b}[1]\cdot C_qC_q^{-1}e_{i,a}\\
      &=-e_{i,a-1}\cdot e_{j,b}+e_{j,b-1}\cdot e_{i,a}\\
      &=\left\{
        \begin{array}{cc}
          1&i=j,\ b=a+1\\
          -1&i=j,\ b=a-1\\
          0&\mathrm{else}
        \end{array}
      \right.
    \end{split}
  \end{align}

  \begin{align}\label{eq:calculate_Euler_form_new}
    \begin{split}
      \cN(e_{i,a},-C_q e_{j,b})&=-\cE(e_{i,a},-C_q e_{j,b})\\
      &\qquad -e_{i,a}[-1]\cdot C_q^{-1}(-C_q
      e_{j,b})+(-C_qe_{j,b})[-1]\cdot C_q^{-1}e_{i,a}\\
      &=-\cE(e_{i,a},-C_q e_{j,b})+e_{i,a+1}\cdot e_{j,b}-e_{j,b+1}\cdot e_{i,a}\\
      &=\left\{
        \begin{array}{cc}
          2&i=j,\ b=a+1\\
          -2&i=j,\ b=a-1\\
          0&\mathrm{else}
        \end{array}
      \right.
    \end{split}
\end{align}
The desired claim follows.
 \end{proof}

\subsection{Computation of characters}
\label{sec:consequence_KR_module}

We compute $q,t$-characters of some simple modules, which will be useful
in the proof of Theorem \ref{thm:consequence}.

Recall that the character $\qtChar$ is a ring homomorphism from the Grothendieck ring $\quotKGp$ to the quantum torus $\YTorus =\Z[Y_{i,a}^\pm]_{(i,a)\in I\times \Z}$, such that the character of a simple $S(w)$ takes the form
$$
\qtChar S(w)=Y^w(1+\sum_{v\neq 0} b_v(t)A^{-v}),\ b_v(t)\in\N[t^\pm].
$$
In Section \ref{sec:KR_module}, we have seen that the character of the Kirillov-Reshetikhin modules $W^{(i)}_{k,a}$, $k\in\N$, $i\in I$, $a\in\Z$, are given by
\begin{align*}
\qtChar W^{(i)}_{k,a}=Y_{i,a}Y_{i,a+2}\cdots Y_{i,a+2k-2}(&1+A_{i,a+2k-1}^{-1}+A_{i,a+2k-3}^{-1}A_{i,a+2k-1}^{-1}\\
&+\cdots +A_{i,a+1}^{-1}A_{i,a+2k-3}^{-1}A_{i,a+2k-1}^{-1}),
\end{align*}
where the leading term is $l$-dominant (namely, the multiplicities of all factors $Y_{j,b}$ are non-negative), and the others are right negative in the sense of Definition \ref{def:right_negative} (namely, there exists a factor $Y_{j,b}$ with highest $b$ and negative multiplicity) .

Recall that we have calculated the Euler form $\cE$ in Lemma
\ref{lem:compatible_twist}. Inspired by Lemma
\ref{lem:right_negative_product}, we consider the following
product. Results of this type was found in \cite{fourier2014schur}.

\begin{Lem}\label{lem:right_negative_product_variant}
Fix integers $h,k\in\N$ such that $k\geq 2$, $k>h\geq 1$. Take monomials $m$ and $m'$ from $\qtChar W^{(i)}_{k,a}$ and
  $\qtChar W^{(i)}_{h,a+2(k-h+1)}$ respectively. If the product $mm'$ is
  not right negative, then the following claims are true.

(i) The monomials take the form
\begin{align}
  m'=&Y_{i,a+2(k-h+1)}\cdots Y_{i,a+2k}\\
m=& Y_{i,a}\cdots Y_{i,a+2(k-1)}\cdot\prod_{t=s+1}^{k}(A_{i,a+2t-1}^{-1}),\ 0\leq s\leq k.
\end{align}

In particular, $mm'$ is $l$-dominant if and only if $(k-h)\leq s\leq k$

(ii) The coefficient of $mm'$ in $[\qtChar
W^{(i)}_{h,a+2(k-h+1)}*\qtChar W^{(i)}_{k,a}]$ is $1$ if $(k-h)<s\leq k$ and $t^{-1}$ if $0\leq
s\leq (k-h)$.
\end{Lem}
Notice that, when $s=k$, $m$ is the leading term in $\qtChar W^{(i)}_{k,a}$.
\begin{Eg}
We consider the vectors $w\in\N^\sW$ with support on
Figure \ref{fig:type_A_quantum_affine}.

Take $k=3,h=2, i=2$. Consider the following dimension vectors
\begin{align*}
w^{(2)}_{k,a}=&e_{2,a}+e_{2,a+2}+e_{2,a+4},\\
w^{(2)}_{h,a+2(k-h+1)}=&e_{2,a+4}+e_{2,a+6},\\
w^{(2)}_{k+1,a}=&e_{2,a}+e_{2,a+2}+e_{2,a+4}+e_{2,a+6},\\
w^{(2)}_{k-h,a+2(k-h+1)}=&e_{2,a+4},\\
w^{(1)}_{h,a+2(k-h)+1+\epsilon_{21}}=&e_{1,a+4}+e_{1,a+6},\\
w^{(3)}_{h,a+2(k-h)+1+\epsilon_{23}}=&e_{3,a+4}+e_{3,a+6}.
\end{align*}

We compute the truncated $q,t$-characters of the Kirillov-Reshetikhin
modules
\begin{align*}
\qtChar_{\leq a+6} W^{(2)}_{k,a}=&Y_{2,a}Y_{2,a+2}Y_{2,a+4}(1+A^{-1}_{2,a+5}+A^{-1}_{2,a+3}A^{-1}_{2,a+5}+A^{-1}_{2,a+1}A^{-1}_{2,a+3}A^{-1}_{2,a+5}),\\
\qtChar_{\leq a+6} W^{(2)}_{h,a+2(k-h+1)}=&Y_{2,a+4}Y_{2,a+6},\\
\qtChar_{\leq a+6} W^{(2)}_{k+1,a}=&Y_{2,a}Y_{2,a+2}Y_{2,a+4}Y_{2,a+6},\\
\qtChar_{\leq a+6} W^{(2)}_{k-h,a+2(k-h+1)}=&Y_{2,a+4}(1+A^{-1}_{2,a+5}).
\end{align*}
For simplicity, specialize $t$ to $1$. It is easy to compute the following difference
\begin{align*}
\qtChar_{\leq a+6}W^{(2)}_{k,a}*\qtChar_{\leq a+6}W^{(2)}_{h,a+2(k-h+1)}-&\qtChar_{\leq a+6}W^{(2)}_{k+1,a}*\qtChar_{\leq a+6}W^{(2)}_{k-h,a+2(k-h+1)}\\
&=Y_{2,a}Y_{1,a+4}Y_{1,a+6}Y_{3,a+4}Y_{3,a+6}(1+A^{-1}_{2,a+1}).
\end{align*}
We observe that the difference has only one $l$-dominant
monomial. In fact, we will show it is the truncated character of the
minuscule module $W^{(2)}_{k-h,h,a}$ in Proposition \ref{prop:KR_module_variant}.
\end{Eg}


\begin{figure}[htb!]
 \centering
\beginpgfgraphicnamed{fig:type_A_quantum_affine}
  \begin{tikzpicture}
  \node [ draw] (v24) at (2,3) {2,a+6};
    \node [ draw] (v34) at (1,2) {3,a+6};
\node [ draw] (v44) at (0,1) {4,a+6};
\node [ draw] (v54) at (-1,0) {5,a+6};

\node [draw] (v16) at (1,4) {1,a+6}; 
\node [draw] (v26) at (0,3) {2,a+4}; 
\node [draw] (v36) at (-1,2) {3,a+4}; 
\node [draw] (v46) at (-2,1) {4,a+4}; 

\node [draw] (v18) at (-1,4) {1,a+4}; 
\node [draw] (v28) at (-2,3) {2,a+2}; 
\node [draw] (v38) at (-3,2) {3,a+2}; 

\node [draw] (v110) at (-3,4) {1,a+2}; 
\node [draw] (v210) at (-4,3) {2,a}; 
  \end{tikzpicture}
\endpgfgraphicnamed
\caption{A subset of $\sW$.}
\label{fig:type_A_quantum_affine}
\end{figure}

For $0\leq h\leq k$, we define:
\begin{align}\label{eq:minuscule_dimension}
  \begin{split}
    w_{k-h,h,a}^{(i)}&=w^{(i)}_{k-h,a}-\sum_{j\neq
      i}C_{ij}w^{(j)}_{h,a+2(k-h)+1+\epsilon_{ij}}\\
    W_{k-h,h,a}^{(i)}&=S(w_{k-h,h,a}^{(i)}).
  \end{split}
\end{align}
Then $W_{k,0,a}^{(i)}=W_{k,a}^{(i)}$.

Similar to Theorem \ref{thm:KR_module}, by using Lemma \ref{lem:right_negative_product_variant}, we have the following consequences on the modules $W_{k-h,h,a}^{(i)}$. The claim (iv) follows from an analogous proof of \cite[Lemma 4.1]{Nakajima03} or from Proposition \ref{prop:factorization_canonical}.

\begin{Prop}
\label{prop:KR_module_variant}
(i) All monomials in $\qtChar
W^{(i)}_{k-h,h,a}$ are right negative except its leading term $Y^{w_{k-h,h,a}^{(i)}}$.

(ii) Let $m$ be a right negative monomial appearing in the truncated
$q,t$-character $\qtChar_{\leq 2k}
W^{(i)}_{k-h,h,a}$, then we have
\begin{align}
  \label{eq:new_right_negative_monomial}
  \begin{split}
    m=&Y^{w_{k-h,h,a}^{(i)}}\prod_{t=s+1}^{k-h}A^{-1}_{i,a+2t-1}
  \end{split}
\end{align}
where $0\leq s< k-h$.

(iii) We have the following equality in the quantum torus $\YTorus$;
\begin{align}
  \label{eq:exact_sequence_variant}
  \begin{split}
    [\qtChar &W^{(i)}_{h,a+2(k-h+1)}*\qtChar W^{(i)}_{k,a}]\\&=[\qtChar
    W^{(i)}_{k-h,a+2(k-h+1)}*\qtChar
    W^{(i)}_{k+1,a}]
    +t^{-1}\qtChar W^{(i)}_{k-h,h,a},
  \end{split}
\end{align}
where $[\ ]$ denote the normalization by $t$-factors.

(iv) In the Grothendieck ring $\quotKGp$ (Section \ref{sec:definition_quiver_variety}), we have $$
 W^{(i)}_{k-h,a+2(k-h+1)}\otimes W^{(i)}_{k+1,a}=t^\alpha S(w^{(i)}_{k+1,a}+ w^{(i)}_{k-h,a+2(k-h+1)}).$$ for some $\alpha\in\Z$,


\end{Prop}

\section{Facts and conjectures about monoidal categorifications}
\label{sec:preliminaries_monoidal}

\subsection{Monoidal categorification}
\label{sec:monoidal_categorification}

\subsubsection*{Definitions}

Let $\clAlg_\Z^\dagger$ be any given commutative cluster algebra. Recall that a tensor category $(\monCat,\otimes)$ is a category $\monCat$ equipped with, up to a natural isomorphisms, an associative bifunctor $\otimes$ and its left and right identity $S(0)$. A simple object $S$ is called real if $S\otimes S$ remains simple, and prime if it has a nontrivial factorization $S\simeq V_1\otimes V_2$, for some $V_1,V_2\in\monCat$. We give the following definition following
  \cite{HernandezLeclerc09}.
\begin{Def}
   We say that $\clAlg_\Z^\dagger$ admits a monoidal
  categorification if there exists an abelian tensor category
  $(\monCat,\otimes)$, such that
  \begin{enumerate}[(i)]
  \item there exists a ring isomorphism $\kappa$ from the commutative cluster algebra 
    $\clAlg_\Z^\dagger$ to the Grothendieck ring $K_0(\monCat)$ of
    $\monCat$;
  \item $\kappa$ sends all cluster variables (resp. all cluster
    monomials) to classes of prime real simple objects (resp. classes
    of real simple objects).
  \end{enumerate}
\end{Def}
We shall also denote $K_0(\monCat)$ by $\cR_{t=1}$.

By abuse of notation, we denote an object and its class in
the Grothendieck ring $K_0(\monCat)$ by the same symbol. Denote the
multiplication in the Grothendieck ring $K_0(\monCat)$ by $\otimes$.


We consider two types of monoidal categories $\monCat$ possessing a graded Grothendieck ring $\cR_t$. The reader can refer to \cite{KKKO14} and \cite{VaragnoloVasserot03} for examples respectively.
\begin{enumerate}[Type(A)]

\item The abelian monoidal category $(\monCat,\otimes)$ is graded and equipped with a grading shift functor [1], \cf \cite[(5.4)]{KKKO14}, such that:
\begin{itemize}
\item $K_0(\monCat)=\oplus_w \Z[t^{\pm}]S(w)$, where $S(w)$ are some simple objects and $w$ some parameters, $t^{\pm}S(w)=S(w)[\pm 1]$,
\item $K_0(\monCat)$ has a bar involution  $\overline{(\ )}:tS(w)\mapsto t^{-1}S(w)$,
\item the tensor product in $\monCat$ becomes a $t$-twisted product $*$ in $K_0(\monCat)$ satisfying, $\forall w^1,w^2$,
\begin{align}
\label{eq:str_const}
S(w^1)* S(w^2)=\sum_{w^3} a_{w^1,w^2}^{w^3}S(w^3),\
a_{w^1,w^2}^{w^3}\in\N[t^\pm].
\end{align}
\end{itemize}
We define $\cR_t$ to be $K_0(\monCat)$.

\item The abelian monoidal category $(\monCat,\otimes)$ is ungraded. Consider the free $\Z[t^\pm]$-module $\cR_t=\Z[t^\pm]\otimes K_0(\monCat)$. Assume that we can equip it with a $t$-deformation $*$ of the tensor product of $\monCat$ satisfying \eqref{eq:str_const}, and a bar involution such that $\overline{tS(w)}=t^{-1}S(w)$, where $\{S(w)\}$ is the set of the simple objects.
\end{enumerate}

The bar-invariant $\Z[t^\pm]$-basis
$\set{S(w)}$ is called the \emph{simple basis}. It is positive by \eqref{eq:str_const}.

Fix an integer $d\geq 1$. Let $\extQuotKGp$ denote the
natural extension
$\quotKGp\otimes_{\Z[t^\pm]}\Z[t^{\pm\frac{1}{d}}]$. By the convention
of this article, we will take $d=2$.
\begin{Def}\label{def:monoidal_categorification}
  Let $\monCat$ be a monoidal category of type (A)(B). We say that it provides a monoidal categorification of a quantum cluster algebra
  $\clAlg^\dagger$ if
  \begin{enumerate}[(i)]
  \item there exists a ring isomorphism $\kappa$ from $\clAlg^\dagger$
    to the graded Grothendieck ring $\extQuotKGp$ of $\monCat$ such
    that $\kappa(q^\Hf)=t^\frac{1}{d}$ and it commutes with the bar-involutions;
  \item $\kappa$ sends all quantum cluster variables (resp. all
    quantum cluster monomials) to classes of prime real simple objects
    (resp. classes of real simple objects).
  \end{enumerate}
\end{Def}
Notice that if we assume $\monCat$ to be a graded monoidal category considered above (Type(A)), then our Definition \ref{def:monoidal_categorification} corresponds to the Grothendieck group level of the monoidal categorification in the sense of \cite[Definition 5.8]{KKKO14}, but we don't consider $R$-matrices and short exact sequences in the category.
\subsubsection*{Standard basis and properties}

In this article, we shall work with the case that the set of parameters $w$ is
$\oplus_{k=1}^{l}\N\beta_k$, whose basis vector $\beta_k$ is called
the \emph{$k$-th root vectors}. The object
$S(\beta_k)$ is called the $k$-th \emph{fundamental object}.

Endow $\N^l$ with the natural lexicographical order $\prec_w$ such
that for any $w=(w_k)$, $w'=(w_k')$, $w\prec_w w'$ if and only if
there exists some $1\leq p \leq l$,  such that $w_k=w_k'$ for any $k<p$ and
$w_p<w_p'$ (this is the convention in \cite[4.3.5]{Kimura10})

Let $M(w)$ denote the class in $\extQuotKGp$
defined as
\begin{align*}
  M(w)=t^{\Hf\alpha_w} S(\beta_1)^{w_1}* S(\beta_2)^{ w_2}*\cdots * S(\beta_l)^{ w_l},
\end{align*}
for some normalization factor $\alpha_w\in\Z$. Denote $\bm=t^{-\Hf}\Z[t^{-\Hf}]$. We shall further assume
that the set of classes $\set{M(w)}$ decomposes $(\prec_w,\bm)$-unitriangularity into
the basis $\set{S(w)}$ in the sense of Definition \ref{def:triangular}. Then it is
also a $\Z[t^{\pm\frac{1}{2}}]$-basis of $\extQuotKGp$,
which is called the \emph{standard basis}.

Assume $\kappa$ is given and we fix the initial seed $t_0$. If the elements $\kappa^{-1}S(w)$ have distinct leading
degrees $\deg S(w)\in \degL(t_0)=\Z^m$ in the quantum torus $\cT(t_0)$ (Definition \ref{def:leading_term}), we
define the map $\theta^{-1}$ sending $w$ to $\deg \kappa^{-1}S(w)$. Denote the elements $\kappa^{-1}S(w)$ and
$\kappa^{-1}M(w)$ in $\clAlg^\dagger$ by $\simp(w)$ and $\stdMod(w)$
respectively.

The following observation states that a monoidal categorification of $\clAlg^\dagger_\Z$ implies that of
$\clAlg^\dagger$.
\begin{Prop}
Let $\monCat$ be a monoidal category which satisfies
\eqref{eq:str_const}, such that its graded Grothendieck
ring $\extQuotKGp$ is isomorphic to a quantum cluster algebra
$\clAlg^\dagger$ by identifying some simples with the initial quantum
cluster variables. Assume that $\monCat$ has the commutative
Grothendieck ring $\cR_{t=1}$ as the specialization of $\extQuotKGp$ at $t^\Hf=1$ and, moreover, $\monCat$ provides a
monoidal categorification of the commutative cluster algebra
$\clAlg^\dagger_\Z$. Then $\monCat$ provides a
monoidal categorification of $\clAlg^\dagger$ as well.
\end{Prop}
\begin{proof}
Let $t$ be any seed and $\set{S_i(t)}$ the collection of simples in
$\cR_{t=1}$ identified with the set of commutative cluster variables
$\set{x_i(t)}$. Then, for any index $k$, there exists some simple
module $S_k(t)^*$ such that we have the following exchange relation in
$\cR_{t=1}$:
\begin{align*}
  S_k(t)\otimes S_k(t)^*=S_++S_-,
\end{align*}
where $S_+$ and $S_-$ are monomials in $\set{S_i(t)}$. In particular,
$S_+$ and $S_-$ are simples.

Now, assume that, via the isomorphism between the graded Grothendieck ring $\extQuotKGp$ and the quantum cluster algebra, $\set{S_i(t)}$ is identified
with $\set{X_i(t)}$ the set of the quantum cluster variables in $t$. The twisted product
takes the form
\begin{align*}
  S_k(t)\otimes S_k(t)^*=f_+S_++f_-S_-,
\end{align*}
where the coefficient $f_+,f_-\in\N[t^{\pm\Hf}]$ by
\eqref{eq:str_const}. It follows that $f_+$ and $f_-$ must be powers of
$t^{\pm\Hf}$ so that they specialize to $1$.

Next, left multiplying $S_k(t)^{-1}$, by the bar-invariance of $S_k(t)^*$, we
must get
\begin{align*}
  S_k(t)^*=[S_k(t)^{-1}*S_+]^t+[S_k(t)^{-1}*S_+]^t,
\end{align*}
where $[\ ]^t$ denote normalization in the quantum torus
$\Z[t^{\pm\Hf}][S_i(t)^\pm]$, \cf Definition \ref{def:pointed_element}. Therefore, $S_k(t)^*$ is identified with the quantum cluster variable
$X_k(\mu_k t)$.

Starting from the initial seed $t_0$ and repeating the above argument, we
obtain the desired claim.  
\end{proof}

\subsection{Quantum cluster algebras associated with words}
\label{sec:quiver_word}


We shall recall quantum cluster algebras arising from quantum unipotent subgroups $A_q(\mfr{n}(w))$ (type (i)), or arising from level $N$ categories $\cC_N$ of finite dimensional representations of quantum affine algebras (type (ii)). Their canonical initial seeds will be described in a uniform and combinatorial way in terms of a word $\bi$. Our conventions will be compatible with those of \cite{GeissLeclercSchroeer10} and \cite{HernandezLeclerc09} respectively, cf. Remark \ref{rem:compare_PBW_generator} \ref{rem:compare_fundamental_module}.

\subsubsection*{Ice quiver of type (i)(ii)}

Fix an non-empty vertex set $I=\{1,\ldots,r\}$ and a generalized Cartan matrix
$C=(C_{i,j})_{i,j\in I}$. For any $i\in I$, let $s_i$ denote the
simple reflection at the $i$-th simple root. By choosing an acyclic orientation $\overline{\Omega}$ on the diagram $\Gamma$ associated with $C$. We obtain an acyclic quiver $(\Gamma,\overline{\Omega})$.


Let $\bi=(i_l,i_{l-1},\ldots,i_2,i_1)$ be a non-empty word of length
$l$ with elements in $I$.

We say the word $\bi$ is \emph{reduced}, if the Weyl group element
$w_\bi=s_{i_l}\cdots s_{i_2}s_{i_1}$ has length $l$.

We say the word $\bi$ is \emph{$\overline{\Omega}$-adaptable} (or
adaptable for short), if there exists an acyclic quiver
$(\Gamma,\overline{\Omega})$ associated with the Cartan matrix $C$, such that $(i_l,\ldots,i_2,i_1)$ is a sink sequence of the quiver when
  read from right to left, \cf \cite[Section 3.7]{GeissLeclercSchroeer07b} for
  more details.

An adaptable word $(i_r,\ldots,i_2,i_1)$ which contains every
$i\in I$ exactly once is called an \emph{acyclic Coxeter word}.

We consider the following types of the pair $(C,\bi)$:
\begin{enumerate}[(i)]
\item The generalized Cartan matrix $C$ is symmetric, and $\bi$ is reduced.
\item The generalized Cartan matrix $C$ is of type $ADE$, and the word $\bi$ is
  $\overline{\Omega}$-adaptable for some acyclic quiver $(\Gamma,\overline{\Omega})$ associated $C$.
\end{enumerate}

Without loss of generality, we can assume that the diagram $\Gamma$
associated with $C$ is connected.

By default, we consider words satisfying the following mild assumption, which will not affect the correctness of our results.
\begin{Assumption*}
	We assume that every element of $I$ appears in the word
	$\bi$ at least once.
\end{Assumption*}

Given a word $\bi$, for any $1\leq k\leq
l$, $i\in I$, we define the following notations (\cf
\cite[Section 9.8]{GeissLeclercSchroeer10} \cite{BerensteinFominZelevinsky05}) 
  \begin{align*}
k^{\max}=&\max\set{s\in [1,l]|i_s=i_k},\\
k^{\min} =&\min\set{s\in [1,l]|i_s=i_k},\\
k^+=&\min\set{l+1,s\in [k+1,l]|i_s=i_k},\\
k^-=&\max\set{0,s\in [1,k-1]|i_s=i_k},\\
k^+(i)=&\min\set{l+1,s\in [k+1,l]|i_s=i}\\
k^-(i)=&\max\set{0,s\in [1,k-1]|i_s=i},\\
^{\max} i=&\max\set{s\in [1,l]|i_s=i},\\
^{\min} i=&\min\set{s\in [1,l]|i_s=i}.
  \end{align*}
For any interval $[a,b]=\set{k\in\Z|a\leq k\leq b}$ in $[1,l]$, define the
multiplicities
\begin{align*}
  m(i,[a,b])=&|\set{k|k\in [a,b],i_k=i}|,\\
m(i)=&m(i,[1,l]),\\
m_k^+=&m(i_k,[k,l])-1,\\
m_k^-=&m(i_k,[1,k])-1.
\end{align*}
We define $k[0]=k$, and for any integer $1\leq d\leq m_k^+$, we recursively
define the $d$-th offset $k[d]$ of $k$ to be $k[d-1]^+$. The integer
$m(i_k,[k^{\min},k])$ is called the \emph{level }of the vertex $k$.

If we define the following subset of $[a,b]$:
  \begin{align*}
    [a,b](i)=&\set{k|k\in [a,b],i_k=i}.
  \end{align*}
Then the cardinality of $[1,l](i)$ is just
$m(i_k)$. The subset $[1,l](i)$ is
naturally a sequence (read from right to left) whose elements are
ordered as the following:
\begin{align*}
  (^{\min}i[m(i)-1],^{\min} i[m(i)-2],\ldots,^{\min} i[1],^{\min}
  i[0]).
\end{align*}

Following\footnote{In our convention, the quiver
  $\Gamma_\bi$ is opposite to that of \cite{GeissLeclercSchroeer10}.}
\cite[Section 2.4]{GeissLeclercSchroeer10} \cite{BerensteinFominZelevinsky05}, we associate the
quiver $\Gamma_\bi$ to $\bi$ as the following:
\begin{itemize}
\item The vertices of $\Gamma_\bi$ are $1,2,\ldots,l$.
\item For $1\leq s,t\leq l$, there are $-C_{i_s i_t}$ arrows from $t$
  to $s$ if $t^+\geq s^+>t>s$. These are called the \emph{ordinary
    arrows} of $\Gamma_\bi$.
\item For any $1\leq s\leq l$, there is an arrow from $s$ to $s^+$ if
  $s^+\leq l$. These are called the \emph{horizontal arrows} of $\Gamma_\bi$. 
\end{itemize}

\begin{Eg}[type $A$]\label{eg:type_A}
  The quiver in Figure \ref{fig:type_A} is the quiver $\Gamma_\bi$
  associated with the adaptable word $\bi=(1, 2, 1, 3,  2, 1, 4, 3, 2, 1)$ and
  the type $A_4$ Cartan matrix given by 
$$C=\left(\begin{array}{cccc}
    2&-1&0&0\\
    -1&2&-1&0\\
    0&-1&2&-1\\\
0&0&-1&2
  \end{array}\right).$$
 The pair $(C,\bi)$
  is of type both (i) and (ii).
\end{Eg}
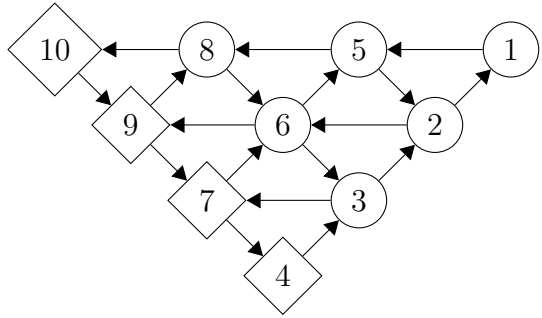
\begin{figure}[htb!]
 \centering
\beginpgfgraphicnamed{fig:type_A}
  \begin{tikzpicture}
 \node [shape=circle, draw] (v1) at (7,4) {1}; 
    \node  [shape=circle, draw] (v2) at (6,3) {2}; 
    \node [shape=circle,  draw] (v3) at (5,2) {3};
\node [shape=diamond, draw] (v4) at (4,1) {4}; 

\node  [shape=circle, draw] (v5) at (5,4) {5}; 
    \node [shape=circle,  draw] (v6) at (4,3) {6};
    
\node [shape=diamond, draw] (v7) at (3,2) {7}; 
    \node  [shape=circle, draw] (v8) at (3,4) {8}; 
    \node [shape=diamond,  draw] (v9) at (2,3) {9};

\node [shape=diamond, draw] (v10) at (1,4) {10};

    \draw[-triangle 60] (v4) edge (v3); 
    \draw[-triangle 60] (v3) edge (v2); 
    \draw[-triangle 60] (v2) edge (v1);
    
    \draw[-triangle 60] (v7) edge (v6); 
    \draw[-triangle 60] (v6) edge (v5);
    
    \draw[-triangle 60] (v9) edge (v8);

    \draw[-triangle 60] (v10) edge (v9); 
    \draw[-triangle 60] (v9) edge (v7); 
    \draw[-triangle 60] (v7) edge (v4);
    
    \draw[-triangle 60] (v8) edge (v6); 
    \draw[-triangle 60] (v6) edge (v3);
    
    \draw[-triangle 60] (v5) edge (v2);

    \draw[-triangle 60] (v1) edge (v5); 
    \draw[-triangle 60] (v5) edge (v8); 
    \draw[-triangle 60] (v8) edge (v10);

    \draw[-triangle 60] (v2) edge (v6); 
    \draw[-triangle 60] (v6) edge (v9); 

\draw[-triangle 60] (v3) edge (v7);
  \end{tikzpicture}
\endpgfgraphicnamed
\caption{A quiver $\Gamma_\bi$ of type $A$}
\label{fig:type_A}
\end{figure}

\begin{Eg}[{\cite[Example 13.2]{GeissLeclercSchroeer10}} ]\label{eg:type_wild}
  The quiver in Figure \ref{fig:type_wild} is the quiver associated
  with the non-adaptable word $\bi=(2,3,2,1,2,1,3,1,2,1)$ and the Cartan matrix $C=
  \left(\begin{array}{ccc}
    2&-3&-2\\
    -3&2&-2\\
    -2&-2&2
  \end{array}\right)
$. The pair $(C,\bi)$ is of type (i). Here, we put a number $s$ on an
arrow to indicate that there are $s$-many arrows drawn here.
\end{Eg}

\begin{figure}[htb!]
 \centering
\beginpgfgraphicnamed{fig:type_wild}
  \begin{tikzpicture}
 \node [shape=circle, draw] (v1) at (10,5) {1}; 
    \node  [shape=circle, draw] (v2) at (9,2) {2}; 
    \node [shape=circle,  draw] (v3) at (8,5) {3};
\node [shape=circle, draw] (v4) at (7,-1) {4}; 

\node  [shape=circle, draw] (v5) at (6,5) {5}; 
    \node [shape=circle,  draw] (v6) at (4,2) {6};
    
\node [shape=diamond, draw] (v7) at (3,5) {7}; 
    \node  [shape=circle, draw] (v8) at (2,2) {8}; 
    \node [shape=diamond,  draw] (v9) at (0,-1) {9};

\node [shape=diamond, draw] (v10) at (-1,2) {10}; 
 
\draw[-triangle 60] (v2) edge  node[near start] {3} (v1);
\draw[-triangle 60] (v4) edge  node[near start] {2} (v2);   
    \draw[-triangle 60] (v4) edge  node[near start] {2} (v3);   
\draw[-triangle 60] (v5) edge  node[near start] {3} (v2);   
\draw[-triangle 60] (v6) edge  node[near end] {3} (v5);   
\draw[-triangle 60] (v7) edge  node[near start] {3} (v6); 
\draw[-triangle 60] (v7) edge  node[near end] {2} (v4);   
\draw[-triangle 60] (v8) edge  node[near start] {2} (v4);   

\draw[-triangle 60] (v9) edge  node[near end] {2} (v8);   
\draw[-triangle 60] (v9) edge  node[near end] {2} (v7);   

\draw[-triangle 60] (v10) edge  node[near start] {3} (v7);   
\draw[-triangle 60] (v10) edge  node[near start] {2} (v9);   

\draw[-triangle 60] (v1) edge (v3);
\draw[-triangle 60] (v3) edge (v5);   
\draw[-triangle 60] (v5) edge (v7);   
\draw[-triangle 60] (v2) edge (v6);
\draw[-triangle 60] (v6) edge (v8);   
\draw[-triangle 60] (v8) edge (v10);   
\draw[-triangle 60] (v4) edge (v9);   
  \end{tikzpicture}
\endpgfgraphicnamed
\caption{The quiver $\Gamma_\bi$ associated with a non-adaptable word $\bi$.}
\label{fig:type_wild}
\end{figure}
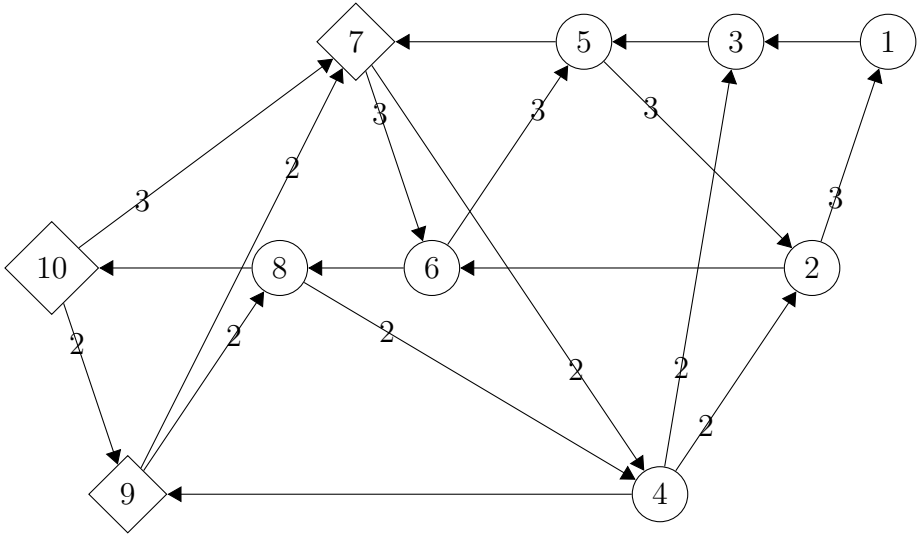

\subsubsection*{Mutation sequences}

Following \cite[Section 13.1]{GeissLeclercSchroeer10}, for any $1\leq
k\leq l$, we define the following mutation sequence (read from right
to left):
\begin{align}
  \label{eq:k_sequence}
  \overleftarrow{\mu}_k=\mu_{k^{\min}[m(i_k,[k,l])-2]}\cdots\mu_{k^{\min}[1]}\cdots\mu_{k^{\min}}.
\end{align}
Notice that, when $m(i_k,[k,l])-2<0$, this sequence is trivial.

Define the mutation sequence (read from right to left)
\begin{align}
  \Sigma_\bi=\overleftarrow{\mu}_l\cdots\overleftarrow{\mu}_2\overleftarrow{\mu}_1\label{eq:GLS_seq}.
\end{align}
We
also define a permutation $\sigma_\bi$ of $[1,l]$ such that, for any $k\in[1,l]$,
$0\leq d\leq m(i_k,[1,l])-2$, we have 
\begin{align*}
  \sigma_\bi (k^{\max})=&k^{\max},\\
\sigma_\bi(k^{\min}[d])=&k^{\min}[
  m(i_k,[1,l])-2-d].
\end{align*}
Notice that $\sigma_\bi=(\sigma_\bi)^{-1}$.

We always view the quiver $\Gamma_\bi$ as an ice quiver by freezing
the vertices $^{\max}i$, $i\in I$. The number
of exchangeable vertices is $n=l-r$. We associate the initial $l\times n$ matrix $\tB(t_0)$ to the quiver $\Gamma_\bi$
such that, at position $(i,j)$, its entry $b_{ij}$ equals the difference
between the number of arrows from $i$ to $j$ and the number of arrows
from $j$ to $i$. Let $\clAlg_\Z(\bi)=\clAlg_\Z(t_0)$ denote the commutative cluster algebra
arising from this initial matrix $\tB(t_0)$.

By \cite[Proposition 13.4]{GeissLeclercSchroeer10}, $t_0$ is
injective-reachable via $(\Sigma_\bi,\sigma_\bi)$ in the sense of
Section \ref{sec:injective_reachable}.

We use $\can(k^{\min},k)$ to denote the initial cluster variables
$X_k(t_0)$, $1\leq k\leq l$. When $\overleftarrow{\mu}_k$ is nontrivial, applying it to the seed
$\overleftarrow{\mu}_{k-1}\cdots\overleftarrow{\mu}_2\overleftarrow{\mu}_1t_0$, we obtain the new cluster variables
$$X_{k^{\min}[d]}(\overleftarrow{\mu}_{k}\cdots\overleftarrow{\mu}_2\overleftarrow{\mu}_1t_0),\ 
0\leq d\leq m(i_k,[k,l])-2,$$
which we will denote by $\can(k^+,k[d]^+)$.
\begin{Rem}\label{rem:compare_PBW_generator}
In \cite[Section 5.4] {GeissLeclercSchroeer11}, the initial quantum cluster variables are denoted by $D(0,b)$, where $0<b\leq l$, which  correspond to our initial quantum cluster variables $\can(b^{\min},b)$. In addition, their quantum minor $D(b,d)$, where $0<b<d\leq l$ and $i_b=i_d$, correspond to our quantum cluster variables $\can(b[1],d)$. Notice that these include all dual PBW basis generators, $D(b^-,b)$, $0<b\leq l$, cf. \cite[Corollary 12.4]{GeissLeclercSchroeer11}, which we denote by $\can(b,b)$.
\end{Rem}
\begin{Eg}
In Example \ref{eg:type_A}, The cluster variables in the canonical initial seed $t_0$ are denoted by
\begin{align*}
  \can(1,10),\can(1,8),\can(1,5),\can(1,1),\\
\can(2,9),\can(2,6),\can(2,2),\\
\can(3,7),\can(3,3),\\
\can(4,4).
\end{align*}
The mutation sequence $\Sigma_\bi$ is the
composition of the following mutation sequences (read from right to left):
\begin{align*}
  \begin{array}{lllll}
  \overleftarrow{\mu}_1=\mu_8\mu_5\mu_1&
 \overleftarrow{\mu}_2=\mu_6\mu_2&
 \overleftarrow{\mu}_3=\mu_3&
 \overleftarrow{\mu}_4=\id&
 \overleftarrow{\mu}_5=\mu_5\mu_1\\
 \overleftarrow{\mu}_6=\mu_2&
 \overleftarrow{\mu}_7=\id&
 \overleftarrow{\mu}_8=\mu_1&
 \overleftarrow{\mu}_9= \id&
 \overleftarrow{\mu}_{10}= \id
\end{array}
\end{align*}
By applying the mutation sequence $\overleftarrow{\mu}_1$ to $t_0$, we obtain the new
cluster variables (from right to left):
\begin{align*}
  \can(5,10),\can(5,8),\can(5,5).
\end{align*}
We continue with the sequence $\overleftarrow{\mu}_2$ to obtain
$\can(6,9),\can(6,6)$. After applying all factors of $\Sigma_\bi$, we
obtain the cluster variables of $\Sigma_\bi t_0$ denoted by
\begin{align*}
  \can(1,10),\can(5,10),\can(8,10),\can(10,10),\\
\can(2,9),\can(6,9),\can(9,9),\\
\can(3,7),\can(7,7),\\
\can(4,4).
\end{align*}
In the convention of Section \ref{sec:injective_reachable}, the cluster variables will be denoted
as follows:
\begin{align*}
  X_{10}(t_0),I_1(t_0),I_5(t_0),I_8(t_0),\\
X_9(t_0),I_2(t_0),I_6(t_0),\\
X_7(t_0),I_3(t_0),\\
X_4(t_0).
\end{align*}
The permutation $\sigma_\bi=
\left(\begin{array}{cccccccccc}
  1&2&3&4&5&6&7&8&9&10\\
8&6&3&4&5&2&7&1&9&10
\end{array}\right).
$
\end{Eg}

\subsection{Monoidal categorification conjecture: type (i)}
\label{sec:category_type_i}

\subsubsection*{Quantum cluster algebra structure}

When the pair $(C,\bi)$ is of type (i), $\clAlg_\Z^\dagger(\bi)$ is isomorphic
to the
coordinate ring of the unipotent subgroup $N(w_\bi)$,
\cf\cite{BerensteinFominZelevinsky05}\cite{GeissLeclercSchroeer10}. Moreover, the canonical initial seed
$t_0$ has a natural
quantization matrix $\Lambda(t_0)$ by
\cite{GeissLeclercSchroeer11}. We define the rational
quantum cluster algebra to be $\clAlg^\dagger_{\Q(q^\Hf)}(\bi)=\clAlg^\dagger(t_0)\otimes \Q(q^\Hf)$.

The quantum
unipotent subgroup $A_{\Q(v)}(\mathfrak{n}(w_\bi))$  is an subalgebra
of $U_{\Q(v)}(\mfr{n})$. It contains its integral form $A_{\Z[v^\pm]}(\mathfrak{n}(w_\bi))$. It has
the dual PBW basis and the dual canonical basis as the restrictions
from that of $U_{\Z[v^\pm]}(\mfr{n})$. We refer the reader to
\cite{Kimura10} for more details. Denote
$A_{\Q(v^\Hf)}(\mathfrak{n}(w_\bi))=A_{\Q(v)}(\mathfrak{n}(w_\bi))\otimes
\Q(v^\Hf)$.

\begin{Thm}[\cite{GeissLeclercSchroeer10}\cite{GeissLeclercSchroeer11}]\label{thm:unipotent_subgroup}
  Assume $(C,\bi)$ is of type (i). Then there is an algebra
  isomorphism\footnote{In
  \cite{GeissLeclercSchroeer11}, the square roots $q^\Hf$ and $v^\Hf$ did not
  appear because the quantum cluster algebra is rescaled and its
  bar-involution is defined differently.} $\kappa$ from the quantum cluster algebra
  $\qClAlg^\dagger_{\Q(q^\Hf)}(\bi)$ to the quantum unipotent subgroup
  $A_{\Q(v^\Hf)}(\mathfrak{n}(w_\bi))$ such that
  $\kappa(q^\Hf)=v^\Hf$ and for any $1\leq a\leq b\leq l$ with $i_a=i_b$, $\kappa\can(a,b)$ are
  $v^\Hf$-shifts of the unipotent quantum minors
  $$D(a^{-},b)=D_{s_{i_1}s_{i_2}\cdots
    s_{i_{a^{-}}}(\varpi_{i_a}),s_{i_1}s_{i_2}\cdots
    s_{i_b}(\varpi_{i_a})}.$$ 

Moreover, $\kappa^{-1}A_{\Z[v^{\pm\Hf}]}(\mathfrak{n}(w_\bi))$ is contained
in $\qClAlg^\dagger(\bi)$.
\end{Thm}
Let $B^*(w_\bi)=\set{b(w)|w\in \N^l}$ denote the dual canonical basis of the quantum unipotent
subgroup $A_{\Q(v)}(\mathfrak{n}(w_\bi))$. The map $\kappa$ transports
the bar-involution on $\cT(t_0)$ to an induced bar-involution\footnote{The map chosen in \cite[Proposition 11.5]{GeissLeclercSchroeer11} sends each initial quantum cluster variable $X_i$ to a dual canonical basis element $v^{\frac{1}{4}(\beta,\beta)}b$ of homogeneous degree $\beta$. $\overline{X_i}=X_i$ induces the bar involution $\overline{b}=v^{\frac{1}{2}(\beta,\beta)}b$, which differs from the standard involution $\sigma$ by a $v$-power, cf. \cite[(6.7)]{GeissLeclercSchroeer11} where our $v$ is denoted by $q$.} on
$A_{\Q(v)}(\mathfrak{n}(w_\bi))$. By multiplying $v^\Hf$-factors,
we modify $B^*(w_\bi)$ into a basis $\set{S(w)|w\in \N^l}$ invariant under the induced bar involution. Then, the
rational quantum cluster algebra has the bar-invariant basis
$\set{\simp (w)}=\set{\kappa^{-1}S(w)}$.

\begin{Conj}[\cite{FominZelevinsky02}\cite{GeissLeclercSchroeer07b}]\label{conj:canonical_basis}
  Assume $(C,\bi)$ is of type (i). Then the quantum cluster monomials of $\qClAlg(\bi)$
  are contained in $\set{\simp (w)}$. 
\end{Conj}

By the work of Khovanov, Lauda \cite{KhovanovLauda08} and Rouquier
\cite{Rouquier08}, the quantum group $U(\mathfrak{n})$ can be viewed as the
Grothendieck ring of modules of the corresponding quiver Hecke algebra,
such that $v$-shifts of the dual canonical basis corresponds to the classes of
finite dimensional simple modules \cite{VaragnoloVasserot09}\cite{rouquier2012quiver}. It follows that the integral form $A_{\Z[v^\pm]}(\mathfrak{n}(w_\bi))$ is isomorphic to the
graded Grothendieck ring $\extQuotKGp(w_\bi)$ of certain monoidal subcategory
$\monCat(w_\bi)$, and each $S(w)$
represent the class of a simple module. 

The category $\monCat(w_\bi)$ fits into the setting of Section
\ref{sec:monoidal_categorification}. In particular, $\extQuotKGp$ has the
basis $\set{M(w)}$ with expected properties, which corresponds to the
normalization of the dual PBW basis. Here, we use Kashiwara's bilinear form following the convention of the \cite{GeissLeclercSchroeer11}.

\begin{Conj}[{Integral form, \cite[Conjecture 12.7]{GeissLeclercSchroeer11}}]\label{conj:integral_quantum_group}
The isomorphism $\kappa$
restricts to an isomorphism from $\clAlg^\dagger$ to
$A_{\Z[v^{\pm\Hf}]}(\mathfrak{n}(w_\bi))$.
\end{Conj}
We shall give a proof of Conjecture \ref{conj:integral_quantum_group} in Theorem \ref{thm:isom_integral}.
\begin{Lem}
  Conjecture
\ref{conj:canonical_basis} implies Conjecture \ref{conj:integral_quantum_group}.
\end{Lem}
\begin{proof}
  Notice the algebra $\clAlg^\dagger$ is generated by the quantum cluster
  variables over $\Z[q^{\pm\Hf}]$. Therefore,
  $\kappa\clAlg^\dagger$ is generated by elements in $\set{S(w)}$ and, consequently, contained in the integral form
  $A_{\Z[v^{\pm\Hf}]}(\mathfrak{n}(w_\bi))$. But the integral form is
  generated by the fundamental
  modules $M(\beta_k)$ over $\Z[v^\pm]$, which are images of quantum
  cluster variables. Therefore, $\kappa\clAlg^\dagger$ is the integral form.
\end{proof}

\subsubsection*{Embeddings of Grothendieck rings}

Let $\bi-i_l$ denote the word $(i_{l-1},\ldots,i_2,i_1)$ whose
associated root vectors are denoted as $\beta_1$, $\beta_2$,\ldots,
$\beta_{l-1}$. Similarly, 
let $\bi-i_1$ denote the word $(i_{l},\ldots,i_3,i_2)$ whose
associated root vectors are denoted as $\beta_2$, $\beta_3$,\ldots,
$\beta_{l}$, and $\bi-i_1-i_l$ the word $(i_{l-1},\ldots,i_3,i_2)$ with
root vectors denoted by $\beta_2,\beta_3,\ldots,\beta_{l-1}$. Then we
have the natural embedding of
$\extQuotKGp(w_{\bi-i_l})$ into $\extQuotKGp(w_\bi)$ sending $S(\beta_k)$ to
$S(\beta_k)$, $\forall 1\leq k\leq l-1$, and the simple basis into the simple basis. 

Moreover, the Lusztig's braid group symmetry $T_{i_1}$ at vertex $i_1$ maps $D_{s_{i_2}\cdots s_{i_k}\varpi_{i_k},s_{i_2}\cdots s_{i_k}\varpi_{i_k}}$ to $(1-v^2)^{\langle h_{i_1},\xi_k\rangle }D_{s_{i_1}s_{i_2}\cdots s_{i_k}\varpi_{i_k},s_{i_1}s_{i_2}\cdots s_{i_k}\varpi_{i_k}}$, where $\langle h_{i_1},\xi_k\rangle$ linearly depends on the homogeneous grading $\xi_k$ of $D_{s_{i_2}\cdots s_{i_k}\varpi_{i_k},s_{i_2}\cdots s_{i_k}\varpi_{i_k}}$. And it will send the dual canonical basis of $A_{\Q(v^\Hf)}(\mathfrak{n}(w_{\bi-i_1}))$ into that of $A_{\Q(v^\Hf)}(\mathfrak{n}(w_\bi))$ modulo such $v$-functions depending on homogeneous gradings, cf. \cite{Kimura10}. Consequently, after rescaling, we have the embedding from $\extQuotKGp(w_{\bi-i_1})$ into $\extQuotKGp(w_\bi)$
sending $S(\beta_k)$ to $S(\beta_k)$ and the simple basis into the simple basis. Notice that, if we work with Lusztig's bilinear form instead, then the above $v$-function does not appear, cf. \cite[38.2.1]{Lusztig93}.

Therefore, there is an natural
embedding from $\extQuotKGp(w_{\bi-i_1-i_l})$ to
$\extQuotKGp(w_\bi)$. Denote the images of the above embeddings by
$\extQuotKGp(w_\bi)_{\hat{l}}$, $\extQuotKGp(w_\bi)_{\hat{1}}$,
$\extQuotKGp(w_\bi)_{\hat{1},\hat{l}}$ respectively.

\subsection{Monoidal categorification conjecture: adaptable word}
\label{sec:adaptable_word}

When the pair $(C,\bi)$ is of type (ii), we refer the reader to \cite{HernandezLeclerc09} for the original monoidal
categorification conjecture of $\clAlg_\Z(\bi)$ in terms
of finite dimensional representation of quantum affine algebras. In this section, we shall
present a quantized and slightly generalized version of this
conjecture in terms of graded quiver
varieties, where $\bi$ is adaptable and $C$ any generalized symmetric Cartan matrix.


\subsubsection*{Subring via adaptable embedding}

Fix the framed repetition quiver
$\tilde{\Gamma}_{\Omega}$ associated with an acyclic quiver
$(\Gamma,\Omega)$.

Let $\bi$ be any given $\overline{\Omega'}$-adaptable word for an acyclic quiver
$(\Gamma,\overline{\Omega'})$. Fix a multidegree
$\ua=(a_k)_{1\leq k\leq l}\in(2\Z)^l$ such that
$a_{k[d+1]}=a_{k[d]}-2$. Define the corresponding map
$\iota_{\ua}$ which sends the vertices $1\leq k\leq l$ of
$\Gamma_\bi$ to $(i_k,a_k)$. It follows that
\begin{align}
  \iota_\ua(^{\min}i[d])=(i,a_{^{\min}i}-2d),\ \forall i\in I,0\leq d\leq m(i)-1.
\end{align}

We further assume the multidegree $\ua$
to be \emph{$\Omega'$-adaptable} (or adaptable for simplicity) such
that the full subquiver of $\tilde{\Gamma}_{\Omega}$ on the
vertices $(^{\min}i,a_{^{\min}i}-1)$, where $i\in I$, is isomorphic to $(\Gamma,\Omega')$. Because the word $\bi$ is adaptable, an adaptable
embedding multidegree $\ua$ always exists. The associated
map $\iota_\ua$ is called an \emph{adaptable embedding}.

\begin{Eg}\label{eg:embed_adaptable_quiver}

We take the framed repetition quiver $\tilde{\Gamma}_\Omega$
associated with $(\Gamma,\Omega)$ in Figure \ref{fig:quiver_A_5}. Part of this quiver is drawn in Figure \ref{fig:quiver_variety_A_5}.

Take the word $\bi=(2 1 3 2 1  4 3  2 15 4 3 2)$ (read from right
to left). It is $\overline{\Omega}$-adaptable with respect to the orientation
$\overline{\Omega}$. The associated ice quiver $\Gamma_\bi$ is drawn
in Figure \ref{fig:ice_quiver_A_5}. Recall that we have
$i_1=i_6=i_{10}=i_{13}=2$, $1[0]=1$, $1[1]=6$,
$1[2]=10$, $1[3]=13$,
$13^{\min}=10^{\min}=6^{\min}=1^{\min}=1$, $m_6^-=1$, $m_6^+=2$.

For any chosen $a\in 2\Z$, we have an
$\Omega$-adaptable multidegree $\ua$ and the associated adaptable
embedding $\iota_\ua$ such that $a_k=a+6$
for $1\leq k\leq 5$, \cf Figure \ref{fig:quiver_variety_A_5}.

Let us take another adaptable word $\bi'=(2 1 3 5 4)$. It is
$\overline{\Omega}'$-adaptable where the acyclic quiver $(\Gamma,\Omega')$ is
given in Figure \ref{fig:quiver_A_5_new}. A possible adaptable
grading $\iota_\ua$ could have the image
$$\set{(1,a+4),(2,a+2),(3,a+4),(4,a+6),(5,a+6)}$$ in the set of vertices of the
framed repetition quiver $\tilde{\Gamma}_{\Omega}$.
\end{Eg}
\begin{figure}[htb!]
 \centering
\beginpgfgraphicnamed{fig:ice_quiver_A_5}
  \begin{tikzpicture}[scale=0.6]
\node [shape=circle,draw] (w16) at (2,8) {5}; 
\node [shape=circle, draw] (w26) at (4,6) {1};
    \node [shape=circle, draw] (w36) at (2,4) {2};
\node [shape=circle, draw] (w46) at (0,2) {3};
\node [ draw] (w56) at (-2,0) {4};

\node [shape=circle,draw] (w14) at (-2,8) {9}; 
\node [shape=circle,draw] (w24) at (0,6) {6}; 
\node [shape=circle,draw] (w34) at (-2,4) {7}; 
\node [draw] (w44) at (-4,2) {8};

\node [draw] (w12) at (-6,8) {12}; 
\node [shape=circle,draw] (w22) at (-4,6) {10}; 
\node [draw] (w32) at (-6,4) {11};

\node [draw] (w20) at (-8,6) {13};

\draw[-triangle 60](w16)edge (w26);
\draw[-triangle 60](w36)edge (w26);
\draw[-triangle 60](w46)edge (w36);
\draw[-triangle 60](w56)edge (w46);

\draw[-triangle 60](w14)edge (w24);
\draw[-triangle 60](w34)edge (w24);
\draw[-triangle 60](w44)edge (w34);

\draw[-triangle 60](w24)edge (w16);
\draw[-triangle 60](w24)edge (w36);
\draw[-triangle 60](w34)edge (w46);

\draw[-triangle 60](w12)edge (w22);
\draw[-triangle 60](w32)edge (w22);

\draw[-triangle 60](w22)edge (w14);
\draw[-triangle 60](w22)edge (w34);

\draw[-triangle 60](w20)edge (w12);
\draw[-triangle 60](w20)edge (w32);
\draw[-triangle 60](w32)edge (w44);
\draw[-triangle 60](w44)edge (w56);

\draw[-triangle 60](w16)edge (w14);
\draw[-triangle 60](w14)edge (w12);
\draw[-triangle 60](w26)edge (w24);
\draw[-triangle 60](w24)edge (w22);
\draw[-triangle 60](w22)edge (w20);
\draw[-triangle 60](w36)edge (w34);
\draw[-triangle 60](w34)edge (w32);
\draw[-triangle 60](w46)edge (w44);
  \end{tikzpicture}
\endpgfgraphicnamed
\caption{The ice quiver $\Gamma_\bi$ associated with an adaptable word
  $\bi$ and $A_5$
  Cartan matrix.}
\label{fig:ice_quiver_A_5}
\end{figure}

\begin{figure}[htb!]
 \centering
\beginpgfgraphicnamed{fig:quiver_A_5_new}
\begin{tikzpicture}
\node [shape=circle, draw] (v1) at (-1,4) {1};
    \node [shape=circle, draw] (v2) at (-2,3) {2};
    \node [shape=circle, draw] (v3) at (-1,2) {3};
\node [shape=circle, draw] (v4) at (0,1) {4};
\node [shape=circle, draw] (v5) at (-1,0) {5};

    \draw[-triangle 60] (v1) edge (v2); 
\draw[-triangle 60] (v3) edge (v2);
\draw[-triangle 60] (v4) edge (v3);
\draw[-triangle 60] (v4) edge (v5);

\end{tikzpicture}
\endpgfgraphicnamed
\caption{An acyclic quiver $(\Gamma,\Omega')$ of Cartan type $A_5$}
\label{fig:quiver_A_5_new}
\end{figure}

Fix an adaptable embedding $\iota_{\ua}$. Define $\extQuotKGp(\bi,\ua)$ to
be the subring of the extended graded Grothendieck ring $\extQuotKGp$ generated by the simples
$W^{(i_k)}_{1,a_k}$, $i\in I$, $1\leq k\leq l$. It has the simple
basis and standard basis as expected in Section
  \ref{sec:monoidal_categorification}.
   
  \begin{Rem}\label{rem:compare_fundamental_module}
  Recall that $W_{1,a_k}^{(i_k)}$ lies in the dual of the Grothendieck ring of perverse sheaves over graded quiver varieties, cf. Section \ref{sec:definition_quiver_variety} \ref{sec:KR_module}. By \cite{Nakajima01}, this geometric object corresponds to the fundamental module $V_{i_k,q^{a_k}}$ in \cite[Section 3.3]{HernandezLeclerc09}. Notice that their ordered tensor products give the standard modules of the quantum affine algebra, cf. \cite[Section 5.3]{HernandezLeclerc11}, such that the indices $k$ increase from left to right.
  \end{Rem}
  
If we take the adaptable word $\bi=c^{N+1}$, where $c$ is an acyclic Coxeter word and $N\in \N$, or $\bi=w_0$ the longest reduced word in the Weyl group, then $\extQuotKGp(\bi,\ua)$ is the $\Z[t^{\pm\Hf}]$-extended $t$-deformed Grothendieck ring of a level $N$ category $\cC_N$ in \cite{HernandezLeclerc09}, \cf Example \ref{eg:level_N_subcategory}, or the category $\cC_Q$ in \cite[Section 5.11]{HernandezLeclerc11} respectively.

\subsubsection*{Quantum cluster algebra structure}

We denote by $\theta_\ua:\cT(t_0)\ra \YTorus^H_{t^\Hf}$ the embedding of the
Laurent polynomial rings such that
\begin{align*}
\theta_\ua(q^\Hf)&=t^\Hf,\\
\theta_\ua(X_k)&=Y_{i_k,a_k} Y_{i_k,a_{k^-}}\cdots Y_{i_k,a_{k^{\min}}},\ 1\leq k\leq l.
\end{align*}
It follows that we have
\begin{align*}
  \theta_\ua(Y_k)=\theta_a(X^{\tB(t_0)e_k})=A^{-1}_{i_k,a_k-1}=A^{-1}_{i_k,a_{i_k}-2 m_k^--1}.
\end{align*}

The twisted multiplication on $\YTorus^H_{t^\Hf}$ then induces the twisted
multiplication on $\cT(t_0)$ via the morphism $\theta_\ua$ such that 
\begin{align*}
  \Lambda(t_0)(e_k,e_s)=\cN(w^{(i_k)}_{m^-_k+1,a_k},w^{(i_s)}_{m^-_s+1,a_s}).
\end{align*}
By Lemma \ref{lem:compatible_twist},
it is compatible with the matrix $\tB$ such that
\begin{align*}
  \Lambda(t_0)\cdot (-\tB)= \left(\begin{array}{c}
      -2\id_n\\
      0
    \end{array}\right).
\end{align*}
We denote the quantum cluster algebra
obtained via this quantization by $\clAlg(\bi)$.


By comparing $T$-systems with specific exchange relations, we obtain the following result, whose proof
goes the same as that of \cite[Proposition 12.1]{GeissLeclercSchroeer11}.
\begin{Thm}\label{thm:categorification_quiver_variety}
Assume that the word $\bi$ and the multidegree $\ua$ are adaptable. Then there exists
an injective algebra homomorphism $\kappa^{-1}$ from the
Grothendieck ring $\extQuotKGp(\bi,\ua)$ to the quantum
cluster algebra $\qClAlg^\dagger(\bi)$, such that
$\kappa^{-1}(t^\Hf)=q^\Hf$ and, for any $1\leq b\leq s\leq l$ such
that $i_b=i_s$, $\kappa^{-1}W^{(i_b)}_{m(i_b,[b,s]),a_s}$ agrees with the quantum cluster variables $\can(b,s)$.
\end{Thm}

\begin{Rem}
When
  the word $\bi$ is reduced adaptable, $\cR(\bi,\ua)=\cR_{t^\Hf}^H(\bi,\ua)$ is also isomorphic to the quantum unipotent subgroup
$A_{\Q(v^\Hf)}(\mathfrak{n}(w_\bi))$, \cf the proofs of \cite[Theorem 12.3]{GeissLeclercSchroeer11}\cite[Theorem
  6.1(a)]{HernandezLeclerc11} based on $T$-systems. In this case, $S(w)$ represents a
simple module of the quiver Hecke algebra.

If the Cartan matrix $C$ is of Dynkin type $ADE$, $\quotKGp(\bi,\ua)$
and $\extQuotKGp(\bi,\ua)=\cR^H_{t^\Hf}(\bi,\ua)$ are $t$-deformations of the Grothendieck ring of the
monoidal subcategory $\monCat(\bi,\ua)$ of finite dimensional representations of the 
quantum affine algebras $U_\varepsilon(\hat{\mathfrak{g}})$, where
each $S(w)$ represents a simple module as constructed by \cite{Nakajima01}.
\end{Rem}

The corresponding monoidal categorification conjecture takes
the following form, where the case $\bi=c^{N+1}$ is a conjecture about level $N$ categories in \cite{HernandezLeclerc11}.
\begin{Conj}\label{conj:simple_module}
  Assume that $\bi$ is adaptable. Then the quantum cluster monomials of $\qClAlg^\dagger(\bi)$
  are contained in $\kappa^{-1}(\set{S(w)}_{w\in\N^l})$.
\end{Conj}

\begin{Thm}[{\cite[Theorem
5.1]{hernandez2013cluster}, \cite[Theorem
3.1]{GeissLeclercSchroeer10}}]
The map $\kappa^{-1}$ induces an isomorphism from the extended Grothendieck
  ring $\cR_{t=1}(\bi,\ua)\otimes \Q$ to $\clAlg(\bi)_{\Z}\otimes\Q$.
\end{Thm}

\subsubsection*{Embeddings of Grothendieck rings}

Let $\ua_{\hat{1}}$ to be the multidegree obtained from $\ua$ by deleting
the component $a_1$. Similarly, construct $\ua_{\hat{l}}$ and
$\ua_{\hat{1},\hat{l}}$ by deleting $a_l$ and $a_1,a_l$ from $\ua$
respectively. By construction, we have natural embeddings of
$\extQuotKGp(\bi-i_1,\ua_{\hat{1}})$, $\extQuotKGp(\bi-i_l,\ua_{\hat{l}})$,
$\extQuotKGp(\bi-i_1-i_l,\ua_{\hat{1},\hat{l}})$ into $\extQuotKGp(\bi,\ua)$, which send $S(\beta_k)$ to $S(\beta_k)$ and the simple
bases into the simple basis. As before, we denote the images by $\extQuotKGp(\bi,\ua)_{\hat{1}}$, $\extQuotKGp(\bi,\ua)_{\hat{l}}$,
$\extQuotKGp(\bi,\ua)_{\hat{1},\hat{l}}$ respectively.

\begin{Eg}[Level $N$ category]\label{eg:level_N_subcategory}
Assume that the quiver $Q$ is a Dynkin quiver, we have the corresponding quantum affine algebra $U_q(\hat{\mathfrak{g}})$ with generic $q\in \C^*$. 

For given adaptable word $\bi$ and adaptable multidegree $\ua$, let $\cC(\bi,\ua)$ be the monoidal category generated by the $U_q(\hat{\mathfrak{g}})$'s fundamental modules $S(\beta_k)=V_{i_k,q^{a_k}}$, $1\leq k\leq l$. Notice that these fundamental modules admit geometric realization given by our $W_{1,a_k}^{(i_k)}$. Then $\extQuotKGp(\bi,\ua)$ is the $\Z[t^{\pm\Hf}]$-extended graded Grothendieck ring of $\cC(\bi,\ua)$.

Further assume $Q$ is bipartite, $\bi=c^{N+1}$ where $c$ is an acyclic word, $N\in\N$. Then for any $1\leq k\leq r$, the left shifted vertex $k[d]=k+2d$, $\forall 0\leq d\leq N$ (cf. Example \ref{eg:embed_adaptable_quiver} for an example of $k[d]$). Choose the multidegree $\ua$ such that for any $1\leq k\leq r$, $0\leq d\leq N$, we have $a_{k[N-d]}=2d$ if $i_k$ is a sink point and in $Q$ and $a_{k[N-d]}=2d+1$ if $i_k$ is a source point. The level $N$ subcategory $\cC_N(\ua)$ in \cite[Proposition 3.2]{HernandezLeclerc09} is\footnote{
The module $V_{i,q^{2d+\xi_i}}$ in \cite[Proposition 3.2]{HernandezLeclerc09} associated with $Q\op$, $j\in I$, is noted as our module $S(\beta_k)=V_{k,q^{a_k}}$, where $k$ satisfies $i_k=i$, $k=k^{min}[N-d]$. We have $\xi_{i_k}=a_{k^{max}}$ and the category $\cC_N$ depends on the choice of $\ua$.} our category $\cC_(c^{N+1},\ua)$.

Now we have the natural inclusion $\cC_{N-1}(\ua)\subset \cC_N(\ua)$ which appeared in \cite{HernandezLeclerc09}. The former Grothendieck ring is the subring of the latter by successively deleting the fundamental modules $S(\beta_1),\ldots,S(\beta_r)$ corresponding to the first $r$ letters:
$$
\extQuotKGp(c^N,\ua_{\hat{1},\hat{2},\cdots,\hat{r}})=\extQuotKGp(c^{N+1},\ua)_{\hat{1},\hat{2},\cdots,\hat{r}}.
$$
We also have the natural inclusion of a different level $N-1$ subcategory $\cC(c^N,\ua+2)\subset\cC(c^{N+1},\ua)$. The former Grothendieck ring is the subring of the latter by successively deleting the fundamental modules $S(\beta_{r[N]}),\ldots,S(\beta_{1[N]})$ corresponding to the last $r$ letters: 
$$
\extQuotKGp(c^N,\ua_{\hat{r[N]},\cdots,\hat{2[N]},\hat{1[N]}})=\extQuotKGp(c^{N+1},\ua)_{\hat{r[N]},\cdots,\hat{2[N]},\hat{1[N]}}.
$$
\end{Eg}

\section{Monoidal categorification conjectures}
\label{sec:application}
We refer the reader to Section \ref{sec:preliminaries}, \ref{sec:preliminaries_monoidal} for the
necessary notations and definitions used in this section.

Fix a pair $(C,\bi)$ of type (i) or (ii). We want to know if the quantum cluster
algebra $\clAlg=\clAlg(\bi)$ has the common triangular basis or not. As before, we impose the mild assumption that every $i\in I$ appears
at least once in $\bi$.

Recall that we have the corresponding extended Grothendieck ring
$\extQuotKGp(\bi)$ which can be taken as an extension of a quantum
unipotent subgroup $A_{\Z[v^\pm]}(\mfr{n}(w_\bi))$ if $\bi$ is reduced or a subring $\extQuotKGp(\bi,\ua)$ for
some adaptable multidegree $\ua$ if $\bi$ is adaptable.

\subsection{Initial triangular basis}
\label{sec:verify_initial_condition}

We work in the canonical initial seed $t_0$ associated with $(C,\bi)$ in this section.

Define the injective linear map $\theta^{-1}$ from
$\N^l=\oplus_{1\leq k\leq l} \N \beta_k$ to $\domDegL(t_0)$ such that, for any $i\in I$, we have
\begin{align*}
\theta^{-1}\beta_{i[d]}&=e_{i[d]}-e_{i[d-1]},\ 1\leq d\leq m(i)-1,\\
\theta^{-1}\beta_{^{\min} i}&=e_{^{\min} i}.\\
\end{align*}

Notice that, $\forall 1\leq k\leq l$, the vector $\theta^{-1}\beta_k$ equals $\deg^{t_0}\can(k,k)$. Recall that there exists an injective homomorphism $\kappa^{-1}$ from $\cR(\bi)$ to the quantum cluster algebra $\qClAlg^\dagger(\bi)$, such that $\kappa^{-1}S(\beta_k)=\can(k,k)$. Consequently, $\kappa^{-1}$ sends the standard basis element $M(w)$ to a pointed element $\kappa^{-1}M(w)\in\cT(t_0)$ with leading degree $\theta^{-1}w$, $\forall w=\sum_{1\leq k\leq l,w_k\in\N} w_k \beta_k$. Moreover, because $M(w)$ has a $(\prec_w,\bm)$-unitriangular decomposition into $\{S(w')\}$, which all become positive Laurent polynomials in seed $t_0$ under the map $\kappa^{-1}$, the leading term of $\kappa^{-1}M(w)$ must be the contribution from $\kappa^{-1}S(w)$ since it has coefficient $1$. So $\kappa^{-1}S(w)\in\cT(t_0)$ is a pointed element with degree\footnote{As an alternative approach to leading degrees, in type (ii), we can explicitly compare the $q$-characters of elements in $\cR(\bi)$ and Laurent polynomials in $\cT$.} $\deg^{t_0}S(w)=\deg^{t_0}M(w)=\theta^{-1}w$.
\begin{Lem}\label{lem:isom_degree_lattice}
  $\theta^{-1}$ is an isomorphism of cones.
\end{Lem}
\begin{proof}
$\kappa^{-1}$ is an isomorphism when $q=1$. Any element $Z$ in $\clAlg^\dagger_{\Z}(\bi)\otimes \Q$ has the expansion into the simple basis elements $\kappa^{-1}S(w)$ pointed at $\theta^{-1}w$. Consequently, the $\prec_{t_0}$-maximal degrees of
  $Z$ takes the form $\theta^{-1}w$. It follows by
  definition that the dominant degree lattice $\domDegL(t_0)$ equals
  $\theta^{-1}\N^l$.
\end{proof}

The lexicographical
order $\prec_{w}$ on $\oplus \N\beta_k$ extends to a lexicographical order $\prec_w$ on $\oplus\Z\beta_k$ and consequently on $\oplus\Z\beta_k\overset{\theta^{-1}}{\simeq}\degL(t)$. Because $\deg^{t_0} Y_k(t_0)=-e_{k[1]}+\sum_{i<k[1]} b_{ik}e_i \prec_w 0$, $\forall 1\leq k\leq n$, we deduce that $\tg^1\prec_{t_0} \tg^2$ in $\degL(t_0)$
implies $\tg^1\prec_{w}\tg^2$. The following claim follows.
\begin{Lem}\label{lem:initial_bounded}
  $(\prec_{t_0},\domDegL(t_0))$ is bounded from below.
\end{Lem}

With the help of Lemma \ref{lem:isom_degree_lattice}
\ref{lem:initial_bounded}, we obtain the following result, which verifies Conjecture \ref{conj:integral_quantum_group} for
type (i) quantum cluster algebras. It was already known for
type (ii) quantum cluster algebras by \cite[Theorem
5.1]{hernandez2013cluster}.
\begin{Thm}\label{thm:isom_integral}
The injective homomorphism $\kappa^{-1}$ from $\extQuotKGp(\bi)$ to
  $\clAlg^\dagger(\bi)$ is an isomorphism. In particular, $\{\simp(w)\}$
  is a basis of $\clAlg^\dagger(\bi)$.
\end{Thm}
\begin{proof}
Work in the initial quantum torus $\cT(t_0)$. Observe that any element of $\clAlg^\dagger(\bi)$ has maximal degrees
in $\domDegL(t_0)$, because it is a polynomial of cluster variables which have maximal degrees. For any given element $Z$ of
$\clAlg^\dagger(\bi)$, we repeat the expansion
algorithm in Remark \ref{rem:expansion_algorithm} to deduce its decomposition in $\domDegL(t_0)$-pointed set $\set{\simp(w)}$:
\begin{enumerate}[(a)]
\item Let $\{\tg_i\}$ denote the set of maximal degrees of $Z$ with coefficients $z_i\in \Z[q^{\pm\Hf}]$. By Lemma \ref{lem:isom_degree_lattice}, we can find $\simp(w_i)=\kappa^{-1}S(w_i)$ pointed at degree $\tg_i$, where the parameter $w_i=\theta \tg_i$.
\item Notice that $Z-\sum_i z_i \simp(w_i)$ has its maximal degrees inferior than those of $Z$. If it is nonzero, use it to replace $Z$ and go to the previous step.
\end{enumerate}
The above process must end after finite steps by Lemma \ref{lem:initial_bounded}. Therefore, $Z$ is a $\Z[t^{\pm\Hf}]$-linear combination of $\simp(w_i)$. The claim follows.
\end{proof}

For any $1\leq a\leq b\leq l$ with $i_a=i_b$, use $S(a,b)$ to denote the simple
module $S(\beta_a+\beta_{a[1]}+\ldots+\beta_{b})$. Then we have the quantum cluster
variables $X_k(t_0)=\simp(\theta
e_k)=\simp(k^{\min},k)$ and, for any $k\neq k^{\max}$, $I_k(t_0)=\simp(k[1],k^{\max})$.

\begin{Rem}
  When the word $\bi$ is adaptable, we can construct the
  graded Grothendieck ring $\extQuotKGp(\bi)=\extQuotKGp(\bi,\ua)$ as in Section
  \ref{sec:quiver_variety} for some choice of adaptable multidegree
  $\ua=(a_k)\in\Z^l$, such that the root vectors $\beta_k$ corresponds
  to the unit vectors $e_{i_k,a_k}$. In this situation, the map
  $\theta$ identify the $X$-variables $X_k$ and $Y$-variables $Y_k=\prod_{1\leq s\leq
    l}X_{s}^{b_{sk}(t_0)}$ with the monomials
  $Y_{i_k,a_k}Y_{i_k,a_k+2}\cdots Y_{i_k,a_{k^{\min}}}$ and
  $A_{i_k,a_k-1}^{-1}$ respectively in the convention of Section
  \ref{sec:quiver_variety}, cf. Section \ref{sec:adaptable_word}.
\end{Rem}

For any $w$, the basis element
$\simp (w)=\kappa^{-1}S(w)$ has the leading degree $\tg=\theta^{-1}w$ in $\cT(t_0)$. Denote the basis $\set{\simp (w)}$ by $\can^{\dagger,t_0}$ and $\simp (w)$ by $\can^{\dagger,t_0}(\tg)$ correspondingly.

Recall that we have another basis formed by the elements
\begin{align}\label{eq:pbw}
  \stdMod(w)=[\simp(\beta_1)^{w_1}*\simp(\beta_2)^{w_2}*\cdots*\simp(\beta_l)^{w_l}]^{t_0}.
\end{align}
It is $(\prec_w,\bm)$-unitriangular to the basis $\can^{\dagger, t_0}$. By Lemma \ref{lem:triangular}(iii), it is $(\prec_{t_0},\bm)$-unitriangular
to $\can^{\dagger, t_0}$ as well. From now on, by default, we shall
only consider the dominance orders $\prec_t$ associated with a seed $t$ (or denoted by
$\prec$ when there is no confusion). Notice that $\N^l=\oplus_k \N
\beta_k$ inherits the order $\prec_{t_0}$ via the isomorphism $\theta$.

\begin{Prop}[Factorization property]\label{prop:factorization_canonical}
  The bar-invariant simple basis $\set{S (w)}$ of $\extQuotKGp$ factors through the
  simples $S(\theta e_s)$, for
  any $1\leq s\leq l$ such that $s=s^{\max}$ (longest
  Kirillov-Reshetikhin modules):
  \begin{align*}
    [S(w)*S(\theta e_s)]=[S(\theta e_s)*S(w)]=S(w+\theta e_s),\
    \forall w\in\N^l,
  \end{align*}
  where $[Z]$ denote the bar-invariant element of the form $t^\alpha Z$, $\alpha\in\frac{\Z}{2}$.
\end{Prop}
Notice that the longest Kirillov-Reshetikhin modules correspond to frozen variables as in \cite{HernandezLeclerc09}.

In type (i), this property was proved in
  \cite{Kimura10}. With the help of the standard basis and $T$-systems, we will give a proof by induction in the end
  of this subsection .

  \begin{Cor}
    Localized at the frozen variables $S(\theta e_s)$, where $s^{\max}=s$, the
    simple basis $\set{S(w)}$ generates a basis of $\clAlg(\bi)$:
\begin{align}\label{eq:initial_weakly_triangular_basis}
\can^{t_0}=\{S(w)\cdot \prod_{1\leq s\leq l:s=s^{\max}} S(\theta e_s)^{f_s}|w=\sum w_i \beta_i,\ f_s=0\mathrm{\ if\ }w_s>0\}.
\end{align}     
     Its
    structure constants are contained in $\N[q^{\pm\Hf}]$.
  \end{Cor}

We use $t_0\langle l^-\rangle$ to denote the
seed from $t_0$ by freezing vertex $l^-$. Because the only exchangeable vertex
connected to $l$ is $l^-$, the frozen vertex $l$ does not connect to
any exchangeable vertex in the corresponding ice quiver $\tQ(t_0\langle
l^-\rangle)$. Consequently, the quantum cluster algebra $\qClAlg^\dagger(t_0\langle
l^-\rangle)$ is generated by $\qClAlg^\dagger(\bi-i_l)$ and the frozen
variable $X_l$. It follows that $\qClAlg^\dagger(t_0\langle l^-\rangle)$ has the following basis
\begin{align}
\can^\dagger(t_0\langle
l^-\rangle)=\set{X_l^d\can(\eta)|\can(\eta)\in \can^\dagger(t_0(\bi-i_l)), d\in\N}.
\end{align}
As a consequence, we obtain the following Lemma.
\begin{Lem}\label{lem:embedding}
The restriction of
$\can^{\dagger,t_0}$ on the subalgebra $\qClAlg^\dagger(t_0\langle
l^-\rangle)$ of $\clAlg(\bi)$ gives its basis $\can^{\dagger, t_0\langle
  l^-\rangle}$.
\end{Lem}


\begin{Prop}(triangular product)\label{prop:relate_PBW}
The $\degL(t_0)$-pointed set $[X_k(t_0)*\can^{t_0}]^{t_0}$ is $(\prec_{t_0},\bm)$-unitriangular to $\can^{t_0}$, for any $1\leq k\leq l$.
\end{Prop}
\begin{proof}
Omit the symbol $t_0$ for simplicity. We prove the claim by induction on the length $l$ of the word
$\bi$. The case $l=|I|$ is trivial.

Now assume the claim has been verified for the words of length
$l-1$.

The quantum cluster variable $X_k$ is the basis element
$\can(e_k)=\simp(\theta e_k)=\simp(\beta_k+\beta_{k[-1]}+\ldots+\beta_{k^{\min}})$. By
Proposition \ref{prop:factorization_canonical}, it suffices to verify
the claim for $k\neq l$.

Let
$\simp(w)$ be any basis elements in $\can^{\dagger,t_0}$, where
$w\in\N^l$. Because the simple basis is unitriangular to the
standard basis, we have the following $(\prec,\bm)$-unitriangular decomposition
\begin{align*}
  \simp(w)=&\sum_{w^1\preceq w}c(w,w^1)\stdMod(w^1)\\
=&\sum_{w^1\preceq w}c(w,w^1)[\stdMod(w^1_{\hat{l}})*\stdMod(w^1_l\beta_l)],
\end{align*}
where $w^1_l\beta_l$ is the $l$-th component of $w^1$ and
$w^1_{\hat{l}}=w^1-w^1_{l}\beta_l$, $[\ ]$ is the normalization.

Now left multiplying the initial cluster variable $X_k$, by Lemma
\ref{lem:preserve_triangular}, we still have a $(\prec,\bm)$-unitriangular decomposition
\begin{align*}
  [X_k*\simp(w)]=&\sum_{w^1\preceq
    w}c(w,w^1)[X_k*\stdMod(w^1_{\hat{l}})*\stdMod(w^1_l\beta_l)]\\
=&\sum_{w^1\preceq w}c(w,w^1)[[X_k*\stdMod(w^1_{\hat{l}})]*\stdMod(w^1_l\beta_l)].
\end{align*}

As discussed before, the quantum cluster algebra
$\qClAlg^{\dagger}(\bi-i_l)$ associated with the length $l-1$ word
$\bi-i_l$ is a subalgebra of $\qClAlg^{\dagger}(\bi)$ and it contains the standard basis element
$\stdMod(w^1_{\hat{l}})$. It also contains $X_k$ because $k\neq l$. By
induction hypothesis, the normalized product
$[X_k*\stdMod(w^1_{\hat{l}})]$ is unitriangular to the simple basis of $\qClAlg^{\dagger}(\bi-i_l)$, and, consequently, to
its standard basis. Namely, we have $(\prec,\bm)$-unitriangular decompositions
\begin{align*}
  [X_k*\stdMod(w^1_{\hat{l}})]=\sum_{w^2\preceq \theta
    e_k+w^1_{\hat{l}}}d(w^2,\theta e_k+w^1_{\hat{l}})\stdMod(w^2).
\end{align*}

Reusing Lemma \ref{lem:preserve_triangular}, we now obtain the $(\prec,\bm)$-unitriangular decomposition
\begin{align*} 
[X_k*\simp(w)]=&\sum_{w^1\preceq w,w^2\preceq \theta
  e_k+w^1_{\hat{l}}}e(w,w^1,w^2)[\stdMod(w^2)*\stdMod(w^1_l\beta_l)],
\end{align*}
where $(w^2)_l=0$, $e(w,w^1,w^2)$ is contained in $\bm$ if $w^1_l+w^2\neq w+\theta e_k$
and equals to $1$ if  $w^1_l\beta_l+w^2= w+\theta e_k$.

Recall that the standard basis of $\clAlg^{\dagger,t_0}$ is
unitriangular to its simple basis. The claim follows.
\end{proof}

\begin{Cor}(triangular basis)
The basis $\can^{t_0}$ given by \eqref{eq:initial_weakly_triangular_basis} is the triangular basis with respect to $t_0$.
\end{Cor}

We finish this subsection by a proof of Proposition \ref{prop:factorization_canonical}.

\begin{Lem}\label{lem:pbw_induction}
Take any $Z$ in $\extQuotKGp$. If $Z$ is $(\prec,\bm)$-unitriangular to
  the standard basis, then for any $w_1,w_l\in \N$, the normalized
  product $[Z*S(w_l\beta_l)]$ and $[S(w_1 \beta_1)*Z]$ are
  $(\prec,\bm)$-unitriangular to the standard basis.
\end{Lem}
\begin{proof}
Notice that we can replace $Z$ by a unitriangular decomposition
into the standard basis. The claim follows from
Lemma \ref{lem:preserve_triangular}.
\end{proof}

\begin{Lem}\label{lem:simple_factor}
  Take any $Z=\sum a_i S(w^i)$ in $\extQuotKGp$ such that
  $a_i\in\N[t^{\pm\Hf}]$. If there exists some simple module $S(w)$ such
  that $Z * S(w)$ or $S(w)*Z$ is a simple module up to $t^\Hf$-shift, then $Z$ is a simple module.
\end{Lem}
\begin{proof}
  The claim follows by the positivity of the structure constants of
  the simple basis.
\end{proof}

\begin{proof}[{Proof of Proposition \ref{prop:factorization_canonical}}]
We prove the claim by induction on the length $l$ of the word. 

If the adaptable word has length $l= |I|$, the statement is trivial.

Assume that the Proposition has been verified for any adaptable word
of length less than $l$. For any given word $\bi$ of
length $l$, $w=\sum_{k=1}^n w_k \beta_k$ with $w_k\in \N$, and any $1\leq s\leq l$ such that $s=s^{\max}$, we want to show the normalized product
$[S(\theta e_s)*S(w)]$ is a simple module. By its bar-invariance, it
suffices to show that this product is $(\prec,\bm)$-unitriangular to the simple basis,
or, equivalently, to the standard basis. 

Let $w_{\hat{1}}$, $w_{\hat{l}}$, $w_{\hat{1},\hat{l}}$ denote the
dimension vector obtained from $w$ by subtracting $w_1\beta_1$,
$w_l\beta_l$, and $w_1\beta_1+w_l\beta_l$ respectively.

(i) Assume $s\neq l$.

Notice that the
module $S(w)$ is $(\prec,\bm)$-unitriangular to the standard basis:
\begin{align*}
  S(w)=&\sum_{w^1\preceq w}c(w,w^1)M(w^1)\\
=&\sum_{w^1\preceq w}c(w,w^1)[M(w^1_{\hat{l}})*M(w^1_l\beta_l)].
\end{align*}

Therefore, we have the unitriangular decomposition
\begin{align*}
  [S(\theta e_s)*S(w)]=&\sum_{w^1\preceq w}c(w,w^1)[S(\theta e_s)*M(w^1_{\hat{l}})*M(w^1_l\beta_l)]\\
=&\sum_{w^1\preceq w}c(w,w^1)[[S(\theta
  e_s)*M(w^1_{\hat{l}})]*M(w^1_l\beta_l)]
\end{align*}

Notice that $M(w^1_{\hat{l}})$ is unitriangular to the basis
$\set{S(w^2)}$ of $\extQuotKGp(\bi)_{\hat{l}}$. The above is written as
\begin{align*}
   [S(\theta e_s)*S(w)]=&\sum_{w^1\preceq w,w^2\preceq w^1_{\hat{l}}}d(w,w^1,w^2)[[S(\theta
  e_s)*S(w^2)]*M(w^1_l\beta_l)]\\
&=\sum_{w^1\preceq w,w^2\preceq w^1_{\hat{l}}}d(w,w^1,w^2)[S(\theta
  e_s+w^2)*M(w^1_l\beta_l)],
\end{align*}
where we used induction hypothesis on $\extQuotKGp_{\hat{l}}\iso
\extQuotKGp(\bi-i_l)$ in the last equality. Here, the coefficients are contained in $\bm$ except the leading term with
$w^1=w,w^2=w^1_{\hat{l}}$ equals to $1$.

By Lemma \ref{lem:pbw_induction}, each factor is
$(\prec,\bm)$-unitriangular to the standard basis. 

(ii) Assume $s=l$. We divide the proof into two
cases.

(ii-1) If that $i_l\neq i_1$, the proof is similar to that of (i).

Consider $(\prec,\bm)$-unitriangular decomposition of $S(w)$:
\begin{align*}
  S(w)=&\sum_{w^1\preceq w}c(w,w^1)M(w^1)\\
=&\sum_{w^1\preceq w}c(w,w^1)[M(w^1_1\beta_1)*M(w^1_{\hat{1}})].
\end{align*}
Since $M(w^1_1\beta_1)$ commutes with $S(\theta
e_l)$ up to a $t^\Hf$-power, we have the following unitriangular decomposition
\begin{align*}
  [S(\theta  e_l)*S(w)]=&\sum_{w^1\preceq w}c(w,w^1)[M(w^1_1\beta_1)*S(\theta  e_l)*M(w^1_{\hat{1}})]\\
=&\sum_{w^1\preceq w}c(w,w^1)[M(w^1_1\beta_1)*[S(\theta 
  e_l)*M(w^1_{\hat{1}})]]
\end{align*}

Decomposing the terms $M(w^1_{\hat{1}})$ into simple basis elements, the above is written as
\begin{align*}
   [S(\theta  e_l)*S(w)]=&\sum_{w^1\preceq w,w^2\preceq w^1_{\hat{1}}}d(w,w^1,w^2)[M(w^1_1\beta_1)*[S(\theta 
  e_l)*S(w^2)]]\\
&=\sum_{w^1\preceq w,w^2\preceq w^1_{\hat{1}}}d(w,w^1,w^2)[M(w^1_1\beta_1)*S(\theta 
  e_l+w^2)],
\end{align*}
where we have used induction hypothesis $\extQuotKGp_{\hat{1}}\iso
\extQuotKGp(\bi-i_1)$ in the last equality. Here, the coefficients are contained in $\bm$
except the leading term with $w^1=w$, $w^2=w^1_{\hat{1}}$ equals to $1$.

By Lemma \ref{lem:pbw_induction}, the above expression is
$(\prec,\bm)$-unitriangular to the standard basis.

(ii-2) Assume that $i_l=i_1$.  

Consider the exchange relation at vertex $l[-1]$ in the seed
$\mu_{l[-2]}\cdots\mu_{1[1]}\mu_1t_0$ of the quantum cluster algebra
$\clAlg(\bi)$ (\cf \cite{GeissLeclercSchroeer10} for its ice quiver). Via the isomorphism $\kappa^{-1}$, it corresponds to the following quantum
$T$-system in $\extQuotKGp$:
\begin{align}\label{eq:T_system_substitute}
[S(\theta e_{l[-1]})*S(1[1],l)]=[S(\theta
e_l)*S(1[1],l[-1])+t^{-1}[\prod_{1\leq t<l,t=t^{\max}}S(\theta
e_t)^{-C_{i_l i_t}}].
\end{align}
Denote the last simple module $\prod_{1\leq t<l,t=t^{\max}}S(\theta
e_t)^{-C_{i_l i_t}}$ by $S(q)$

Let us take any $w'$ such that $w'_1=w'_l=0$. Notice that, in the quantum
torus $\cT(t_0)$, the leading monomials of $[S(\theta e_{l[-1]})*S(1[1],l)]$
and $S(q)$ differ by $Y_{l[-1]}\ldots Y_{l}Y_{1}$, which commutes
with the monomial $X^{\theta^{-1}w'}$ by the definition of compatible pair. If we deal with the type (ii),
this commutativity also follows from \eqref{eq:calculate_Euler_form}\eqref{eq:calculate_Euler_form_new}. Therefore, by using
the above $T$-system, we obtain
\begin{align*}
  [S(\theta e_l)*S(1[1],l[-1])*S(w')]=&[[S(\theta e_{l[-1]})*S(1[1],l)*S(w)']\\
  &-t^{-1}[S(q)*S(w')].
\end{align*}

By the induction hypothesis on $\extQuotKGp_{\hat{1}}$, $[S(1[1],l)*S(w)']$ is a
simple, which we denote by $S(p)$. By the induction hypothesis on either $\extQuotKGp_{\hat{1}}$ or
  $\extQuotKGp_{\hat{l}}$, $[S(q)*S(w')]$ is a simple, which we denote by
    $S(p')$. So we have
\begin{align*}
[S(\theta e_l)*S(1,[1],l[-1])*S(w')]=&[S(\theta e_{l[-1]})*S(p)]-t^{-1}S(p')
\end{align*}
By Lemma \ref{lem:pbw_induction}, $\RHS$ is unitriangular to the
standard basis. By induction hypothesis, on $\extQuotKGp_{\hat{1},\hat{l}}\iso \extQuotKGp(\bi-i_1-i_l)$, $S(1[1],l[-1])$ quasi-commutes with
$S(w')$. Also, $S(\theta e_l)$ quasi-commutes with $S(w')$ as it is a
frozen variable in the corresponding quantum cluster algebra. Therefore,
the normalized product $[S(\theta e_l)*S(1[1],l[-1])*S(w')]$ is
bar-invariant. Consequently, it is a simple module. By Lemma \ref{lem:simple_factor}, $[S(\theta e_l)*S(w')]$ is a simple module.

Consider the $(\prec,\bm)$-unitriangular decompositions of the module $S(w)$ as
\begin{align*}
  S(w)=&\sum_{w^1\preceq w} b(w,w^1)[M(w^1_1\beta_1)*M(w^1_{\hat{1},\hat{l}})*M(w^1_l\beta_l)]\\
=&\sum_{w^1\preceq w,w'\preceq w^1_{\hat{1},\hat{l}}} c(w,w^1,w')[M(w^1_1\beta_1)*S(w')*M(w^1_l\beta_l)]
\end{align*}
It follows that we have the $(\prec,\bm)$-unitriangular decomposition
\begin{align*}
  [S(\theta e_l)*S(w)]&=\sum d(w,w^1,w')[S(\theta  e_l)*M(w^1_1\beta_1)*S(w')*M(w^1_l\beta_l)]\\
&=\sum d(w,w^1,w')[M(w^1_1\beta_1)*[S(\theta  e_l)*S(w')]*M(w^1_l\beta_l)].
\end{align*}
Because $[S(\theta  e_l)*S(w')]$ is a simple
module, the above expansion of $S(\theta e_l)*S(w)$
is $(\prec,\bm)$-unitriangular to the standard basis by Lemma
\ref{lem:pbw_induction}. The desired claim follows by its bar-invariance.
\end{proof}

\subsection{Compatibility}

By Theorem \ref{thm:reduction}, in order to apply the existence theorem to the quantum cluster algebra
$\qClAlg(\bi)$, it remains to check the following property. 
\begin{Conj} \label{conj:desired_module}
The quantum cluster variables obtained along the sequence
$\sigma^{-1}\Sigma_\bi^{-1}$ from $t_0$ to $t_0[-1]$ are contained in $\can^{t_0}$.
\end{Conj}
We shall prove this conjecture for the quantum cluster algebras of type (ii) or of type (i) with an adaptable reduced word. Given that we have seen the triangular basis $\can^{t_0}$ is generated by the simples, Conjecture \ref{conj:desired_module} is a weaker form of the monoidal categorification Conjecture in the introduction.

In
Section \ref{sec:type_A}, we will give an example as a different quick proof
of this property for type $A_4$
based on the level-rank duality.

In this section, we verify that the triangular bases \wrt $t_0$ and $t_0[-1]$
are compatible when $\bi=c^{N+1}$ for some acyclic Coxeter word $c$ and $N\in\N$.

Instead of working with the previously introduced mutation sequence
$\Sigma_\bi$, we define the following sequences (read from right to left)
\begin{align*}
  \overleftarrow{\mu}^k=&\mu_{k^{\max}[-m_k^-]}\cdots
  \mu_{k^{\max}[-2]}\mu_{k^{\max}[-1]},\ 1\leq k\leq l,\\
  \sigma\overleftarrow{\mu}^k=&\mu_{k^{\min}[m_k^- -1]}\cdots
  \mu_{k^{\min}[1]}\mu_{k^{\min}},\ 1\leq k\leq l,\\
\Sigma^\bi=&(\sigma\overleftarrow{\mu}^l)^{-1}\cdots (\sigma
\overleftarrow{\mu}^2)^{-1} (\sigma\overleftarrow{\mu}^1)^{-1}.\\
\end{align*}
Notice that $\sigma$ is an involution. We have
$$
\sigma(\Sigma^\bi)^{-1}=(\sigma \Sigma^\bi)^{-1}=\overleftarrow{\mu}^1\cdots \overleftarrow{\mu}^2\overleftarrow{\mu}^l.
$$

\begin{Eg}\label{eg:inverse_sequence}
Let us continue Example \ref{eg:type_A}. The sequences
$\overleftarrow{\mu}^{10}=\mu_1\mu_5\mu_8$,  $\overleftarrow{\mu}^5=\mu_8$, and $\overleftarrow{\mu}^1$ is
trivial. The sequence $\Sigma^\bi$ and $(\sigma\Sigma^\bi)^{-1}$ (read
from right to left) are
given by
\begin{align*}
  \Sigma^\bi&=(1, 5, 8, 2, 6, 1, 5, 3, 2, 1)\\
(\sigma\Sigma^\bi)^{-1}&=(8, 6, 3, 5, 8, 2, 6, 1, 5, 8).
\end{align*}
\end{Eg}

Let  $\bi-i_l$ denote the word
obtained from $\bi$ by removing the last element $i_l$ and
$\Sigma^{\bi-i_l}$ the corresponding sequence. Then we have
$\Sigma^\bi=(\sigma\overleftarrow{\mu}^l)^{-1}\Sigma^{\bi-i_l}=\overleftarrow{\mu}^l\Sigma^{\bi-i_l}$. This observation
inductively implies the following result.

\begin{Lem}
For any $d\in\Z$, the seed $t[d]$ is injective-reachable via either
$(\sigma^d\Sigma_\bi,\sigma)$ or $(\sigma^d\Sigma^\bi,\sigma)$. Moreover, the
  sets of the quantum cluster variables obtained along both sequences
  starting from $t[d]$ are the same.
\end{Lem}
\begin{proof}
  The statement for $d=0$ is obvious by induction, from which we deduce the general
  statement by noticing that, up to the permutation $\sigma$ of vertices, all the seeds involving have the same
  principal part.
\end{proof}

Recall that when $\bi$ is adaptable, we can choose an adaptable
multidegree $\ua=(a_k)\in (2\Z)^l$, and identify the root vector
$\beta_k$ with the unit vector $e_{i_k,a_k}$ on graded quiver
varieties by the embedding $\iota_\ua$. The graded Grothendieck ring
$\extQuotKGp$ can be realized as $\extQuotKGp(\bi,\ua)$ \cf Section
\ref{sec:preliminaries} \ref{sec:quiver_variety}.

\begin{Prop}\label{prop:calculate_variable}
  Assume $\bi=c^{N+1}$. Then, for any $1\leq k< l$ such that
  $i_k=i_l$, the quantum cluster variables
  $X_{k}(\mu_k\mu_{k[1]}\cdots \mu_l t_0)$ is the simple module
  $\kappa^{-1} W^{(i_l)}_{m_k^+,m_k^-+1,a_k}$.
\end{Prop}
\begin{proof}
  It is easy to compute the Laurent expansions of these
  quantum cluster variables. By Proposition
  \ref{prop:KR_module_variant}(i)(ii), the Laurent expansions of the
  simple modules $\kappa^{-1}W^{(i_l)}_{m_k^+,m_k^-+1,a_k}$ take the same values.
\end{proof}

For any given mutation sequence $\overleftarrow{\mu}$ and integer $0\leq d\leq N-1$, let $\overleftarrow{\mu}_{\geq d}$ denote the mutation subsequence obtained from $\overleftarrow{\mu}$
by taking only the mutations on the vertices $i[d']$ with $i\in I$, $d'\geq d$, and
$\overleftarrow{\mu}_{<d}$ the subsequence by taking only the mutations on the vertices $i[d']$ with $i\in I$, $d'< d$. 

\begin{Eg}\label{eg:quiver_level}
 Let the Cartan matrix be of type $A_3$, $c=(3 2 1)$ and $N=4$. The ice quiver $\tQ$ associated with the initial seed $t_0$ is drawn in Figure \ref{fig:quiver_level}. Recall that $\overleftarrow{\mu}^{15}=\mu_3\mu_6\mu_9\mu_{12}$, $\overleftarrow{\mu}^{14}=\mu_2\mu_5\mu_8\mu_{11}$,$\overleftarrow{\mu}^{15}=\mu_1\mu_4\mu_7\mu_{10}$, $\overleftarrow{\mu}^{12}=\mu_6\mu_9\mu_{12}$, \ldots.
 
 Define the mutation sequence $\overleftarrow{\mu}_c=\mu_{12}\mu_9\mu_6\mu_3\mu_{11}\mu_8\mu_5\mu_2\mu_{10}\mu_7\mu_4\mu_1(=\overleftarrow{\mu}_3\overleftarrow{\mu}_2\overleftarrow{\mu}_1)$. Take
  $d=2$, then we obtain the subsequences $(\overleftarrow{\mu}_c)_{<2}=\mu_6\mu_3\mu_5\mu_2\mu_4\mu_1$ and $(\overleftarrow{\mu}_c)_{\geq 2}=\mu_{12}\mu_9\mu_{11}\mu_8\mu_{10}\mu_7$.
\end{Eg}

\begin{figure}[htb!]
 \centering
\beginpgfgraphicnamed{fig:quiver_level}
\begin{tikzpicture}
\node [shape=circle, draw] (v1) at (6.5,8) {1};
    \node [shape=circle, draw] (v2) at (6,6) {2};
    \node [shape=circle, draw] (v3) at (5.5,4) {3};

\node [shape=circle, draw] (v4) at (4.5,8) {4};
    \node [shape=circle, draw] (v5) at (4,6) {5};
    \node [shape=circle, draw] (v6) at (3.5,4) {6};

\node [shape=circle, draw] (v7) at (2.5,8) {7};
    \node [shape=circle, draw] (v8) at (2,6) {8};
    \node [shape=circle, draw] (v9) at (1.5,4) {9};

\node [shape=circle, draw] (v10) at (0.5,8) {10};
    \node [shape=circle, draw] (v11) at (0,6) {11};
    \node [shape=circle, draw] (v12) at (-0.5,4) {12};

\node [shape=diamond, draw] (v13) at (-1.5,8) {13};
    \node [shape=diamond, draw] (v14) at (-2,6) {14};
    \node [shape=diamond, draw] (v15) at (-2.5,4) {15};
    
\draw[-triangle 60] (v2) edge (v1); 
\draw[-triangle 60] (v3) edge (v2);
\draw[-triangle 60] (v5) edge (v4); 
\draw[-triangle 60] (v6) edge (v5);
\draw[-triangle 60] (v8) edge (v7); 
\draw[-triangle 60] (v9) edge (v8);
\draw[-triangle 60] (v11) edge (v10); 
\draw[-triangle 60] (v12) edge (v11);

\draw[-triangle 60] (v10) edge (v13);
\draw[-triangle 60] (v11) edge (v14);
\draw[-triangle 60] (v12) edge (v15);

\draw[-triangle 60] (v1) edge (v4); 

; 
\draw[-triangle 60] (v2) edge (v5); 

\draw[-triangle 60] (v3) edge (v6); 

\draw[-triangle 60] (v4) edge (v7); 
\draw[-triangle 60] (v7) edge (v10);

\draw[-triangle 60] (v5) edge (v8); 
\draw[-triangle 60] (v8) edge (v11);

\draw[-triangle 60] (v6) edge (v9); 
\draw[-triangle 60] (v9) edge (v12);

\draw[-triangle 60] (v4) edge (v2);
\draw[-triangle 60] (v5) edge (v3);
\draw[-triangle 60] (v7) edge (v5);
\draw[-triangle 60] (v8) edge (v6);
\draw[-triangle 60] (v10) edge (v8);
\draw[-triangle 60] (v11) edge (v9);
\draw[-triangle 60] (v13) edge (v11);
\draw[-triangle 60] (v14) edge (v12);

\draw[-triangle 60] (v15) edge (v14);
\draw[-triangle 60] (v14) edge (v13);

\end{tikzpicture}
\endpgfgraphicnamed
\caption{The quiver $Q$ with the word $\bi=(321)^5$.}
\label{fig:quiver_level}
\end{figure}

\begin{Lem}[Equivalent mutation sequences]\label{lem:equivalent_sequence}
(i) Let $\bi=c^3$ and denote $\overleftarrow{\mu}_c=(\mu_{2r} \mu_{r})\cdots (\mu_{r+2}\mu_2)(\mu_{r+1} \mu_1)$. Then the seed  $\overleftarrow{\mu}_ct_0$ is the same as the seed $(\mu_{2r}\cdots \mu_{r+1})(\mu_r\cdots \mu_2 \mu_1)t_0$. In other words, the mutation sequences $\overleftarrow{\mu}_c$ and $(\overleftarrow{\mu}_c)_{\geq 1}(\overleftarrow{\mu}_c)_{<1}$ starting from $t_0$ are equivalent.

(ii) Let $\bi=c^{N+1}$, then the mutation sequences $\overleftarrow{\mu}_c=\overleftarrow{\mu}_r\cdots \overleftarrow{\mu}_2\overleftarrow{\mu}_1$ and $(\overleftarrow{\mu}_c)_{\geq d}(\overleftarrow{\mu}_c)_{<d}$ starting from $t_0$ are equivalent for any $0\leq d\leq N-1$.

(iii) Let $\bi=c^{N+1}$ and $\overleftarrow{\mu}^c$ denote the sequence
$\overleftarrow{\mu}^{l-r+1}\ldots \overleftarrow{\mu}^{l-1}\overleftarrow{\mu}^{l}$, then the mutation sequences $\overleftarrow{\mu}^c$ and $\overleftarrow{\mu}^c_{< d}\overleftarrow{\mu}^c_{\geq d}$ starting from $t_0$ are equivalent for any $0\leq d\leq N-1$.
\end{Lem}
\begin{proof}
(i) This cluster algebra of level $N=2$ is not difficult and one can easily find that the ice quiver associated with both seeds $\Sigma_\bi t_0$ and $(\Sigma_\bi)_{\geq 1}(\Sigma_\bi)_{<1}t_0$ are the same. We still have to compare the cluster variables in the two seeds.

The reader could verify the claim directly by comparing cluster variables appearing along both sequences. Alternatively, identify this cluster algebra with the Grothendieck ring of representations of quantum affine algebras, such that the initial quantum cluster variable $X_k(t_0)$ is identified with the isoclass of the simple module $\simp(\theta e_k)=\simp(k^{\min},k)$.

we observe that the mutations along both sequences arise from $T$-systems and, consequently, the new cluster variables obtained by both sequences are identified with Kirillov-Reshetikhin modules $\simp(k^{min}[1],k[1])$, $1\leq k\leq 2r$. The claim follows.

(ii) Notice that the equivalence between mutation sequences is not affected by changing the frozen part of the ice quiver. Use ``$=$" to denote the equivalence. Repeatedly applying (i), we obtain
$\overleftarrow{\mu}_c=(\overleftarrow{\mu}_c)_{\geq 1}(\overleftarrow{\mu}_c)_{<1}=(\overleftarrow{\mu}_c)_{\geq 2}\cdot((\overleftarrow{\mu}_c)_{<2})_{\geq 1}\cdot(\overleftarrow{\mu}_c)_{<1}=(\overleftarrow{\mu}_c)_{\geq 2}(\overleftarrow{\mu}_c)_{<2}=\cdots =(\overleftarrow{\mu}_c)_{\geq {N-1}}(\overleftarrow{\mu}_c)_{<{N-1}}$.

(iii)
First observe that $\overleftarrow{\mu}^c=\overleftarrow{\mu}_c^{-1}$. It is a sequence from $\overleftarrow{\mu}_c t_0$ to $t_0$. By (ii), the sequences $\overleftarrow{\mu}^c$ and $(\overleftarrow{\mu}^c)_{<d}(\overleftarrow{\mu}^c)_{\geq d}$ starting from $\overleftarrow{\mu}_c t_0$ are equivalent. Because $\overleftarrow{\mu}_c t_0$ and $t_0$ share the same principal quiver $Q$ and the equivalence between mutation sequences is not affected by the frozen part of the ice quiver, we deduce that $\overleftarrow{\mu}^c$ and  $(\overleftarrow{\mu}^c)_{<d}(\overleftarrow{\mu}^c)_{\geq d}$ starting from $t_0$ are equivalent.
\end{proof}

\begin{Lem}\label{lem:enough_variable}
Assume $\bi=c^{N+1}$. Then all the quantum cluster variables appearing
along the
sequence $(\sigma\Sigma^\bi)^{-1}$ from $t_0$ to $t_0[-1]$ also appear along the
sequence $(\overleftarrow{\mu}^c)^N$.
\end{Lem}
\begin{proof}
  The claim follows by tracking the exchange relations
  involved. Recall that the quantum cluster variables $X_i(t)$, as
  well as their specialization $x_i(t)=X_i(t)|_{q^\Hf\mapsto 1}$, are in
  bijection with the object $T_i(t)$ in Section \ref{sec:ice_quiver}. By using Calabi-Yau reduction,
  it suffices to remove the frozen vertices and verify that the commutative cluster variables
  $x_i(t)$ appearing on the sequence $(\sigma\Sigma^\bi)^{-1}$
  starting from quiver $Q(t_0)$ are contained in those appearing on
  the sequence $(\overleftarrow{\mu}^c)^N$.

For any given $0\leq d\leq N-1$, by Lemma \ref{lem:equivalent_sequence}(iii), we see that
$\overleftarrow{\mu}^c$ and $\overleftarrow{\mu}^c_{<d}\overleftarrow{\mu}^c_{\geq
d}$ give identical collection of cluster variables starting from the initial principal quiver $Q(t_0)$. Because $(\sigma\Sigma^\bi)^{-1}=\overleftarrow{\mu}^c_{\geq
{N-1}}\cdots\overleftarrow{\mu}^c_{\geq
1}\overleftarrow{\mu}^c_{\geq
0}$, it suffices to verify that, starting from the quiver $Q(t_0)$,
the mutation sequences
$$
(\overleftarrow{\mu}^c_{<N-1}\overleftarrow{\mu}^c_{\geq
N-1})\cdots (\overleftarrow{\mu}^c_{<1}\overleftarrow{\mu}^c_{\geq
1})(\overleftarrow{\mu}^c_{<0}\overleftarrow{\mu}^c_{\geq
0})
$$
and
$$ 
(\overleftarrow{\mu}^c_{<N-1}\cdots \overleftarrow{\mu}^c_{<1}  \overleftarrow{\mu}^c_{<0})(\overleftarrow{\mu}^c_{\geq
N-1}\cdots\overleftarrow{\mu}^c_{\geq
1}\overleftarrow{\mu}^c_{\geq
0}).
$$
give identical collection of cluster
variables.

Let us denote $A_d=\overleftarrow{\mu}^c_{\geq
d+1}$ and $B_d=\overleftarrow{\mu}^c_{< d}\cdots
\overleftarrow{\mu}^c_{< 1} \overleftarrow{\mu}^c_{<0}$. In the quiver $Q'=\overleftarrow{\mu}^c_{\geq
d}\cdots \overleftarrow{\mu}^c_{\geq
0}Q(t_0)$, the set of
vertices acted by $A_d$ and that acted by $B_d$ are separated by the full
subquiver on $\set{^{\min}i[d]|i\in I}$ of $Q'$ (cf. Example \ref{eg:quiver_seperation}). By definition of
mutations, this property is
preserved by the
quivers appearing along $A_dB_d$ or $B_dA_d$. Therefore, the mutation sequences $A_dB_d$ and $B_dA_d$ give identical
set of vertices. The claim follows by recursively swapping $A_d$ and $B_d$ for all $d$.
\end{proof}

\begin{Eg}\label{eg:quiver_seperation}
  Let we continue Example \ref{eg:quiver_level}. Take
  $d=2$, then the principal quiver $Q'=\overleftarrow{\mu}^c_{\geq
d}\cdots \overleftarrow{\mu}^c_{\geq
0}Q(t_0)$ is drawn in Figure \ref{fig:quiver_seperation}. Take
$A=\overleftarrow{\mu}^c_{\geq d+1}=\mu_{10}\mu_{11}\mu_{12}$ and
$B=\overleftarrow{\mu}^c_{<2}\overleftarrow{\mu}^c_{<1}=\mu_1\mu_4\mu_2\mu_5\mu_3\mu_6\mu_1\mu_2\mu_3$. The
sets $\set{1,2,3,4,5,6}$ and $\set{10,11,12}$ are always separated by
the full subquiver on $\set{7,8,9}$ of the quivers appearing along mutation
sequence $AB$ or $BA$.

Moreover, any such quiver appearing is a union of small blocks in Figure \ref{fig:quiver_block} for
$i\neq j\in I$, $0\leq d'\leq N-2$, with at
most one diagonal edge of multiplicity $-C_{i j}$.
\end{Eg}

\begin{figure}[htb!]
 \centering
\beginpgfgraphicnamed{fig:quiver_seperation}
\begin{tikzpicture}
\node [shape=circle, draw] (v1) at (6,8) {1};
    \node [shape=circle, draw] (v2) at (6,6) {2};
    \node [shape=circle, draw] (v3) at (6,4) {3};

\node [shape=circle, draw] (v4) at (4,8) {4};
    \node [shape=circle, draw] (v5) at (4,6) {5};
    \node [shape=circle, draw] (v6) at (4,4) {6};

\node [shape=circle, draw] (v7) at (2,8) {7};
    \node [shape=circle, draw] (v8) at (2,6) {8};
    \node [shape=circle, draw] (v9) at (2,4) {9};

\node [shape=circle, draw] (v10) at (0,8) {10};
    \node [shape=circle, draw] (v11) at (0,6) {11};
    \node [shape=circle, draw] (v12) at (0,4) {12};

\draw[-triangle 60] (v2) edge (v1); 
\draw[-triangle 60] (v3) edge (v2);
\draw[-triangle 60] (v5) edge (v4); 
\draw[-triangle 60] (v6) edge (v5);
\draw[-triangle 60] (v8) edge (v7); 
\draw[-triangle 60] (v9) edge (v8);
\draw[-triangle 60] (v11) edge (v10); 
\draw[-triangle 60] (v12) edge (v11);

\draw[-triangle 60] (v4) edge (v1); 

; 
\draw[-triangle 60] (v5) edge (v2); 

\draw[-triangle 60] (v6) edge (v3); 

\draw[-triangle 60] (v7) edge (v4); 
\draw[-triangle 60] (v7) edge (v10);

\draw[-triangle 60] (v8) edge (v5); 
\draw[-triangle 60] (v8) edge (v11);

\draw[-triangle 60] (v9) edge (v6); 
\draw[-triangle 60] (v9) edge (v12);

\draw[-triangle 60] (v1) edge (v5);
\draw[-triangle 60] (v2) edge (v6);

\draw[-triangle 60] (v4) edge (v8);
\draw[-triangle 60] (v5) edge (v9);
\draw[-triangle 60] (v10) edge (v8);
\draw[-triangle 60] (v11) edge (v9);

\end{tikzpicture}
\endpgfgraphicnamed
\caption{A quiver $Q'$}
\label{fig:quiver_seperation}
\end{figure}
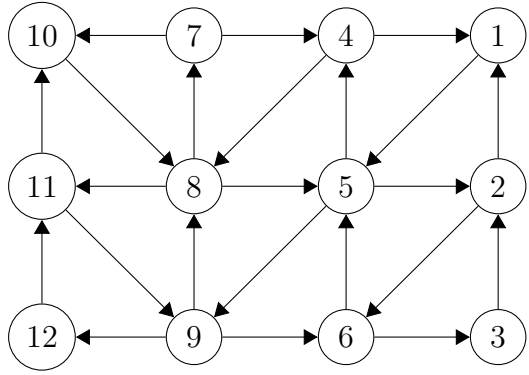

\begin{figure}[htb!]
 \centering
\beginpgfgraphicnamed{fig:quiver_block}
\begin{tikzpicture}

\node (v1) at (6,8) {$i[d']$};
    \node  (v2) at (6,6) {$j[d']$};
    \node  (v3) at (4,8) {$i[d'+1]$};
    \node (v4) at (4,6) {$j[d'+1]$};

\draw (v1) edge node[right]{$-C_{ij}$} (v2);
\draw (v1) edge (v3);
\draw (v3) edge node[left]{$-C_{ij}$} (v4);
\draw (v4) edge (v2);
\end{tikzpicture}
\endpgfgraphicnamed
\caption{small block }
\label{fig:quiver_block}
\end{figure}

Recall that the set $I$ denote $\set{1,2,\ldots,r}$.
\begin{Prop}\label{prop:compatible} 
  When $\bi=c^{N+1}$, the triangular bases in the seeds $t_0$ and
  $t_0[-1]$ are compatible. Moreover, they contain all the
  quantum cluster variables appearing along the sequence
  $(\sigma\Sigma^\bi)^{-1}$ from $t_0$ to $t_0[-1]$.
\end{Prop}
\begin{proof}
Repeatedly applying Proposition
\ref{prop:condition_change_seed}, we have the following argument:

\emph{
Let $\can^t$ be the triangular basis \wrt a seed $t$ and assume it to be positive. Let $u_1,u_2$ be mutation sequences such that their composition $\overleftarrow{u}=u_2u_1$ is a mutation sequence from $t$ to $t[1]$. If the cluster variables appearing along the mutation sequences from $u_2u_1$ and from $u_2^{-1}$ starting from $t$ are contained in the initial triangular basis $\can^t$ of $t$, then the common triangular basis \wrt the seeds along the mutation sequence $u_2^{-1}$ starting from $t$ to $u_2^{-1}t$ exists.}

Recall that, for $\bi=(i_r,\cdots,i_2,i_1)^{N+1}$, the vertices of the ice quiver for the initial seed $t_0=t_0(\bi)$ are given by \begin{align*}
&(1,2,\cdots,r,r+1,\cdots,2r,\cdots,l)\\
=&(^{\min}i_1,^{\min}i_2,\cdots, ^{\min}i_r,^{\min}i_1[1],\cdots, ^{\min}i_r[1],\cdots,^{\min}i_r[N])\\
=&(^{\max}i_1[-N],^{\max}i_2[-N],\cdots, ^{\max}i_r[-N],^{\max}i_1[-N+1],\\
&\cdots, ^{\max}i_r[-N+1],\cdots,^{\max}i_r)
\end{align*}
 where $l=(N+1)r$.

We have shown that the triangular basis $\can^{t_0}$ \wrt $t_0$ exists and is positive. Now observe that all quantum cluster variables along the mutation
sequence $\Sigma^\bi$
from $t_0$ to $t_0[1]$ are known to be Kirillov-Reshetikhin modules,
and, in particular, contained in the initial triangular basis
$\can^t_0$. Also, the cluster variables obtained along the sequence $\overleftarrow{\mu}^l=\mu_{^{\max}i_r[-N]}\cdots \mu_{^{\max}i_r[-2]}\mu_{^{\max}i_r[-1]}$ from
$t_0$ to $\overleftarrow{\mu}^l t_0$ are in $\can^{t_0}$ by Proposition
\ref{prop:calculate_variable}. By the previous argument, the common triangular basis \wrt the seeds along the sequence $\overleftarrow{\mu}^l$ from
$t_0$ to $\overleftarrow{\mu}^l t_0$ exists.

Notice that the principal quiver $Q(\overleftarrow{\mu}^l t_0)$ is the same as that of
the ice quiver associated with the word
$\bi'=(i_{r-1}, \ldots, i_1,i_r)^{N+1}$ for the acyclic Coxeter word
$(i_{r-1}, \ldots, i_1,i_r)$ (cf. Example \ref{eg:similar_after_mutation}). Applying Proposition
\ref{prop:calculate_variable} to the seed $t_0(\bi')$ associated with $\bi'$, we obtain that the quantum cluster variables along the sequence $\mu_{^{\max}i_{r-1}[-N]}\cdots\mu_{^{\max}i_{r-1}[-2]}\cdots\mu_{^{\max}i_{r-1}[-1]}$ starting from
$t_0(\bi')$ are contained in the triangular basis $\can^{t_0(\bi')}$ for this new seed $t_0(\bi')$. Consequently, we deduce that the
quantum cluster variables obtained along the sequence $\overleftarrow{\mu}^{l-1}=\mu_{^{\max}i_{r-1}[-N]}\cdots\mu_{^{\max}i_{r-1}[-2]}\mu_{^{\max}i_{r-1}[-1]}$
starting from $\overleftarrow{\mu}^l t_0$ is contained in the canonical basis
$\can^{\overleftarrow{\mu}^l t_0}$ for the seed $\overleftarrow{\mu}^l t_0$, because this property is shared by similar seeds $t_0(\bi')$ and $\overleftarrow{\mu}^l t_0$, cf. Lemma \ref{lem:variation_triangular_basis}(iii)). By the previous argument, the common triangular basis \wrt the seeds along the sequence $\overleftarrow{\mu}^{l-1}$ from
$\overleftarrow{\mu}^l t_0$ to $\overleftarrow{\mu}^{l-1}\overleftarrow{\mu}^l t_0$ exists.

Repeating this process, we conclude that the common triangular basis \wrt the seeds along the sequence $(\overleftarrow{\mu}^c)^N$ from
$t_0$ to $(\overleftarrow{\mu}^c)^N t_0$ exists. In particular, the initial triangular basis $\can^{t_0}$ contains all the
quantum cluster variables along this sequence. By Lemma
\ref{lem:enough_variable}, $\can^{t_0}$ already contains the quantum cluster
variables along the sequence $(\sigma\Sigma^\bi)^{-1}$ starting from
$t_0$. Repeating the previous argument, we obtain that the common triangular basis \wrt the seeds along the mutation sequence $(\sigma\Sigma^\bi)^{-1}$ from $t_0$ and $t_0[-1]$ exists.
\end{proof}

\begin{Eg}\label{eg:similar_after_mutation}
Let us continue Example \ref{eg:quiver_level}. Applying Proposition \ref{prop:calculate_variable} to the initial seed $t_0$, we know that the cluster variables appearing along the mutation sequence $\overleftarrow{\mu}^l=\mu_{^{\max}3[-N]}\cdots \mu_{^{\max}3[-1]}=\mu_3\mu_6\mu_9\mu_{12}$ correspond to Kirillov-Reshetikhin modules and belong in the triangular basis $\can^{t_0}$. Consequently, by Proposition \ref{prop:condition_change_seed}, we have the common triangular basis \wrt all seeds along this mutation sequence from $t_0$ to $\overleftarrow{\mu}^l t_0$.

Next, look at the ice quiver of the $\overleftarrow{\mu}^l t_0$ drawn in figure \ref{fig:mutated_quiver}. It shares the same principal part with the ice quiver associated with the word $\bi'=(213)^5$, cf. Figure \ref{fig:similar_mutated_quiver}. In other word, the seeds $\overleftarrow{\mu}^l t_0$ and $t_0(\bi')$ are similar. Applying Proposition \ref{prop:calculate_variable} to the seed $t_0(\bi')$, we know that the cluster variables appearing along the mutation sequence $\mu_{^{\max}2[-1]}\cdots \mu_{^{\max}2[-N]}=\mu_3\mu_6\mu_9\mu_{12}$ belong to the triangular basis $\can^{t_0(\bi')}$. Therefore, this property holds for the similar seed $\overleftarrow{\mu}^l t_0$ as well, namely, the cluster variables appearing along the mutation sequence $\overleftarrow{\mu}^{l-1}==\mu_2\mu_5\mu_8\mu_{11}$ are contained in the triangular basis of $\can^{\overleftarrow{\mu}^l t_0}$. Consequently, we have the common triangular basis \wrt all seeds along the mutation sequence $\overleftarrow{\mu}^{l-1}$ from $\overleftarrow{\mu}^lt_0$ to $\overleftarrow{\mu}^{l-1}\overleftarrow{\mu}^l t_0$.
\end{Eg}

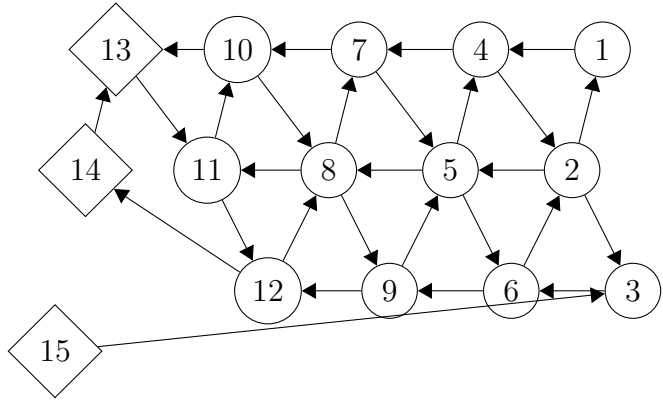
\begin{figure}[htb!]
 \centering
\beginpgfgraphicnamed{fig:mutated_quiver}
\begin{tikzpicture}[scale=0.8]
\node [shape=circle, draw] (v1) at (6.5,8) {1};
    \node [shape=circle, draw] (v2) at (6,6) {2};
    \node [shape=circle, draw] (v3) at (7,4) {3};

\node [shape=circle, draw] (v4) at (4.5,8) {4};
    \node [shape=circle, draw] (v5) at (4,6) {5};
    \node [shape=circle, draw] (v6) at (5,4) {6};

\node [shape=circle, draw] (v7) at (2.5,8) {7};
    \node [shape=circle, draw] (v8) at (2,6) {8};
    \node [shape=circle, draw] (v9) at (3,4) {9};

\node [shape=circle, draw] (v10) at (0.5,8) {10};
    \node [shape=circle, draw] (v11) at (0,6) {11};
    \node [shape=circle, draw] (v12) at (1,4) {12};

\node [shape=diamond, draw] (v13) at (-1.5,8) {13};
    \node [shape=diamond, draw] (v14) at (-2,6) {14};
    \node [shape=diamond, draw] (v15) at (-2.5,3) {15};
    
\draw[-triangle 60] (v2) edge (v1); 
\draw[-triangle 60] (v2) edge (v3);
\draw[-triangle 60] (v5) edge (v4); 
\draw[-triangle 60] (v5) edge (v6);
\draw[-triangle 60] (v8) edge (v7); 
\draw[-triangle 60] (v8) edge (v9);
\draw[-triangle 60] (v11) edge (v10); 
\draw[-triangle 60] (v11) edge (v12);

\draw[-triangle 60] (v10) edge (v13);
\draw[-triangle 60] (v14) edge (v13);

\draw[-triangle 60] (v1) edge (v4); 

; 
\draw[-triangle 60] (v2) edge (v5); 

\draw[-triangle 60] (v3) edge (v6); 

\draw[-triangle 60] (v4) edge (v7); 
\draw[-triangle 60] (v7) edge (v10);

\draw[-triangle 60] (v5) edge (v8); 
\draw[-triangle 60] (v8) edge (v11);

\draw[-triangle 60] (v6) edge (v9); 
\draw[-triangle 60] (v9) edge (v12);

\draw[-triangle 60] (v4) edge (v2);
\draw[-triangle 60] (v6) edge (v2);
\draw[-triangle 60] (v7) edge (v5);
\draw[-triangle 60] (v9) edge (v5);
\draw[-triangle 60] (v10) edge (v8);
\draw[-triangle 60] (v12) edge (v8);
\draw[-triangle 60] (v13) edge (v11);
\draw[-triangle 60] (v12) edge (v14);

\draw[-triangle 60] (v15) edge (v3);

\end{tikzpicture}
\endpgfgraphicnamed
\caption{The ice quiver of the mutated seed $\mu_3\mu_6\mu_9\mu_{12}t_0((321)^5)$}
\label{fig:mutated_quiver}
\end{figure}

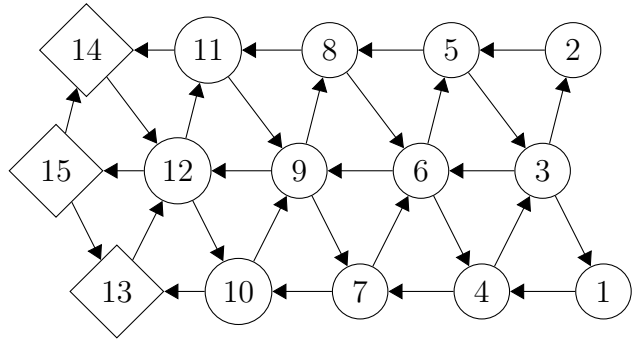
\begin{figure}[htb!]
 \centering
\beginpgfgraphicnamed{fig:similar_mutated_quiver}
\begin{tikzpicture}[scale=0.8]
\node [shape=circle, draw] (v1) at (6.5,8) {2};
    \node [shape=circle, draw] (v2) at (6,6) {3};
    \node [shape=circle, draw] (v3) at (5,4) {4};

\node [shape=circle, draw] (v4) at (4.5,8) {5};
    \node [shape=circle, draw] (v5) at (4,6) {6};
    \node [shape=circle, draw] (v6) at (3,4) {7};

\node [shape=circle, draw] (v7) at (2.5,8) {8};
    \node [shape=circle, draw] (v8) at (2,6) {9};
    \node [shape=circle, draw] (v9) at (1,4) {10};

\node [shape=circle, draw] (v10) at (0.5,8) {11};
    \node [shape=circle, draw] (v11) at (0,6) {12};
    \node [shape=diamond, draw] (v12) at (-1,4) {13};

\node [shape=diamond, draw] (v13) at (-1.5,8) {14};
    \node [shape=diamond, draw] (v14) at (-2,6) {15};

    \node [shape=circle, draw] (v15) at (7,4) {1};
    
\draw[-triangle 60] (v2) edge (v1); 
\draw[-triangle 60] (v3) edge (v2);
\draw[-triangle 60] (v5) edge (v4); 
\draw[-triangle 60] (v6) edge (v5);
\draw[-triangle 60] (v8) edge (v7); 
\draw[-triangle 60] (v9) edge (v8);
\draw[-triangle 60] (v11) edge (v10); 
\draw[-triangle 60] (v12) edge (v11);

\draw[-triangle 60] (v10) edge (v13);
\draw[-triangle 60] (v11) edge (v14);
\draw[-triangle 60] (v2) edge (v15);

\draw[-triangle 60] (v1) edge (v4); 

; 
\draw[-triangle 60] (v2) edge (v5); 

\draw[-triangle 60] (v3) edge (v6); 

\draw[-triangle 60] (v4) edge (v7); 
\draw[-triangle 60] (v7) edge (v10);

\draw[-triangle 60] (v5) edge (v8); 
\draw[-triangle 60] (v8) edge (v11);

\draw[-triangle 60] (v6) edge (v9); 
\draw[-triangle 60] (v9) edge (v12);

\draw[-triangle 60] (v4) edge (v2);
\draw[-triangle 60] (v5) edge (v3);
\draw[-triangle 60] (v7) edge (v5);
\draw[-triangle 60] (v8) edge (v6);
\draw[-triangle 60] (v10) edge (v8);
\draw[-triangle 60] (v11) edge (v9);
\draw[-triangle 60] (v13) edge (v11);
\draw[-triangle 60] (v14) edge (v12);

\draw[-triangle 60] (v15) edge (v3);

\draw[-triangle 60] (v14) edge (v13);

\end{tikzpicture}
\endpgfgraphicnamed
\caption{The ice quiver for the word $\bi'=(213)^5$}
\label{fig:similar_mutated_quiver}
\end{figure}

\subsection{Compatibility in type $A$ via level-rank duality}\label{sec:type_A}
In this section, we consider the special case when $C$ is of type
$A_{r}$ and $\bi=(1,2,1,3,2,1
\ldots,r-1\ldots,2,1,r,\ldots ,2,1)$. Then $w_{\bi}$ is the longest element in the Weyl group. The pair $(C,\bi)$
is of both type (i) and (ii).

In the following example, we check Conjecture
\ref{conj:desired_module} for the special case $A_4$, which,
together with Proposition \ref{prop:initial_seed_cond}, will
imply the type $A_4$ case in Theorem \ref{thm:consequence}
\ref{thm:simple_module} \ref{thm:canonical_basis}. This approach seems to be effective for general $r$, but we don't pursue the generalization in this paper.

\begin{Eg}
  Consider Figure \ref{fig:type_A}. The exchange relations along the
  mutation sequence $\Sigma^\bi=(1, 5, 8, 2, 6, 1, 5, 3, 2, 1)$ are $T$-system
  relations. Consider the trivial permutation $\var^*$ of
  $[1,l]$. It induces an isomorphism from principal quiver $Q(t_0')$ (Figure \ref{fig:type_A_rotated}) to the principal quiver $\tQ(t_0)$. We have a variation map $\Var$ from
$\ptSet(t)$ to $\ptSet(t')$. The mutation sequence (read from right to left)
$(\sigma\Sigma^\bi)^{-1}=(8, 6, 3, 5, 8, 2, 6, 1, 5, 8)$ becomes a new
sequence on $\tQ(t_0')$. It is straightforward to check that the
quantum cluster variables appearing along this sequence correspond to
  Kirillov-Reshetikhin modules for the seed $t_0'$, which is again
  associated with the longest word. In particular, they are contained in the triangular basis $\can^{t_0'}$. By the correction technique in Lemma \ref{lem:variation}, similar quantum cluster variables obtained along $(\sigma\Sigma^\bi)^{-1}$ are contained in $\can^{t_0}$.
\end{Eg}
\begin{figure}[htb!]
 \centering
\beginpgfgraphicnamed{fig:type_A}
  \begin{tikzpicture}[scale=0.8]
 \node [shape=diamond, draw] (v1) at (7,4) {4}; 
    \node  [shape=diamond, draw] (v2) at (6,3) {7}; 
    \node [shape=diamond,  draw] (v3) at (5,2) {9};
\node [shape=diamond, draw] (v4) at (4,1) {10}; 

\node  [shape=circle, draw] (v5) at (5,4) {1}; 
    \node [shape=circle,  draw] (v6) at (4,3) {2};
    
\node [shape=circle, draw] (v7) at (3,2) {3}; 
    \node  [shape=circle, draw] (v8) at (3,4) {5}; 
    \node [shape=circle,  draw] (v9) at (2,3) {6};

\node [shape=circle, draw] (v10) at (1,4) {8};

    \draw[-triangle 60] (v4) edge (v3); 
    \draw[-triangle 60] (v3) edge (v2); 
    \draw[-triangle 60] (v2) edge (v1);
    
    \draw[-triangle 60] (v7) edge (v6); 
    \draw[-triangle 60] (v6) edge (v5);
    
    \draw[-triangle 60] (v9) edge (v8);

    \draw[-triangle 60] (v10) edge (v9); 
    \draw[-triangle 60] (v9) edge (v7); 
    \draw[-triangle 60] (v7) edge (v4);
    
    \draw[-triangle 60] (v8) edge (v6); 
    \draw[-triangle 60] (v6) edge (v3);
    
    \draw[-triangle 60] (v5) edge (v2);

    \draw[-triangle 60] (v1) edge (v5); 
    \draw[-triangle 60] (v5) edge (v8); 
    \draw[-triangle 60] (v8) edge (v10);

    \draw[-triangle 60] (v2) edge (v6); 
    \draw[-triangle 60] (v6) edge (v9); 

\draw[-triangle 60] (v3) edge (v7);
  \end{tikzpicture}
\endpgfgraphicnamed
\caption{Ice quiver $\tQ(t_0')$ obtained by rotating the quiver in Figure
  \ref{fig:type_A} and changing coefficient pattern.}
\label{fig:type_A_rotated}
\end{figure}
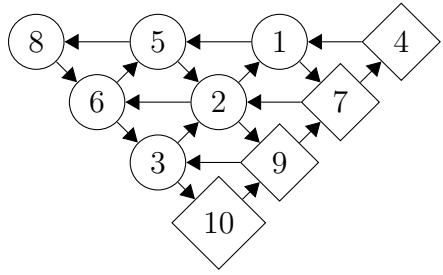

\subsection{Consequences}
\label{sec:consequence}

We summarize the consequences of the previous discussion.

We have seen that the simple basis $\set{\simp (w)}$ provides the
initial triangular basis \wrt the seed $t_0$ after localization. By Proposition
\ref{prop:compatible}, we have verified all the conditions demanded in Reduced Existence Theorem (Theorem
\ref{thm:reduction}) for the word $c^{N+1}$. Notice that, to any adaptable
word $\bi$, we can associate\footnote{The quiver associated with an $\overline{\Omega}$-adaptable word $\bi$ is a connected full subquiver $\tQ$ in the Auslander-Reiten quiver of the bounded derived category of $(\Gamma,\Omega)$-representations, such that its intersection with each $\tau$-orbit is connected, where $\tau$ denote the Auslander-Reiten translation. We can thus expand $\tQ$ following the direction of $\tau$ into a quiver associated with $c^{N+1}$, which has an initial adaptable word $\bi'$ producing $\tQ$. Notice that, for type (i), $\tQ$ is categorified by a terminal $(\Gamma,\Omega)$-representation, cf. \cite[Section 2.2]{GeissLeclercSchroeer07b} for detailed treatment.} an initial subword $\bi'$ of some
$c^{N+1}$, such that $w_\bi=w_{\bi'}$ and, consequently, $\clAlg(\bi)$
agrees with $\clAlg(\bi')$. Notice that $\bi'$ is adaptable as well. By working with
type (ii) quantum cluster algebras, we have the
embedding from $\extQuotKGp(\bi',\ua')$ to $\extQuotKGp(c^{N+1},\ua)$ and the
following result. 
\begin{Thm}\label{thm:consequence}
Assume that $C$ is symmetric and $\bi$ an adaptable word. Then the
simple basis $\set{\simp (w)}$, after localization at the frozen
variables, gives the common triangular basis of the
quantum cluster algebra $\qClAlg(\bi)$. In particular, it
verifies Fock-Goncharov Conjecture, and it contains all quantum cluster monomials. 
\end{Thm}

The following result is a direct consequence.
\begin{Thm}\label{thm:simple_module}
  Conjecture \ref{conj:simple_module} holds true. In particular, the quantum cluster
monomials in $\clAlg$ of type (ii) are simple modules of quantum
affine algebras. 
\end{Thm}

For quantum cluster algebras arising from quantum unipotent subgroups, we have verified Conjecture \ref{conj:integral_quantum_group} in
Theorem \ref{thm:isom_integral}. Moreover, Theorem \ref{thm:consequence} implies that the quantum cluster monomials of
$\clAlg$ of type (i) are contained in the dual canonical basis for the
following cases.
\begin{Thm}\label{thm:canonical_basis}
Conjecture \ref{conj:canonical_basis} holds true for any reduced word
$w$ which is inferior than some adaptable word with respect to the left
or right weak order. In particular, it holds true for any reduced word $w$ when the Cartan
type is $ADE$.
\end{Thm}
\begin{proof}
  The first claim follows from Theorem \ref{thm:consequence} and
  compositions of the embeddings of Grothendieck ring in the end of Section \ref{sec:category_type_i}.

For type $ADE$, the claim holds for the longest word because
  it is adaptable. Any reduced word $w$ is inferior than the
  longest word with respect to the left and right weak order, \cf
  \cite{bjorner2005combinatorics}. The second claim follows.
\end{proof}

\begin{Rem}\label{rem:Leclerc_conjecture}
  When the quantum cluster algebra is of type (i). The triangularity property of the common triangular basis might be compared with the following weaker form of Leclerc's conjecture
  \cite[Section 3.1]{Leclerc03}:

Let $b_1, b_2$ be elements of the dual canonical basis $B^*$. Suppose
that $b_1$ corresponds to a cluster variable. Then the expansion of
$b_1,b_2$ on $B^*$ takes the form
\begin{align}\label{eq:Leclerc_conjecture}
  b_1b_2=q^sb''+\sum_{c\neq b''}\gamma^c_{b_1b_2}(q)c,
\end{align}
where $s\in\Z$, $\gamma^c_{b_1 b_2}\in q^{s-1}\Z[q^{-1}]$.

Leclerc's conjecture was recently proved in \cite{KKKO14}.
\end{Rem}




\begin{thebibliography}{LLRZ14b}

\bibitem[Ami09]{Amiot09}
Claire Amiot, \emph{Cluster categories for algebras of global dimension $2$ and
  quivers with potential}, Annales de l'institut {F}ourier \textbf{59} (2009),
  no.~6, 2525--2590, \href {http://arxiv.org/abs/0805.1035}
  {\path{arXiv:0805.1035}}.

\bibitem[BB05]{bjorner2005combinatorics}
Anders Bjorner and Francesco Brenti, \emph{Combinatorics of coxeter groups},
  Springer, 2005.

\bibitem[BFZ05]{BerensteinFominZelevinsky05}
Arkady Berenstein, Sergey Fomin, and Andrei Zelevinsky, \emph{Cluster algebras.
  {III}. {U}pper bounds and double {B}ruhat cells}, Duke Math. J. \textbf{126}
  (2005), no.~1, 1--52.

\bibitem[BZ05]{BerensteinZelevinsky05}
Arkady Berenstein and Andrei Zelevinsky, \emph{Quantum cluster algebras}, Adv.
  Math. \textbf{195} (2005), no.~2, 405--455, \href
  {http://arxiv.org/abs/math/0404446v2} {\path{arXiv:math/0404446v2}}.

\bibitem[BZ14]{BerensteinZelevinsky12}
\bysame, \emph{Triangular bases in quantum cluster algebras}, International Mathematics Research Notices \textbf{2014} (2014), no.~6, 1651--1688, \href
  {http://arxiv.org/abs/1206.3586} {\path{arXiv:1206.3586}}.

\bibitem[Dav16]{Davison16}
\bysame, \emph{Positivity for quantum cluster algebras}, 2016, \href
  {http://arxiv.org/abs/1601.07918} {\path{arXiv:1601.07918}}.

\bibitem[DK07]{DehyKeller07}
Raika Dehy and Bernhard Keller, \emph{On the combinatorics of rigid objects in
  $2$-{C}alabi-{Y}au categories}, Int. Math. Res. Not. IMRN, \textbf{2008} (2008), no.~11, \href {http://arxiv.org/abs/0709.0882}
  {\path{arXiv:0709.0882}}.

\bibitem[DWZ08]{DerksenWeymanZelevinsky07}
Harm Derksen, Jerzy Weyman, and Andrei Zelevinsky, \emph{Quivers with
  potentials and their representations {I}: Mutations}, Selecta Mathematica
  \textbf{14} (2008), no.~1, 59--119.

\bibitem[DWZ10]{DerksenWeymanZelevinsky09}
\bysame, \emph{Quivers with potentials and their representations {II}:
  {Applications to cluster algebras}}, J. Amer. Math. Soc. \textbf{23} (2010),
  no.~3, 749--790, \href {http://arxiv.org/abs/0904.0676v2}
  {\path{arXiv:0904.0676v2}}.

\bibitem[FG06]{FockGoncharov03conj}
V.~V. Fock and A.~B. Goncharov,  \emph{Moduli spaces of local
systems and higher Teichm\"{u}ller theory}, Publications Mat\'ematiques de
l'Institut des Hautes \'Etudes Scientifiques \textbf{103} (2006), no.~1, 1--211.

\bibitem[FG09]{FockGoncharov03}
\bysame, \emph{{Cluster ensembles, quantization and the
  dilogarithm}}, Ann. Sci. \'Ecole Norm. Sup. (4) \textbf{42} (2009), no.~6,
  865--930, \href {http://arxiv.org/abs/math.AG/0311245}
  {\path{arXiv:math.AG/0311245}}.


\bibitem[FH14]{fourier2014schur}
Ghislain Fourier and David Hernandez, \emph{Schur positivity and
  {K}irillov-{R}eshetikhin modules}, SIGMA. Symmetry, Integrability and Geometry:
  Methods and Applications \textbf{10} (2014), 058.

\bibitem[FZ02]{FominZelevinsky02}
Sergey Fomin and Andrei Zelevinsky, \emph{Cluster algebras. {I}.
  {F}oundations}, J. Amer. Math. Soc. \textbf{15} (2002), no.~2, 497--529
  (electronic), \href {http://arxiv.org/abs/math/0104151v1}
  {\path{arXiv:math/0104151v1}}.

\bibitem[FZ07]{FominZelevinsky07}
\bysame, \emph{Cluster algebras {IV}: Coefficients}, Compositio Mathematica
  \textbf{143} (2007), 112--164, \href {http://arxiv.org/abs/math/0602259v3}
  {\path{arXiv:math/0602259v3}}.

\bibitem[FM01]{frenkel2001combinatorics}Edward Frenkel and Evgeny Mukhin,  \emph{Combinatorics of q-characters of finite-dimensional representations of quantum affine algebras}, Communications in Mathematical Physics \textbf{216} (2001), no.~1, 23-57.


\bibitem[FR99]{FrenkelReshetikhin99}
Edward Frenkel and Nicolai Reshetikhin, \emph{The q-characters of representations of quantum affine algebras and deformations of {$\mathcal{W}$}-algebras},
in Recent developments in quantum affine algebras and related topics (Raleigh, NC, 1998), Contemp.
Math., 248, Amer. Math. Soc., Providence, RI, (1999), 163–-205, \href{http://arxiv.org/abs/math/9810055}{\path{arXiv:math/9810055v5}}.

\bibitem[GHKK14]{gross2014canonical}
Mark Gross, Paul Hacking, Sean Keel, and Maxim Kontsevich, \emph{Canonical
  bases for cluster algebras}, 2014, \href {http://arxiv.org/abs/1411.1394}
  {\path{arXiv:1411.1394}}.

\bibitem[GLS07]{GeissLeclercSchroeer07b}
Christof Gei\ss, Bernard Leclerc, and Jan Schr{\"o}er, \emph{Cluster algebra
  structures and semicanonical bases for unipotent groups}, 2007, \href
  {http://arxiv.org/abs/math/0703039} {\path{arXiv:math/0703039}}.

\bibitem[GLS11]{GeissLeclercSchroeer10}
\bysame, \emph{Kac-{M}oody groups and cluster algebras}, Advances in
  Mathematics \textbf{228} (2011), no.~1, 329--433, \href
  {http://arxiv.org/abs/1001.3545v2} {\path{arXiv:1001.3545v2}}.


\bibitem[GLS12]{GLSbasis}
\bysame, \emph{Generic bases for cluster algebras and the Chamber Ansatz}, J. Amer. Math. Soc. \textbf{25} (2012), no.~1, 21--76, \href
{http://arxiv.org/abs/1004.2781} {\path{arXiv:1004.2781}}.

\bibitem[GLS13]{GeissLeclercSchroeer11}
\bysame, \emph{Cluster structures on quantum coordinate rings}, Selecta
  Mathematica \textbf{19} (2013), no.~2, 337--397, \href
  {http://arxiv.org/abs/1104.0531} {\path{arXiv:1104.0531}}.

\bibitem[Her04]{Hernandez02}
David Hernandez, \emph{Algebraic approach to q,t-characters}, Advances in
  Mathematics \textbf{187} (2004), no.~1, 1--52, \href
  {http://arxiv.org/abs/math/0212257} {\path{arXiv:math/0212257}}, \href
  {http://dx.doi.org/10.1016/j.aim.2003.07.016}
  {\path{doi:10.1016/j.aim.2003.07.016}}.

\bibitem[HL10]{HernandezLeclerc09}
David Hernandez and Bernard Leclerc, \emph{Cluster algebras and quantum affine
  algebras}, Duke Math. J. \textbf{154} (2010), no.~2, 265--341, \href
  {http://arxiv.org/abs/0903.1452} {\path{arXiv:0903.1452}}.



\bibitem[HL13a]{hernandez2013cluster}
\bysame, \emph{A cluster algebra approach to q-characters of
  {K}irillov-{R}eshetikhin modules}, 2013, \href {http://arxiv.org/abs/1303.0744}
  {\path{arXiv:1303.0744}}.

\bibitem[HL13b]{hernandez2013monoidal}
\bysame, \emph{Monoidal categorifications of cluster algebras of type {A} and {D}},
  Symmetries, Integrable Systems and Representations, Springer, 2013,
  pp.~175--193.
  
\bibitem[HL15]{HernandezLeclerc11}
\bysame, \emph{Quantum {G}rothendieck rings and derived {H}all algebras}, Journal f\"{u}r die reine und angewandte Mathematik (Crelles Journal), \textbf{2015} (2015), no.~701: 77--126, \href
  {http://arxiv.org/abs/1109.0862v1} {\path{arXiv:1109.0862v1}}.

\bibitem[IY08]{IyamaYoshino08}
Osamu Iyama and Yuji Yoshino, \emph{{Mutations in triangulated categories and
  rigid Cohen-Macaulay modules}}, Inv. Math. \textbf{172} (2008), no.~1, 117--168.

\bibitem[Kas90]{Kashiwara90}
Masaki Kashiwara, \emph{Bases cristallines}, C. R. Acad. Sci. Paris
    S\'er. I Math. \textbf{311} (1990), no. 6, 277-280.


\bibitem[KL79]{kazhdan1979representations}David Kazhdan and George Lusztig, \emph{Representations of {C}oxeter groups and {H}ecke algebras}, Invent. Math. \textbf{53} (1979), no. 2, 165–184.

\bibitem[Kel08]{Keller08c}
Bernhard Keller, \emph{Cluster algebras, quiver representations and
  triangulated categories}, \href {http://arxiv.org/abs/0807.1960v10}
  {\path{arXiv:0807.1960v10}}.

\bibitem[Kel12]{Keller12}
\bysame, \emph{Cluster algebras and derived categories}, 2012, \href
  {http://arxiv.org/abs/1202.4161} {\path{arXiv:1202.4161}}.

\bibitem[Kel13]{keller2013periodicity}
\bysame, \emph{The periodicity conjecture for pairs of Dynkin diagrams}, Ann.
  of Math.(2) \textbf{177} (2013), no.~1, 111--170.

\bibitem[Kim12]{Kimura10}
Yoshiyuki Kimura, \emph{Quantum unipotent subgroup and dual canonical basis},
  Kyoto J. Math. \textbf{52} (2012), no.~2, 277--331, \href
  {http://arxiv.org/abs/1010.4242} {\path{arXiv:1010.4242}}, \href
  {http://dx.doi.org/10.1215/21562261-1550976}
  {\path{doi:10.1215/21562261-1550976}}.

\bibitem[KKKO14]{KKKO14}
S.-J. {Kang}, M.~{Kashiwara}, M.~{Kim}, and S.-j. {Oh}, \emph{{Monoidal
  categorification of cluster algebras}}, ArXiv e-prints (2014), \href
  {http://arxiv.org/abs/1412.8106} {\path{arXiv:1412.8106}}.


\bibitem[KL09]{KhovanovLauda08}
Mikhail Khovanov and Aaron~D. Lauda, \emph{A diagrammatic approach to
  categorification of quantum groups {I}}, Represent. Theory \textbf{13}
  (2009), 309--347, \href {http://arxiv.org/abs/0803.4121v2}
  {\path{arXiv:0803.4121v2}}.

\bibitem[KQ14]{KimuraQin11}
Yoshiyuki Kimura and Fan Qin, \emph{Graded quiver varieties, quantum cluster
  algebras and dual canonical basis}, Adv. Math. \textbf{262} (2014): 261-312, \href
  {http://arxiv.org/abs/1205.2066} {\path{arXiv:1205.2066}}.



\bibitem[KR90]{KirillovReshetikhin90}
A.N. Kirillov and N.~Reshetikhin, \emph{Representations of yangians and
  multiplicities of the inclusion of the irreducible components of the tensor
  product of representations of simple lie algebras}, translation in J. Soviet.
  Math. \textbf{52} (1990), no.~3, 3156--3164.

\bibitem[KY11]{KellerYang09}
Bernhard Keller and Dong Yang, \emph{Derived equivalences from mutations of
  quivers with potential}, Adv. Math. \textbf{226} (2011), no.~3, 2118--2168,
  \href {http://arxiv.org/abs/0906.0761v3} {\path{arXiv:0906.0761v3}}.

\bibitem[Lec03]{Leclerc03}
B.~Leclerc, \emph{Imaginary vectors in the dual canonical basis of
  {$U_q(\mathfrak{n})$}}, Transform. Groups \textbf{8} (2003), no.~1, 95--104.

\bibitem[Lee13]{Lee13}
Kyungyong Lee, \emph{Every finite acyclic quiver is a full subquiver of a quiver mutation equivalent to a bipartite quiver}, 2013, \href{http://arxiv.org/abs/1311.0711} {\path{arXiv:1311.0711}}.
  
\bibitem[LLRZ14a]{lee2014existence}
Kyungyong Lee, Li~Li, Dylan Rupel, and Andrei Zelevinsky, \emph{The existence
  of greedy bases in rank 2 quantum cluster algebras}, 2014, \href
  {http://arxiv.org/abs/1405.2414} {\path{arXiv:1405.2414}}.

\bibitem[LLRZ14b]{lee2014greedyQuantum}
\bysame, \emph{Greedy bases in rank 2 quantum cluster algebras}, Proceedings of the National Academy of Sciences \textbf{111} (2014), no.~27: 9712--9716, \href
  {http://arxiv.org/abs/1405.2311} {\path{arXiv:1405.2311}}.

\bibitem[LLZ14]{lee2014greedy}
Kyungyong Lee, Li~Li, and Andrei Zelevinsky, \emph{Greedy elements in rank 2
  cluster algebras}, Selecta Mathematica \textbf{20} (2014), no.~1, 57--82.

\bibitem[LS13]{LeeSchiffler13}
Kyungyong Lee and Ralf Schiffler, \emph{Positivity for cluster algebras}, to
  appear in Annals of Mathematics, 2013, \href {http://arxiv.org/abs/1306.2415}
  {\path{arXiv:1306.2415}}.

\bibitem[Lus90]{Lusztig90}
George Lusztig, \emph{Canonical bases arising from quantized enveloping
  algebras}, J. Amer. Math. Soc. \textbf{3} (1990), no. 2, 447--498.

\bibitem[Lus93]{Lusztig93}
\bysame, \emph{Introduction to quantum groups}, Progr. Math. \textbf{110}, Birkh\"auser, 1993.

\bibitem[Lus94]{Lusztig94}
\bysame, \emph{Total positivity in reductive groups}, Lie theory
and geometry: in honor of Bertram Kostant, Progr. Math. \textbf{123}, Birkh\"auser, 1994.

\bibitem[Lus00]{Lusztig00}
\bysame, \emph{Semicanonical bases arising from enveloping algebras}, Adv.
  Math. \textbf{151} (2000), no.~2, 129--139.

\bibitem[Mul15]{Muller15}
Greg Muller, \emph{The existence of a maximal green sequence is not invariant under quiver mutation}, 2015, \href
{http://arxiv.org/abs/1503.04675} {\path{arXiv:1503.04675}}.

\bibitem[MFK94]{mumford1994geometric}
David Mumford, John Fogarty, and Frances~Clare Kirwan, \emph{Geometric
  invariant theory}, vol.~34, Springer Science \& Business, 1994.

\bibitem[MSW13]{musiker2013bases}
Gregg Musiker, Ralf Schiffler, and Lauren Williams, \emph{Bases for cluster
  algebras from surfaces}, Compositio Mathematica \textbf{149} (2013), no.~02,
  217--263.

\bibitem[Nag13]{Nagao10}
Kentaro Nagao, \emph{Donaldson--{T}homas theory and cluster algebras}, Duke
  Mathematical Journal \textbf{162} (2013), no.~7, 1313--1367, \href
  {http://arxiv.org/abs/1002.4884} {\path{arXiv:1002.4884}}.

\bibitem[Nak01]{Nakajima01}
Hiraku Nakajima, \emph{Quiver varieties and finite-dimensional representations
  of quantum affine algebras}, J. Amer. Math. Soc. \textbf{14} (2001), no.~1,
  145--238 (electronic), \href {http://arxiv.org/abs/math/9912158}
  {\path{arXiv:math/9912158}}.

\bibitem[Nak03]{Nakajima03}
\bysame, \emph{t-analogs of q-characters of {K}irillov-{R}eshetikhin modules of
  quantum affine algebras}, Represent. Theory \textbf{7} (2003), no.~2,
  259--274.

\bibitem[Nak04]{Nakajima04}
\bysame, \emph{Quiver varieties and {$t$}-analogs of {$q$}-characters of
  quantum affine algebras}, Ann. of Math. (2) \textbf{160} (2004), no.~3,
  1057--1097, \href {http://arxiv.org/abs/math/0105173v2}
  {\path{arXiv:math/0105173v2}}.

\bibitem[Nak11]{Nakajima09}
\bysame, \emph{Quiver varieties and cluster algebras}, Kyoto J. Math.
  \textbf{51} (2011), no.~1, 71--126, \href {http://arxiv.org/abs/0905.0002v5}
  {\path{arXiv:0905.0002v5}}.

\bibitem[Pal08]{Palu08a}
Yann Palu, \emph{Cluster characters for 2-{C}alabi-{Y}au triangulated
  categories}, Ann. Inst. Fourier (Grenoble) \textbf{58} (2008), no.~6,
  2221--2248, \href {http://arxiv.org/abs/math/0703540v2}
  {\path{arXiv:math/0703540v2}}.

\bibitem[Pla11a]{Plamondon10b}
Pierre-Guy Plamondon, \emph{Cluster algebras via cluster categories with
  in\-fi\-ni\-te-\-di\-men\-sional morphism spaces}, Compos. Math. \textbf{147}
  (2011), no.~6, 1921--1934, \href {http://arxiv.org/abs/1004.0830v1}
  {\path{arXiv:1004.0830v1}}.

\bibitem[Pla11b]{Plamondon10a}
\bysame, \emph{Cluster characters for cluster categories with
  in\-fi\-ni\-te-\-di\-men\-sional morphism spaces}, Adv. in Math. \textbf{227}
  (2011), no.~1, 1--39, \href {http://arxiv.org/abs/1002.4956v2}
  {\path{arXiv:1002.4956v2}}.

\bibitem[Pla13]{plamondon2013generic}
\bysame, \emph{Generic bases for cluster algebras from the cluster category},
  International Mathematics Research Notices \textbf{2013} (2013), no.~10,
  2368--2420.

\bibitem[Qin12]{Qin10}
Fan Qin, \emph{Quantum cluster variables via {S}erre polynomials}, J. Reine
  Angew. Math. \textbf{2012} (2012), no.~668, 149--190, with an appendix by
  Bernhard Keller, \href {http://dx.doi.org/10.1515/CRELLE.2011.129}
  {\path{doi:10.1515/CRELLE.2011.129}}.

\bibitem[Qin14]{Qin12}
\bysame, \emph{t-analog of q-characters, bases of quantum cluster algebras, and
  a correction technique}, International Mathematics Research Notices \textbf{2014} (2014), no.~22, 6175-6232,
  \href {http://arxiv.org/abs/1207.6604} {\path{arXiv:1207.6604}}.

\bibitem[Rou08]{Rouquier08}
Rapha\"el Rouquier, \emph{2-{K}ac-{M}oody algebras}, 2008, \href
  {http://arxiv.org/abs/0812.5023} {\path{arXiv:0812.5023}}.

\bibitem[Rou12]{rouquier2012quiver}
Rapha{\"e}l Rouquier, \emph{Quiver hecke algebras and 2-{L}ie algebras},
  Algebra colloquium, vol.~19, World Scientific, 2012, pp.~359--410.

\bibitem[Thu14]{thurston2013positive}
Dylan~P Thurston, \emph{A positive basis for surface skein algebras}, Proceedings of the National Academy of Sciences \textbf{111} (2014), no.~27: 9725-9732, 
  \href {http://arxiv.org/abs/1310.1959} {\path{arXiv:1310.1959}}.

\bibitem[Tra11]{Tran09}
Thao Tran, \emph{F-polynomials in quantum cluster algebras}, Algebr. Represent.
  Theory \textbf{14} (2011), no.~6, 1025--1061, \href
  {http://arxiv.org/abs/0904.3291v1} {\path{arXiv:0904.3291v1}}.

\bibitem[VV03]{VaragnoloVasserot03}
M.~Varagnolo and E.~Vasserot, \emph{Perverse sheaves and quantum {G}rothendieck
  rings}, Studies in memory of {I}ssai {S}chur ({C}hevaleret/{R}ehovot, 2000),
  Progr. Math., vol. 210, Birkh\"auser Boston, Boston, MA, 2003, pp.~345--365,
  \href {http://arxiv.org/abs/math/0103182v3} {\path{arXiv:math/0103182v3}}.
  \MR{MR1985732 (2004d:17023)}

\bibitem[VV11]{VaragnoloVasserot09}
\bysame, \emph{Canonical bases and {KLR}-algebras}, J. Reine Angew. Math.
  \textbf{2011} (2011), no.~659, 67--100, \href
  {http://arxiv.org/abs/0901.3992} {\path{arXiv:0901.3992}}, \href
  {http://dx.doi.org/10.1515/crelle.2011.068}
  {\path{doi:10.1515/crelle.2011.068}}.

\end{thebibliography}
\def\cprime{$'$}
\providecommand{\bysame}{\leavevmode\hbox to3em{\hrulefill}\thinspace}
\providecommand{\MR}{\relax\ifhmode\unskip\space\fi MR }
\providecommand{\MRhref}[2]{%
  \href{http://www.ams.org/mathscinet-getitem?mr=#1}{#2}
}
\providecommand{\href}[2]{#2}

\end{document}